\documentclass[12pt,reqno]{amsart}
\usepackage[margin=1in]{geometry}
\usepackage{amsmath,amssymb,amsthm,graphicx,amsxtra, setspace}
\usepackage[utf8]{inputenc}
\usepackage{mathrsfs}
\usepackage{hyperref}
\usepackage{upgreek}
\usepackage{mathtools}
\usepackage[mathcal]{euscript}
\usepackage{xcolor}
\allowdisplaybreaks

\usepackage[pagewise]{lineno}

\usepackage{graphicx,eurosym}
\usepackage[cyr]{aeguill}
\colorlet{darkblue}{blue!50!black}
\hypersetup{
	colorlinks,%
	citecolor=blue,%
	filecolor=red,%
	linkcolor=darkblue,%
	urlcolor=blue,%
	pdfnewwindow=true,%
	pdfstartview={FitH}
}

\usepackage{graphicx,amscd,mathrsfs,wrapfig,mathrsfs,lipsum}
\usepackage{eufrak}
\usepackage{tikz}
\usepackage{multicol}
\usepackage{caption}
\usetikzlibrary{arrows}

\colorlet{darkblue}{red!100!black}

\makeatletter

\newtheorem{theorem}{Theorem}[section]
\newtheorem{lemma}[theorem]{Lemma}

\newtheorem{example}[theorem]{Example}
\newtheorem{remark}[theorem]{Remark}

\let\originalleft\left
\let\originalright\right
\renewcommand{\left}{\mathopen{}\mathclose\bgroup\originalleft}
\renewcommand{\right}{\aftergroup\egroup\originalright}

\renewcommand{\d}{\/\mathrm{d}\/}

\def\w{\textbf{W}^{\varepsilon}_{{\theta}^{\varepsilon}}}

\def\T{T\wedge\tau_N}
\def\L{\mathbb{L}}

\def\C{\mathrm{C}}
\def\f{\boldsymbol{f}}

\def\D{\mathrm{D}}
\def\y{\boldsymbol{y}}

\def\X{\mathbb{X}}

\def\V{\mathbb{v}}
\def\T{\mathbb{T}^d}
\def\w{\boldsymbol{w}}

\def\G{\mathbb{G}}

\def\N{\mathbb{N}}

\def\no{\nonumber}
\def\V{\mathbb{V}}
\def\wi{\widetilde}

\def\u{\mathrm{U}}
\def\P{\mathrm{P}}
\def\u{\boldsymbol{v}}
\def\H{\mathbb{H}}
\def\n{\boldsymbol{n}}

\newcommand{\R}{\mathbb{R}}

\renewcommand{\d}{\/\mathrm{d}\/}

\newcommand{\Addresses}{{% additional braces for segregating \footnotesize
		\footnote{
			\noindent \textsuperscript{1,2}Department of Mathematics, Indian Institute of Technology Roorkee-IIT Roorkee,
			Haridwar Highway, Roorkee, Uttarakhand 247667, INDIA.\par\nopagebreak
			\noindent 	\textit{e-mail:} \texttt{Pardeep Kumar: pkumar3@ma.iitr.ac.in.}
			
			\noindent  \textit{e-mail:} \texttt{Manil T. Mohan: manilfma@iitr.ac.in, maniltmohan@gmail.com.}
			
			\noindent \textsuperscript{*}Corresponding author.

			\textit{Key words:} Convective Brinkman-Forchheimer equations, inverse problem, final overdetermination, Tikhonov's fixed point theorem.
			
			Mathematics Subject Classification (2020): Primary 35R30; Secondary 35Q35, 35Q30.

}}}

\begin{document}
	
	%	\linenumbers
	
	\title[An inverse problem for convective Brinkman-Forchheimer equations]{An inverse source problem for convective Brinkman-Forchheimer equations with the final overdetermination
		\Addresses}
	\author[P. Kumar and M. T. Mohan ]{Pardeep Kumar\textsuperscript{1} and Manil T. Mohan\textsuperscript{2*}}

	\maketitle
	
	\begin{abstract}
		In this paper, we examine an inverse problem for the following convective Brinkman-Forchheimer (CBF) equations or damped Navier-Stokes equations:
	\begin{align*}
		\boldsymbol{v}_t-\mu \Delta\boldsymbol{v}+(\boldsymbol{v}\cdot\nabla)\boldsymbol{v}+\alpha\boldsymbol{v}+\beta|\boldsymbol{v}|^{r-1}\boldsymbol{v}+\nabla p=\boldsymbol{F}:=\boldsymbol{f} g, \ \ \  \nabla\cdot\boldsymbol{v}=0,
	\end{align*}
on a torus  $\mathbb{T}^d$, $d=2,3$.  The inverse problem under consideration consists of determining the vector-valued velocity function $\boldsymbol {v}$, the pressure gradient $\nabla p$ and the vector-valued forcing function $\boldsymbol{f} $. Using the Tikhonov fixed point theorem, we prove the existence of a solution for  the inverse problem for 2D and 3D CBF equations with the final overdetermination data for the divergence free initial data in the energy space $ \mathbb{L}^2(\mathbb{T}^d)$. A concrete example is also provided to validate the obtained result. Moreover, we overcome the technical difficulties while proving  the uniqueness and H\"older type stability results by using the regularity results available for the direct problem for CBF equations.    The well-posedness results hold for $r \geq 1$ in two dimensions  and  for $r \geq 3$ in  three dimensions  for appropriate values of $\alpha,\mu$ and $\beta$. The nonlinear damping term $|\boldsymbol{v}|^{r-1}\boldsymbol{v}$ plays a crucial role in obtaining the required results.  In the case of supercritical growth ($r>3$), we obtain better results than that are available  in the literature for 2D Navier-Stokes equations. 
	\end{abstract}

%\tableofcontents

\section{Introduction}\label{sec1}\setcounter{equation}{0}
The convective Brinkman-Forchheimer (CBF) equations characterize the motion of incomp-ressible fluid flows in a saturated porous medium (cf. \cite{AB,XC}).
% \textcolor{blue}{The inverse problems for CBF equations have a wide range of applications in various fields such as fluid dynamics, heat transfer, medical imaging, and environmental science etc.} 
            The major objective of this work is to examine the well-posedness of an inverse problem to  CBF equations with periodic boundary conditions for the divergence free initial data  in the energy space  $\mathbb{L}^2(\mathbb{T}^d)$. 
\subsection{The mathematical model and the direct problem}
Let $L>0$ and $\T=\mathbb{R}^d/ (L\mathbb{Z})^d\cong (\mathbb{R}/ L\mathbb{Z})^d,$ $d=2,3$,   be the $d$-dimensional torus. We consider the following  CBF equations on the torus $\T$: 
\begin{align}
	\u_t-\mu \Delta\u+(\u\cdot\nabla)\u+\alpha\u+\beta|\u|^{r-1}\u+\nabla p=\boldsymbol{F}&:=\f g, \ \text{ in } \ \T \times[0,T), \label{1a}\\ \nabla\cdot\u&=0, \ \ \ \ \text{ in } \ \T\times[0,T),\label{1b}
\end{align}
with the initial condition
\begin{align}\label{1c}
	\u=\u_0, \ \text{ in } \ \T \times \{0\},
\end{align}
and	$\u(\cdot,\cdot)$ and $p(\cdot,\cdot)$ satisfy the following periodicity conditions:
\begin{align}\label{2a}
	\u(x+{L} e_i,t)=\u(x,t), \  \ \text{and} \ \ 	p(x+{L} e_i,t)=p(x,t),\ \text{ for all }\ (x,t)\in\R^d\times[0,T],
\end{align}
for $i=1,\ldots,d$, where $\{e_1,\ldots,e_d\}$ is the canonical basis of $\R^d$.    Here $\u(\cdot,\cdot):\T \times[0,T] \to  \R^d$ represents the velocity field, $p(\cdot,\cdot):\T \times[0,T] \to  \R$ denotes the pressure field and $\boldsymbol{F}(\cdot,\cdot):\T \times[0,T] \to  \R^d$ stands for an  external force which is periodic in the first variable, that is, $\boldsymbol{F}(x+{L} e_i,t)=\boldsymbol{F}(x,t), \text{ for all } (x,t)\in\R^d\times[0,T]$. The constant $\mu$ denotes the positive Brinkman coefficient (effective viscosity), while the positive constants $\alpha$ and $\beta$ stand for the Darcy coefficient (permeability of the porous medium) and the Forchheimer coefficient (proportional to the porosity of the material), respectively (cf. \cite{KH,PAM}). The parameter $r\in[1,\infty)$ is known as the absorption exponent  and the cases, $r=3$ and $r>3$, are referred as the critical exponent and the fast growing nonlinearity, respectively (see \cite{KT2}).  When $\alpha=\beta=0$, the classical $d$-dimensional Navier-Stokes equations (NSE) are obtained.  Thus, the system \eqref{1a}-\eqref{2a} can be viewed as a modification (by the introduction of an absorption term $\alpha\u+\beta|\u|^{r-1}\u$) of the classical NSE (or damped Navier-Stokes equations), and the damping term helps to obtain global solvability results even in three dimensions. By imposing the condition $\int_{\T}p(x,t)\d x=0, $ for $t\in [0,T]$, one can obtain the uniqueness of the pressure $p$. The model given in \eqref{1a}-\eqref{2a} is recognized to be more accurate when the flow velocity is too large for the Darcy's law to be valid alone, and apart from that, the porosity is not too small, thus, we call these types of models as \emph{non-Darcy models} (cf. \cite{PAM}). In Proposition 1.1, \cite{KWH}, it is demonstrated that the critical homogeneous CBF equations have the same scaling as NSE only when $\alpha=0$ and no scale invariance property for other values of $\alpha$ and $r$. 

The existence and uniqueness of weak solutions satisfying the energy equality and strong solutions for CBF equations in bounded and periodic domains is established in the works \cite{SNA,CLF,KWH,KWH1,KT2,KKMTM,MTM4}, etc., and references therein, and for the whole space, the results can be accessed from \cite{ZCQJ,ZZXW}, etc. The Navier-Stokes problem with a modified absorption term $|\u|^{r-1}\u$, for $r>1$, in bounded domains with compact boundary is considered in \cite{SNA}. The existence of Leray-Hopf weak solutions, for any dimension $d\geq 2$, and its uniqueness for $d=2$ is established in \cite{SNA}. In \cite{KWH}, the authors obtained a simple proof of the existence of global-in-time	smooth solutions of 3D CBF equations in periodic domains with the absorption exponent $r >3$. For the	critical value $r = 3$, the existence of a unique global, regular solution is proved, provided	that the coefficients satisfy a relation $4 \beta \mu \geq 1$. The authors in  \cite{KWH1} proved that  the strong solutions	of 3D CBF equations in periodic domains with the absorption exponent $r \in [1, 3]$ remain	strong under small changes of the initial condition and forcing function. Recently, for  $r \geq 3$ $(\beta, \mu >0  \ \text{for} \ r >3$ and  $2\beta\mu\geq 1$ for $r=3$), the long time behavior of  3D deterministic and stochastic CBF equations defined on a torus is discussed in \cite{KKMTM}.

\subsection{Investigation of the inverse problem}
	Despite the importance of the direct problem, it necessitates the knowledge of physical parameters such as the Brinkman coefficient $\mu$, Darcy coefficient $\alpha$, Forchheimer coefficient $\beta$ and the forcing term $\boldsymbol{F}:=\f g$. When, in addition to the solution of the equation, recovery of some physical properties of the investigated object or the effects of external sources are needed, it is better to use inverse problems to determine a coefficient or to handle the right hand side of the differential equation arising in a mathematical model of a physical phenomena. However, posing an inverse problem requires some additional information on the solution besides the given initial and boundary conditions. 
	
		 	In this work, as an additional information, we  use the trace of the velocity $\u$ and the pressure gradient $\nabla p$, as prescribed at the final moment $t=T$ of the segment $[0,T]$.  We assume that $\boldsymbol{F}$, the vector-valued external force in \eqref{1a}, can be written as  $$\boldsymbol{F}(x,t):=\f(x)g(x,t),$$ where $\f$ is an unknown  divergence-free vector-valued function  and $g$ is a given scalar function such that  $g$ and  $g_t$ are continuous on $\T \times[0,T]$. We consider the nonlinear inverse problem of determining the functions $\{\boldsymbol {v},\nabla p, \f\}$, satisfying the system \eqref{1a}-\eqref{2a}, with the final overdetermination condition:
		 	\begin{align}\label{1d}
		 		\u(x,T)=\boldsymbol{\varphi}(x), \ \ \ \ \nabla p(x,T)=\nabla\psi(x), \ \ \ x \in \T,
		 	\end{align} 
		 	where the functions $\u_0, \boldsymbol{\varphi}, \nabla \psi, \mu, \alpha, \beta$ and $g$ are given.  Note that $\boldsymbol{F}$ satisfies
		 	\begin{align}\label{g1}
		 		\nabla \cdot \boldsymbol{F}(x,t)=0 \ \ \text{for all} \ (x,t) \in \T \times [0,T),
		 	\end{align}
		 and this condition	is required for the uniqueness of determining the spatially varying 	divergence-free factor $\f$ of the source term $\boldsymbol{F}(x,t):=\f(x) g(x,t)$,  where $g$ is a given function. 
		 	
		 	 Let us now provide one example to show that the divergence free condition on $\boldsymbol{F}$ is necessary for obtaining uniqueness of the inverse problem. Suppose that the condition \eqref{g1} is not satisfied. Then, for
		\begin{align*}
			\mathcal{L}\u=	\u_t-\mu \Delta\u+(\u\cdot\nabla)\u+\alpha\u+\beta|\u|^{r-1}\u, \ \ \text{and} \ \ g \equiv 1,
		\end{align*}
		we assume that $(\u, \nabla p)$ satisfies
		\begin{align*}
			\mathcal{L} \u +\nabla p=\mathbf{0}, \ \ \ \nabla \cdot \u=0,
		\end{align*}
		and we fix $\u(x,0)$. Under the condition \eqref{2a} and $\int_{\T}p(x,t)\d x=0, $ for $t\in [0,T],$ it is trivial that $(\u, \nabla p)=(\mathbf{0},\mathbf{0})$ satisfies \eqref{1a}-\eqref{2a} with $\f\equiv\mathbf{0}$. Now, we choose $\psi \in \C^\infty_0(\T)$ satisfying $\nabla \psi \not\equiv \mathbf{0}$ in $\T$. Then $(\u, \nabla p)=(\mathbf{0},\nabla \psi)$ satisfies the same system \eqref{1a}-\eqref{2a} with $\f=\nabla \psi$. In other words, in the presence of  the pressure $ p$ in the CBF equations, there is no possibility of uniquely determining  the component of $\f$  given by a scalar potential. If the condition \eqref{g1} is satisfied, then $\psi\equiv 0$ and $\f$ can be determined uniquely. 

	Inverse problems with final overdetermination  conditions have been well studied in the literature (see \cite{JF,AH,AH1,VI,VI1,PKMTM,Pk,PV,PV1,POV,SY,VP}, etc.,  and references therein). An inverse problem
		with the final overdetermination for NSE has first been considered by A. I. Prilepko and I. A. Vasin in  \cite{PV} (cf. Sections 4.2 and 4.3, Chapter 4, \cite{POV}). In the work \cite{POV}, they proved the existence and uniqueness results of an inverse problem  for nonstationary linearized NSE with final and integral overdetermination conditions using Schauder's fixed point theorem.  By an application of Schauder's fixed point theorem, the solvability results of an inverse problem for the nonlinear nonstationary Navier-Stokes systems in the case of final overdetermination is established in \cite{PV1,VP} (cf. Chapter 4, \cite{POV}). However, neither uniqueness nor stability are taken into account for the same problem in the works \cite{PV1,VP}.  For an extensive study on numerous inverse problems corresponding to Navier-Stokes equations and related models, where one requires to determine the density of external forces or some coefficients of the equations on the basis of integral or functional overdetermination, we refer the interested readers to  \cite{MBFC,AYC2,AYC,AYC1,MC,JFGN,OYMY,VI,YJJF,AIK,PKKK,PKMTM1,RYL,Pk,PV,POV}, etc. and  references therein. 
	
The simultaneous determination of source terms in a linear parabolic problem from the observation of the final state is achieved through the use of a weak solution approach, as demonstrated in \cite{AH} (cf. \cite{AH1}).	The authors  in \cite{JF} examined the well-posedness of an inverse problem for 2D NSE  with the final overdetermination data using the Tikhonov fixed point theorem. To prove the same, they assumed that the initial data $\u_0 \in \H$ and the viscosity constant is sufficiently large. Recently, based on the existence of strong solutions of  CBF equations, the well-posedness of an inverse problem for 2D and 3D CBF equations with the final overdetermination data is established in \cite{PKMTM} using Schauder's fixed point theorem, where  the authors assumed that the initial data is sufficiently smooth $	(\u_0 \in \H^2(\Omega) \cap \V, \  \mbox{ where }\ \Omega \ \text{is a bounded domain})$.  The main difference of our work with the results obtained in \cite{PKMTM} is that we are proving the well-posedness of the   final overdetermination problem with $\u_0\in\H$ under much relaxed conditions than that obtained in \cite{JF}.  The nonlinear  damping term $|\u|^{r-1}\u$  helps us to control the convective term $(\u\cdot\nabla)\u$  and achieve the required results. An inverse problem of determining the initial condition for 2D and 3D CBF equations, given direct observations of the time dependent velocity field at a finite set of points at positive times (Eulerian observations) in periodic domains, is examined in \cite{MTM7}.

\subsection{Technical difficulties and approaches}
 We emphasize here that the method used in \cite{JF} (for the initial data $\u_0\in\H$) may only be applicable for the case of $d=2, \ r\in[1,3]$ (see \cite{PKMTM1}), due to a technical difficulty in working with bounded domains. Note that in the case of bounded domains,  $\mathrm{P}_{\H}(|\u|^{r-1}\u)$ ($\mathrm{P}_{\H}$ is the Helmholtz-Hodge orthogonal projection, see Subsection \ref{sub2.1}) need not be zero on the boundary, and $\mathrm{P}_{\H}$ and $-\Delta$ are not necessarily commuting (for a counter example, see Example 2.19, \cite{JCR4}). Furthermore, while taking the inner product with $-\Delta\u$ in \eqref{1a}, $-\Delta\u\cdot\n\neq 0$ on the boundary of the domain ($\n$ is the outward drawn normal to the boundary $\partial\Omega$), in general and the term with pressure may not disappear (see \cite{KT2}). As a result, the equality (\cite{KWH})
\begin{align}\label{1g}
	&\int_{\T}(-\Delta\u(x))\cdot|\u(x)|^{r-1}\u(x)\d x\no\\&
		=\int_{\T} |\nabla\u(x)|^{r-1}\u(x) \d x + 4\bigg( \frac{r-1}{(r+1)^2}\bigg) \int_{\T}| \nabla|\u(x)|^\frac{r+1}{2}|^r \d x 
	\nonumber\\&=\int_{\T}|\nabla\u(x)|^2|\u(x)|^{r-1}\d x+\frac{r-1}{4}\int_{\T}|\u(x)|^{r-3}|\nabla|\u(x)|^2|^2\d x,
\end{align}
may not be useful in the context of bounded domains. So, we restrict ourselves to periodic domains in this work, and the equality \eqref{1g} plays a crucial role in obtaining the well-posedness of solutions of the inverse problem \eqref{1a}-\eqref{1d}. Recently, the authors in  \cite{DsSz} addressed the above regularity problem for Dirichlet's boundary conditions and the well-posedness of such kinds of inverse problems for CBF equations  in bounded domains will be a future work. 
		\subsection{Main results and novelties of the work}
	By a solution of the inverse problem  \eqref{1a}-\eqref{1d}, we mean a set of vector-valued functions $\{\u, \nabla p,\f\}$ such that $$	\u\in\mathrm{L}^{\infty}(0,T;\H)\cap\mathrm{L}^{2}(0,T;\V) \cap \mathrm{L}^{r+1}(0,T;\widetilde{\L}^{r+1}),\ \nabla p(\cdot,t)\in \G(\T),\ \f\in\H,$$ for any $t\in[0,T]$ and the triplet  $\{\u,\nabla p,\f\}$  satisfies all the relations \eqref{1a}-\eqref{1d} in the weak sense. 
In order to prove the well-posedness of the inverse problem for the CBF equations in \cite{PKMTM},  the authors exploited the existence  of unique strong solutions for the direct problem for the CBF equations with smooth initial data (like $\u_0 \in \H^2(\Omega) \cap \V$). However, in practice, the data $\u_0, \boldsymbol{\varphi}, \nabla \psi$, and $g$ that are obtained from physical measurements may not be smooth functions. As a result, the methods based on the assumption of smooth initial data  are not applicable in such cases. In this paper, we employ a weak solution approach to formulate and solve the inverse problem for the triplet $\{\boldsymbol {v},\nabla p, \f\}$ by using Tikhonov's fixed point theorem.	We employ the methods developed in \cite{JF,PV1} to prove the existence of solutions and  \cite{JF} to establish the uniqueness and stability to the above formulated inverse problem. The aim of the present paper is to remove the growth restriction (see \cite{PKMTM1}) and verify the well-posedness of solutions of the inverse problem \eqref{1a}-\eqref{1d} with an arbitrary growth exponent for $r\geq1$, in 2D and for $r \geq3$, in 3D. Moreover, for the supercritical growth (that is, for $r>3$), the conditions on $\mu$ are much weaker than that obtained in \cite{JF}  for 2D NSE (cf. Remark \ref{rem4.2} below).  The major goals of this paper is to prove 
	\begin{enumerate}
		\item [(i)] the existence of a solution (using Tikhonov's fixed point theorem)  and its uniqueness, 
		\item [(ii)] the stability of the solution in the norm of the corresponding function spaces, 
	\end{enumerate}
	to the inverse problem \eqref{1a}-\eqref{1d} under the assumptions:
		\begin{align}\label{1e}
		|g(x,T)|\geq g_T >0 \  \text{ for some positive constant} \ g_T \ \text{for} \ x\in \T,
	\end{align}
and 
	\begin{align}\label{1f}
		\u_0 \in \H, \ \boldsymbol{\varphi} \in \H^2(\T) \cap \V , \ \ \ \nabla \psi \in \G(\T).
		\end{align}
In contrast to the results obtained for the CBF equations in \cite{PKMTM,PKMTM1}, and \cite{JF,POV}, etc., for NSE, the well-posedness of the generalized solution of the inverse problem holds for the initial data $\u_0 \in \H$ in the periodic domains.  Moreover, a H\"older type stability result is also obtained in this work. 
	
We now state the main result on the well-posedness of	solutions of the inverse problem \eqref{1a}-\eqref{1d}.
	\begin{theorem}\label{thm2}
		Let $\T  \ (d=2,3)$ be the $d$-dimensional torus, $\u_0 \in \H, \ \boldsymbol{\varphi} \in \H^2(\T) \cap \V , \  \nabla \psi \in \G(\T)$ and $g, g_t \in \C(\T \times[0,T])$ satisfy assumption \eqref{1e}. Moreover, let the conditions 	
		\begin{align}
			\frac{3}{4 \alpha^\frac{1}{3}}\| \boldsymbol{\varphi}\|_{\widetilde{\L}^4}^\frac{4}{3} \leq \mu, \ \alpha >0, \ \ \ \text{for} \ d=2 \ \text{and} \ r \in [1,3], \label{1g1}\\ 
			\frac{2(r-3)}{\alpha(r-1)}\left(\frac{8}{\beta\mu (r-1)}\right)^{\frac{2}{r-3}}< \mu, \  \alpha>0, \ \ \ \ \text{for} \ d=2,3 \ \text{and} \ r> 3, \label{1g2}\\
			\frac{1}{\beta}<\mu, \ \alpha>0, \  \ \ \ \ \ \text{for} \ d=3 \  \text{and} \ r=3,\label{1g3}
		\end{align}
hold,		where  $\alpha, \beta$ and $\mu$ be sufficiently large as discussed in Remark \ref{rem4.2} below. Then, the following assertions hold for the inverse problem \eqref{1a}-\eqref{1d}:
		\begin{enumerate}
			\item [$(i)$]  There exists a  solution $\{\u, \nabla p,\f\}$ to the inverse problem \eqref{1a}-\eqref{1d}.
			\item [$(ii)$]  Let $\{\u_i, \nabla p_i,\f_i\}$ $(i=1,2)$ be two solutions to the inverse problem \eqref{1a}-\eqref{1d} correspo-nding to the input data  $(\u_{0i},\boldsymbol{\varphi}_i,\nabla\psi_i,g_i) \ (i=1,2)$. Then, there exists a constant $C$ such that
			\begin{itemize}
				\item [$(1)$] for $d=2$ and $r \geq 1$
				\begin{align}\label{1k1}
				\nonumber&\|\u_1-\u_2\|_{\mathrm{L}^\infty(0,T;\H)}+	\|\u_1-\u_2\|_{\mathrm{L}^2(0,T;\V)}+	\|\u_1-\u_2\|_{\mathrm{L}^{r+1}(0,T;\widetilde{\L}^{r+1})}\no\\&\quad+\|\nabla(p_1-p_2)\|_{\mathrm{L}^\frac{r+1}{r}(0,T;\H^{-1})}+\|\f_1-\f_2\|_{\H} \nonumber\\&\leq C\big(\|\u_{01}-\u_{02}\|_{\H}^{\frac{2}{r+1}}+\|g_1-g_2\|_0^{\frac{2}{r+1}}+\|(g_1-g_2)_t\|_0^{\frac{2}{r+1}}\nonumber\\&\quad+\| \nabla (\boldsymbol{\varphi}_1-\boldsymbol{\varphi}_2)\|_{\H}^{\frac{2}{r+1}}+\| \nabla (\psi_1-\psi_2)-\mu \Delta(\boldsymbol{\varphi}_1-\boldsymbol{\varphi}_2)\|_{\L^2}^{\frac{2}{r+1}}\big),
				\end{align}
			\item [$(2)$] for $d=3$ and $r \geq 3$
			\begin{align}\label{1k}
				\nonumber&\|\u_1-\u_2\|_{\mathrm{L}^\infty(0,T;\H)}+	\|\u_1-\u_2\|_{\mathrm{L}^2(0,T;\V)}+	\|\u_1-\u_2\|_{\mathrm{L}^{r+1}(0,T;\widetilde{\L}^{r+1})}\no\\&\quad+\|\nabla(p_1-p_2)\|_{\mathrm{L}^\frac{r+1}{r}\big(0,T;\H^{-1}+{\L}^\frac{r+1}{r}\big)}+\|\f_1-\f_2\|_{\H} \nonumber\\&\leq C\left(\|\u_{01}-\u_{02}\|_{\H}^{\frac{2}{r+1}}+\|g_1-g_2\|_0^{\frac{2}{r+1}}+\|(g_1-g_2)_t\|_0^{\frac{2}{r+1}}\right.\nonumber\\&\quad\left.+\| \nabla (\boldsymbol{\varphi}_1-\boldsymbol{\varphi}_2)\|_{\H}^{\frac{2}{r+1}}+\| \nabla (\psi_1-\psi_2)-\mu \Delta(\boldsymbol{\varphi}_1-\boldsymbol{\varphi}_2)\|_{\L^2}^{\frac{2}{r+1}}\right), 
			\end{align}
			where $C$ depends on the input data, $\mu,\alpha,\beta,r$ and $T$. 
				\end{itemize}
		\end{enumerate}	
	\end{theorem}
Theorem \ref{thm2} immediately implies the uniqueness in determining the solutions $\{\u, \nabla p,\f\}$ to the inverse problem \eqref{1a}-\eqref{1d}, that is,  the uniqueness of solutions can be derived from \eqref{1k1} and  \eqref{1k}.

	%\noindent
The rest of the paper is structured as follows: In the next section, we first discuss the function spaces and some important inequalities. After defining the function spaces, we provide the relation between the solvability of the inverse problem \eqref{1a}-\eqref{1d} and the equivalent nonlinear operator equation (Theorem \ref{thm1}).  We derive a number of \emph{a-priori} estimates and some  regularity results for the solutions of the CBF equations \eqref{1a}-\eqref{2a} required to investigate the inverse problem \eqref{1a}-\eqref{1d} in Section \ref{sec3}.  We prove the first part of our main  result (Theorem \ref{thm2} $(i)$)  in  Section \ref{sec4}, by first proving the existence of a solution to the equivalent operator equation by using the Tikhonov fixed point theorem. We have provided a concrete example also to validate our claims in the same section (Example \ref{ex4.5}). The second part of Theorem \ref{thm2} $(ii)$ is proved in Section \ref{sec4b} by establishing the uniqueness and stability of the solution to the inverse problem \eqref{1a}-\eqref{1d}. We provide the proof of Theorem \ref{thm1} in Appendix \ref{sec5}. Moreover, we deduce some useful energy estimates that are required to investigate the inverse problem \eqref{1a}-\eqref{1d}  in  Appendix  \ref{sec6}. 
\section{Mathematical Formulation}\label{sec2}\setcounter{equation}{0}
This section begins by introducing the function spaces and standard notations that will be used throughout the paper. We consider the problem \eqref{1a}-\eqref{1d} on a $d$-dimensional torus $\T=\mathbb{R}^d/ (L\mathbb{Z})^d, \ L>0$. Then, we provide a mapping which transforms the original inverse problem \eqref{1a}-\eqref{1d}  into an equivalent nonlinear operator equation of second kind \eqref{1i} and verify their equivalence (Theorem \ref{thm1}).

\subsection{Function spaces}\label{sub2.1} Let ${\C}_p^{\infty}(\T;\R^d)$ denote the space of all infinitely differentiable functions  ($\R^d$-valued) such that $\u(x+{L} e_i)=\u(x), $ for all $ x \in \R^d$, and $i=1,\ldots,d$, where $\{e_1,\ldots,e_d\}$ is the canonical basis of $\R^d$. As mentioned in \cite{KWH}, the absorption term $\beta|\u|^{r-1}\u$ does not preserve zero mean-value constraints on the functions, that is, $\int_{\T} \u(x),\d x = 0$, as is the case with the NSE. Therefore, we cannot apply the well-known Poincar\'e  inequality $\|\u\|_{\L^2} \leq C  \|\nabla\u\|_{\L^2}$, and we must consider the full $\H^1$-norm.

The Sobolev space ${\H}_p^s(\T):={\mathrm{H}}_p^s(\T;\mathbb{R}^d)$ is the completion of ${\C}_p^{\infty}(\T;\R^d)$ with respect to the $\H^s$-norm $$\|\u\|_{{\H}_p^s}:=\bigg(\sum_{0 \leq|\alpha| \leq s}\|\D^{\alpha} \u\|_{\L^2(\T)}^2 \bigg)^{1/2}.$$ 
The Sobolev space of periodic functions $\H_{\mathrm{p}}^s(\mathbb{T}^d)$ is the same as (\cite{JCR}) $$\left\{\y:\y(x)=\sum_{k\in\mathbb{Z}^d}\y_{k} \mathrm{e}^{2\pi i k\cdot x},\ \overline{\y}_{k}=\y_{-k},\ \|\y\|_{\H^s_f}:=\left(\sum_{k\in\mathbb{Z}^d}(1+|k|^{2s})|\y_{k}|^2\right)^{\frac{1}{2}}<\infty\right\}.$$ From \cite[Proposition 5.38]{JCR}, we infer that the norms $\|\cdot\|_{{\H}^s_p}$ and $\|\cdot\|_{{\H}^s_f}$ are equivalent.
	Let us define 
	\begin{align*} 
		\mathcal{V}&:=\{\u\in {\C}_p^{\infty}(\T;\R^d):\nabla\cdot\u=0\},\\
		\mathbb{H}&:=\text{the closure of }\ \mathcal{V} \ \text{ in the Lebesgue space } \L^2(\T)=\mathrm{L}^2(\T;\R^d),\\
		\mathbb{V}&:=\text{the closure of }\ \mathcal{V} \ \text{ in the Sobolev space } \ \H^1(\T)=\mathrm{H}^1(\T;\R^d),\\
		\widetilde{\L}^{p}&:=\text{the closure of }\ \mathcal{V} \ \text{ in the Lebesgue space } \L^p(\T)=\mathrm{L}^p(\T;\R^d),
	\end{align*}
	for $p\in(2,\infty]$. Then, we characterize the spaces $\H$, $\V$, $\widetilde{\L}^p$ and $\widetilde{\L}^\infty$ with the norms
\begin{align*}
\|\u\|_{\H}^2:=\int_{\T}|\u(x)|^2\d x, \quad
	\|\u\|_{\V}^2:=\int_{\T}\left(|\u(x)|^2+|\nabla\u(x)|^2\right)\d x, \quad
	\|\u\|_{\widetilde{\L}^p}^p:=\int_{\T}|\u(x)|^p\d x, \ \ 
\end{align*}
for $p \in(2,\infty),$ and  $$\|\u\|_{\widetilde{\L}^\infty}:=\operatorname*{ess\,sup}\limits_{x \in \T}|\u(x)|,$$ respectively. 

 For the Fourier expansion $\u(x)=\sum\limits_{k\in\mathbb{Z}^d} e^{2\pi i k\cdot x} \u_{k} ,$ we have by using Parseval's identity
\begin{align*}
	\|\u\|_{\H}^2=\sum\limits_{k\in\mathbb{Z}^d} |\u_{k}|^2 \  \text{and} \ \|\Delta \u\|_{\H}^2=(2\pi)^4\sum_{k\in\mathbb{Z}^d}|k|^{4}|\u_{k}|^2.
\end{align*}
Therefore, we have
\begin{align*}
	\|\u\|_{\H^2_\mathrm{p}(\mathbb{T}^d)}^2=\sum_{k\in\mathbb{Z}^d}(1+|k|^{4})|\u_{k}|^2=\|\u\|_{\H}^2+\frac{1}{(2\pi)^4}\|\Delta\u\|_{\H}^2\leq\|\u\|_{\H}^2+\|\Delta\u\|_{\H}^2.
\end{align*}
Moreover, by the definition of $\|\cdot\|_{\H^2_\mathrm{p}(\mathbb{T}^d)}$, we have $ \|\u\|_{\H^2_\mathrm{p}(\mathbb{T}^d)}^2\geq\|\u\|_{\H}^2+\|\Delta \u\|_{\H}^2$ and hence it is immediate that $\|\u\|_{\H}^2+\|\Delta \u\|_{\H}^2\simeq\|\u\|_{\H^2_\mathrm{p}(\mathbb{T}^d)}^2$.

In the case of  $\T$ (see \cite{JCR4}),  every vector  $\u\in {\L}^2(\T)$ can be uniquely decomposed into $$\u=\w+\nabla q,$$ where $\w\in \H$ and $\nabla q\in\mathbb{H}^{1}$ (Helmholtz-Weyl or Helmholtz-Hodge decomposition). We can express this decomposition as $${\L}^2(\T)=\H \oplus \G(\T),$$ where $\G(\T)$ is the orthogonal complement of $\H$ in ${\L}^2(\T)$, which consists of gradients of scalar functions belonging $\mathbb{H}^1$. 

Let $(\cdot,\cdot)$ denote the inner product in the Hilbert space $\H$ and $\langle\cdot,\cdot \rangle$ represent the induced duality between the spaces $\V$ and its dual $\V'$ as well as $\widetilde{\L}^p$ and its dual $\widetilde{\L}^{p'}$, where $\frac{1}{p}+\frac{1}{p'}=1$. Note that $\H$ can be identified with its dual $\H'$.
	Wherever needed, we assume that $p_0\in\mathrm{H}^1(\T)\cap\mathrm{L}^2_0(\T),$ where $\mathrm{L}^2_0(\T):=\left\{p\in\mathrm{L}^2(\T):\int_{\T}p(x)\d x=0\right\}$. 
	The norm in the space $\C(\T\times [0,T])$ is denoted by $\|\cdot\|_0$, that is,  $\|g\|_0:=\sup\limits_{(x,t)\in\T\times [0,T]}|g(x,t)|$.

\subsubsection{Important inequalities} In the sequel, $C$ denotes a generic constant which may take different values at different places. 	The following Gagliardo-Nirenberg's and Agmon's inequa-lities are used repeatedly in the paper: 
	\begin{align*}
		\|\u\|_{\L^p}&\leq C\|\nabla\u\|_{\L^2}^{d\left(\frac{1}{2}-\frac{1}{p}\right)}\|\u\|_{\L^2}^{1-d\left(\frac{1}{2}-\frac{1}{p}\right)}, \ \text{ for all }\ \u\in\H^1(\T),\\
		\|\nabla\u\|_{\L^p}&\leq C\|\u\|_{\H^2}^{\frac{1}{2}+\frac{d}{2}\left(\frac{1}{2}-\frac{1}{p}\right)}\|\u\|_{\L^2}^{\frac{1}{2}-\frac{d}{2}\left(\frac{1}{2}-\frac{1}{p}\right)}, \ \text{ for all }\ \u\in\H^2(\T),
	\end{align*}
	where  $2\leq p<\infty$ for $d=2$ and $2\leq p\leq 6$ for $d=3$, and 
\begin{align*}
		\|\u\|_{\L^{\infty}}\leq C\|\u\|_{\L^2}^{1/2}\|\u\|_{\H_p^2}^{1/2} &\leq C\|\u\|_{\L^2}^{1/2}\left(\|\u\|_{\L^2}^{1/2}+\|\Delta\u\|_{\L^2}^{1/2}\right)\\&\leq C \left(\|\u\|_{\L^2}+\|\u\|_{\L^2}^\frac{1}{2}\|\Delta\u\|_{\L^2}^\frac{1}{2}\right), \ \ \ \text{ for }d=2, \\
			\|\u\|_{\L^{\infty}}\leq C\|\u\|_{\L^2}^{1/4}\|\u\|_{\H_p^2}^{3/4} &\leq C\|\u\|_{\L^2}^{1/4}\left(\|\u\|_{\L^2}^{3/4}+\|\Delta\u\|_{\L^2}^{3/4}\right)\\&\leq C \left(\|\u\|_{\L^2}+\|\u\|_{\L^2}^\frac{1}{4}\|\Delta\u\|_{\L^2}^\frac{3}{4}\right), \ \ \ \text{ for }d=3,
\end{align*}
	for all $\u\in\H_p^2(\T)$. In 2D, the well-known Ladyzhenskaya's inequality, that is, $\|\u\|_{\L^4}^2\leq \sqrt{2}\|\u\|_{\L^2}\|\nabla\u\|_{\L^2}$, for all $\u\in\H^1(\mathbb{T}^2)$ and the inequality  $\|\nabla\u\|_{\L^6} \leq C\|\Delta\u\|_{\L^2},$ for all $\u\in\mathbb{T}^d$ will also be used (see Lemma 2.6, \cite{KH}).

	\subsubsection{Nonlinear operator}
	Let us now consider the operator $\mathcal{C}:\widetilde{\L}^{r+1} \to \widetilde{\L}^\frac{r+1}{r}$ defined by $\mathcal{C}(\u):=|\u|^{r-1}\u$. It can be easily seen that $\langle\mathcal{C}(\u),\u\rangle =\|\u\|_{\wi\L^{r+1}}^{r+1}$. Furthermore, for all $\u \in \widetilde{\L}^{r+1}$, the map is Gateaux differentiable with the Gateaux derivative
	\begin{align}\label{29}
		\mathcal{C}'(\u)\w&=\left\{\begin{array}{cl}\w,&\text{ for }r=1,\\ \left\{\begin{array}{cc}|\u|^{r-1}\w+(r-1)\frac{\u}{|\u|^{3-r}}(\u\cdot\w),&\text{ if }\u\neq \mathbf{0},\\\mathbf{0},&\text{ if }\u=\mathbf{0},\end{array}\right.&\text{ for } 1<r<3,\\ |\u|^{r-1}\w+(r-1)\u|\u|^{r-3}(\u\cdot\w), &\text{ for }r\geq 3,\end{array}\right.
	\end{align}
	for all $\w \in \widetilde{\L}^{r+1}$. For $\u,\w \in \widetilde{\L}^{r+1}$, it can be easily verified that
	\begin{align}\label{30}
		\langle\mathcal{C}'(\u)\w,\w\rangle=\int_{\T} |\u(x)|^{r-1}|\w(x)|^2 \d x+(r-1)\int_{\T}|\u(x)|^{r-3}|\u(x)\cdot\w(x)|^2 \d x \geq 0,
	\end{align}
	for $r \geq1$ (note that for the case $1<r<3$, \eqref{30} holds, since in that case the second integral becomes $\int_{\{x \in \T:\u(x) \neq0\}}\frac{1}{|\u(x)|^{3-r}}|\u(x)\cdot\w(x)|^2 \d x \geq 0$). For $r \geq 3$, we have
	\begin{align*}
		\mathcal{C}''(\u)(\w \otimes \boldsymbol{\vartheta})&=(r-1)\big\{|\u|^{r-3}[(\u\cdot \boldsymbol{\vartheta})\w+(\u\cdot\w)\boldsymbol{\vartheta}+( \boldsymbol{\vartheta} \cdot \w)\u]\big\} \\ &\quad+(r-1)(r-3)\big[|\u|^{r-5}(\u\cdot\w)(\u\cdot \boldsymbol{\vartheta})\u\big],
	\end{align*}
	for all $\u \neq \mathbf{0},\w, \boldsymbol{\vartheta} \in \widetilde{\L}^{r+1}$ and is zero for $\u=\mathbf{0}$.

	\subsection{Equivalent formulation}
For the inverse problem \eqref{1a}-\eqref{1d}, let us provide an equivalent formulation as a nonlinear operator equation. Consider $\mathcal{D}$ to be a subset of $\H$ defined by
\begin{align*}
	\mathcal{D}:=\left\{\f \in \H \ : \ \|\f\|_{\H}\leq M\right\},
\end{align*}
where $M$ is a positive constant, which  will be specified  later.	We define a nonlinear operator $\mathcal{A}:\mathcal{D} \to \H$ by
	\begin{align}\label{1h}
		(\mathcal{A}\f)(x):=\u_t(x,T), \ \text{ for }\ x \in \T,
	\end{align}
	where $\u(x,T)$ is the function entering the set $(\u,\nabla p)$ and solving the direct problem \eqref{1a}-\eqref{2a}. Then, for $\f$, we analyze the following nonlinear operator equation of the second kind:
	\begin{align}\label{1i}
		\f=\mathcal{B}\f:=\frac{1}{g(x,T)}\left(\mathcal{A}\f+(\boldsymbol{\varphi} \cdot \nabla)\boldsymbol{\varphi}+\nabla \psi-\mu \Delta \boldsymbol{\varphi}+\alpha \boldsymbol{\varphi}+\beta|\boldsymbol{\varphi}|^{r-1}\boldsymbol{\varphi}\right),
	\end{align}
	over the space $\mathcal{D}$.
	
	The connection between the solvability of the nonlinear operator equation of the second kind \eqref{1i} and the inverse problem \eqref{1a}-\eqref{1d}  is shown in the following theorem:
	\begin{theorem}\label{thm1}
		Let $\T \subset \mathbb{R}^d \ (d=2,3)$ be a torus, $\u_0 \in \H, \ \boldsymbol{\varphi} \in \H^2(\T) \cap \V , \  \nabla \psi \in \G(\T)$ and $g, g_t \in \C(\T \times[0,T])$ satisfy the assumption \eqref{1e}, and let the  conditions \eqref{1g1}-\eqref{1g3}, and  $\alpha, \beta$ and $\mu$ are sufficiently large  (see Remark \ref{rem4.2}), be satisfied. 	 Then, the operator equation \eqref{1i} has a solution lying within $\mathcal{D}$ if and only if	the inverse problem \eqref{1a}-\eqref{1d}  has a solution.
	\end{theorem}
	\begin{proof}
		See Appendix \ref{sec5}.
	\end{proof}
		\begin{remark}
			For $d=2,3$ and $r\in (3,\infty)$,  the diffusion term $-\mu\Delta\u$ and the nonlinear  damping term $\beta|\u|^{r-1}\u$ dominate the convective term $(\u \cdot\nabla)\u$, and it helps us to obtain the solvability of the operator equation \eqref{1i} without any restriction on the data (cf. \eqref{1g2}). But in the case of $d=2$ and $r\in[1,3]$, such a domination seems to be not possible and  we have to enforce  a restriction on the data (see \eqref{1g1}).
			\end{remark}

	\section{Some Useful Energy Estimates}\label{sec3}\setcounter{equation}{0} 
Here  we derive a number of useful \emph{a-priori} energy estimates for the solutions of the CBF equations \eqref{1a}-\eqref{2a} required to investigate the inverse problem \eqref{1a}-\eqref{1d} (see Appendix \ref{sec6} for a discussion on well-posedness and energy estimates).
\begin{lemma}\label{lemmae1}
	Let $(\u(\cdot),\nabla p(\cdot))$ be the unique solution of the CBF  equations \eqref{1a}-\eqref{2a} and $\u_0 \in \H$. Then, for $r \geq 1$, the following estimate holds:
\begin{align}\label{n-1}
	\sup_{t\in [0,T]}	\|\u(t)\|_{\H}\leq \|\u_0\|_{\H}+\frac{1}{ \alpha}\|g\|_0\|\f\|_{\H}, 
\end{align}
and 
\begin{align}\label{n-2}
&	\mu\int_0^t\|\nabla\u(s)\|_{\H}^2\d s+\beta\int_0^t\|\u(s)\|_{\wi\L^{r+1}}^{r+1}\d s\nonumber\\&\leq \frac{1}{2}\|\u_0\|_{\H}^2+t\|g\|_0\|\f\|_{\H}\left(\|\u_0\|_{\H}+\frac{1}{ \alpha}\|g\|_0\|\f\|_{\H}\right).
\end{align}
Moreover, there exists a time $t_1 \in \big(\frac{T}{8},\frac{2T}{8}\big)$ such that
	\begin{align}
		\|\nabla \u(t_1)\|_{\H}^2 	&\leq \frac{4}{\mu } \left\{ \left(\frac{1}{T}+\frac{\alpha}{8}\right)\|\u_0\|_{\H}^2+\frac{1}{ \alpha}\|g\|_0^2\|\f\|_{\H}^2\right\},\label{n-3}\\
		\| \u(t_1)\|_{\widetilde{\L}^{r+1}}^{r+1} 	&\leq  \frac{4}{\beta }\left\{ \left(\frac{1}{T}+\frac{\alpha}{8}\right)\|\u_0\|_{\H}^2+\frac{1}{ \alpha}\|g\|_0^2\|\f\|_{\H}^2\right\},\label{n-4}
	\end{align}
	holds.
\end{lemma} 
\begin{proof}
	Let us take the inner product with $\u(\cdot)$ to the equation \eqref{1a} and use the fact that $\langle(\u \cdot \nabla)\u,\u\rangle=0$ and $\langle\nabla p,\u\rangle=0$ to obtain 
	\begin{align}\label{n6}
	&	\frac{1}{2}\frac{\d}{\d t}\|\u(t)\|_{\H}^2+\mu \|\nabla\u(t)\|_{\H}^2+\alpha \|\u(t)\|_{\H}^2+\beta\|\u(t)\|_{\widetilde{\L}^{r+1}}^{r+1}\no\\&=(\f g(t),\u(t))\leq \|g\|_{\mathrm{L}^{\infty}}\|\f\|_{\H}\|\u(t)\|_{\H},
	\end{align}
	for a.e. $t\in[0,T]$, where we have performed the integration by parts over $\T$ and used H\"older's inequality. Thus, it is immediate that 
	\begin{align*}
		\frac{\d}{\d t}\|\u(t)\|_{\H}+\alpha\|\u(t)\|_{\H}\leq \|g\|_{\mathrm{L}^{\infty}}\|\f\|_{\H}.
	\end{align*}
The variation of constants formula yields for all $t\in[0,T]$
\begin{align*}
	\|\u(t)\|_{\H}\leq e^{-\alpha t}\|\u_0\|_{\H}+\frac{1-e^{-\alpha t}}{ \alpha}\|g\|_0\|\f\|_{\H}, 
\end{align*}
and the estimate \eqref{n-1} follows. Integrating \eqref{n6}  from $0$ to $t$ and then using \eqref{n-1}, one can deduce \eqref{n-2}.

Let us now prove \eqref{n-3} and \eqref{n-4}. From Remark \ref{remb7}, we have 
	$$\u \in  \C^1\big([T/8,T];\V\big).$$  By the mean value theorem, there exists a time $t_1 \in \big(\frac{T}{8},\frac{2T}{8}\big)$ such that
	\begin{align*}
		\|\nabla\u(t_1)\|_{\H}^2&=\frac{1}{\frac{2T}{8}-\frac{T}{8}}\int_{T/8}^{2T/8}	\|\nabla\u(t)\|_{\H}^2 \d t= \frac{8}{T} \int_{0}^{2T/8} \|\nabla\u(t)\|_{\H}^2 \d t,
	\end{align*}
	which	leads to \eqref{n-3} by using \eqref{n-2} and Young's inequality. Using the similar arguments as above, we obtain 
	\begin{align*}
		\| \u(t_1)\|_{\widetilde{\L}^{r+1}}^{r+1} =\frac{8}{T}\int_{T/8}^{2T/8} \| \u(t)\|_{\widetilde{\L}^{r+1}}^{r+1}	 \d t\leq\frac{8}{T}\int_{0}^{2T/8} \| \u(t)\|_{\widetilde{\L}^{r+1}}^{r+1}\d t,
	\end{align*}
	and \eqref{n-4} follows from \eqref{n-2} and Young's inequality.
\end{proof}

\begin{lemma}\label{lemmabe2}
	Let $(\u(\cdot),\nabla p(\cdot))$ be the unique solution of the CBF  equations \eqref{1a}-\eqref{2a} and $\u_0 \in \H$. Let $t_1 \in \big(\frac{T}{8},\frac{2T}{8}\big),$ be the same time obtained in Lemma \ref{lemmae1}. Then, for all $t \in [t_1,T]$,
	\begin{itemize}
		\item [$(i)$] 	for $d=2$ and $r \in[1,3],$  we have
		\begin{align}\label{n111}
			\sup_{t\in[t_1,T]}\|\nabla\u(t)\|_{\H}^2 &\leq   \frac{4}{\mu } \left[ \left(\frac{1}{T}+\frac{\alpha}{8}\right)\|\u_0\|_{\H}^2+\frac{1}{\alpha}\|g\|_0^2\|\f\|_{\H}^2\right] +\frac{1}{2\mu\alpha}\|g\|_0^2\|\f\|_{\H}^2, 
		\end{align} 
	and
		\begin{align}\label{n12}
			\int_{t_{1}}^t\|\Delta \u(s)\|_{\H}^2\d s&\leq   \frac{4}{\mu^2 } \left[ \left(\frac{1}{T}+\frac{\alpha}{8}\right)\|\u_0\|_{\H}^2+\frac{1}{ \alpha}\|g\|_0^2\|\f\|_{\H}^2\right]+\frac{t-t_1}{\mu^2}\|g\|_0^2\|\f\|_{\H}^2,
		\end{align} 
		\item [$(ii)$]
		for $d=2,3$ and $r >3$ with $\eta=\frac{2(r-3)}{\mu(r-1)}\left(\frac{4}{\beta \mu  (r-1)}\right)^\frac{2}{r-3}<2\alpha$,  we have 
		\begin{align}\label{n13}
			\nonumber&\sup_{t\in[t_1,t]}\|\nabla \u(t)\|_{\H}^2+\big(2\alpha- \eta\big) \int_{t_1}^t\|\nabla\u(t)\|_{\H}^2\d t+\beta  \int_{t_1}^t \big\||\u(t)|^\frac{r-1}{2}|\nabla\u(t)|\big\|_{\H}^2 \d t \no \\&\leq  \frac{4}{\mu } \left[ \left(\frac{1}{T}+\frac{\alpha}{8}\right)\|\u_0\|_{\H}^2+\frac{1}{\alpha}\|g\|_0^2\|\f\|_{\H}^2\right]+\frac{2(t-t_1)}{\mu}\|g\|_0^2\|\f\|_{\H}^2,
		\end{align}
		\item [$(iii)$] for $d=r=3$ with $ \beta \mu > 1$, we have 
		\begin{align}\label{n14}
			\nonumber&\sup_{t\in[t_1,t]}\|\nabla \u(t)\|_{\H}^2+\mu\int_{t_{1}}^t\|\Delta \u(t)\|_{\H}^2\d t +2\bigg(\beta -\frac{1}{\mu}\bigg)\int_{t_{1}}^t \big\||\u(t)| |\nabla\u(t)|\big\|_{\H}^2 \d t
			\no \\&\leq \frac{4}{\mu } \left[ \left(\frac{1}{T}+\frac{\alpha}{8}\right)\|\u_0\|_{\H}^2+\frac{1}{\alpha}\|g\|_0^2\|\f\|_{\H}^2\right]+\frac{2(t-t_1)}{\mu}\|g\|_0^2\|\f\|_{\H}^2.
		\end{align}
	\end{itemize}
\end{lemma}
\begin{proof} 
	Taking the inner product with $- \Delta \u(\cdot)$ in \eqref{1a} and integrating the resulting equation over $\T$, we obtain 
	\begin{align}\label{n15}
		\nonumber&\frac{1}{2} \frac{\d}{\d t}\|\nabla\u(t)\|_{\H}^2+\mu\|\Delta\u(t)\|_{\H}^2+\alpha \|\nabla\u(t)\|_{\H}^2+\beta \big\| |\u(t)|^{\frac{r-1}{2}} |\nabla \u(t)| \big\|_{\H}^2\nonumber\\&=(\f g(t),-\Delta \u(t))-\langle(\u(t) \cdot\nabla)\u(t),-\Delta\u(t)\rangle, 
	\end{align}
	for a.e. $t \in [t_1,T]$.	We observe that
	\begin{align*}
		&\int_{\T}(-\Delta \u(x)) \cdot |\u(x)|^{r-1}\u(x) \d x \\
		&=\int_{\T} |\nabla\u(x)|^{r-1}\u(x) \d x + 4\bigg( \frac{r-1}{(r+1)^2}\bigg) \int_{\T}| \nabla|\u(x)|^\frac{r+1}{2}|^r \d x \\
		&=\int_{\T} |\nabla\u(x)|^{r-1}\u(x) \d x +  \frac{r-1}{4} \int_{\T}|\u(x)|^{r-3}| \nabla|\u(x)|^2|^2 \d x.
	\end{align*}
	Moreover, we have the following result (see Lemma 2.1, \cite{MTM6}):
	\begin{align*}
		0 \leq \int_{\T} |\nabla\u(x)|^{2}|\u(x)|^{r-1} \d x	&\leq \int_{\T} |\u(x)|^{r-1}\u(x) \cdot (-\Delta \u(x)) \d x \\&\leq 	r\int_{\T} |\nabla\u(x)|^{2}|\u(x)|^{r-1} \d x.
	\end{align*}
	Using an integration by parts over $\T$ and then using the divergence free condition on the velocity $\u$ (that is, $\nabla \cdot \u=0$) and periodicity of the pressure $p$ in the resulting equation, we obtain
	\begin{align}\label{n16}
		\langle \nabla p,-  \Delta \u \rangle=0.
	\end{align}
	\vskip 0.2 cm
	\noindent\textbf{Case I:} \emph{$d=2$ and $r \in [1,3]$.}  
		We estimate the term $|(\f g,-\Delta \u)|$ using H\"older's and Young's inequalities as
	\begin{align}\label{n17}
		|(\f g,-\Delta \u)|&\leq\|g\|_{\mathrm{L}^{\infty}}\|\f\|_{\H}\|\Delta\u\|_{\H}\leq\frac{\mu}{2}\|\Delta\u\|_{\H}^2+\frac{1}{2\mu}\|g\|_{\mathrm{L}^{\infty}}^2\|\f\|_{\H}^2. 
	\end{align}
	For $d=2$, over a torus, we know that (see Lemma 3.1, \cite{Te,Te1}),
	\begin{align}\label{n18}
		\langle(\u \cdot \nabla) \u,-\Delta \u \rangle=0.
	\end{align}
	Substituting the estimates \eqref{n16}-\eqref{n18} in \eqref{n15},  we deduce 
		\begin{align*}
		 \frac{\d}{\d t}\|\nabla\u(t)\|_{\H}^2+2\alpha \|\nabla\u(t)\|_{\H}^2 \leq \frac{1}{\mu}\|g\|_{\mathrm{L}^{\infty}}^2\|\f\|_{\H}^2. 
	\end{align*}
An application of the variation of constants formula yields for all $t\in[t_1,T]$
\begin{align*}
	\|\nabla\u(t)\|_{\H}^2\leq e^{-2\alpha(t-t_1)} \|\nabla\u(t_1)\|_{\H}^2+\frac{1-e^{-2\alpha (t-t_1)}}{2\mu\alpha}\|g\|_0^2\|\f\|_{\H}^2, 
\end{align*}
and \eqref{n111} follows by using \eqref{n-3}.	Substituting \eqref{n16}-\eqref{n18} in \eqref{n15}, and integrating the resulting estimate from $t_1$ to $t$, we arrive at  
	\begin{align}\label{n19}
		\nonumber&\|\nabla \u(t)\|_{\H}^2+\mu\int_{t_{1}}^t\|\Delta \u(s)\|_{\H}^2\d s+2\beta  \int_{t_{1}}^t\big\||\u(s)|^\frac{r-1}{2}|\nabla \u(s)|\big\|_{\H}^2\d s+2 \alpha \int_{t_{1}}^t\|\nabla\u(s)\|_{\H}^2\d s\nonumber\\&\leq  \|\nabla\u(t_1)\|_{\H}^2+\frac{t-t_1}{\mu}\|g\|_0^2\|\f\|_{\H}^2,
	\end{align}
	for all $t\in[t_1,T]$ and we obtain \eqref{n12} by using the estimate \eqref{n-3}.
	\vskip 0.2 cm
	\noindent\textbf{Case II:} 	\emph{$d=2,3$ and $r > 3$.}   Using H\"older's and Young's inequalities, we estimate the terms 	$|(\f g,-\Delta \u)|$ and $|	\langle(\u \cdot \nabla) \u,-  \Delta \u \rangle|$ as
	\begin{align}
		|(\f g,-\Delta \u)|&\leq\|g\|_{\mathrm{L}^{\infty}}\|\f\|_{\H}\|\Delta\u\|_{\H}\leq\frac{\mu}{4}\|\Delta\u\|_{\H}^2+\frac{1}{\mu}\|g\|_{\mathrm{L}^{\infty}}^2\|\f\|_{\H}^2, \label{n20}\\ 
		|	\langle(\u \cdot \nabla) \u,-  \Delta \u \rangle| &\leq \big\||\u| |\nabla \u|\big\|_{\H} \|\Delta\u\|_{\H}
		\leq  \frac{\mu}{4}\|\Delta\u\|_{\H}^2+\frac{1}{\mu}\big\||\u| |\nabla \u|\big\|_{\H}^2.\label{n21}
	\end{align}
	Again using H\"older's and Young's inequalities, we observe the following estimate for $r >3$ 
	\begin{align*}
		&\int_{\T}|\u(x)|^2|\nabla \u(x)|^2 \d x
		=\int_{\T}|\u(x)|^2|\nabla \u(x)|^\frac{4}{r-1}|\nabla \u(x)|^\frac{2(r-3)}{r-1} \d x\\
		&\leq \left(\int_{\T}|\u(x)|^{r-1}|\nabla \u(x)|^2 \d x\right)^\frac{2}{r-1} \left(\int_{\T}|\nabla \u(x)|^2 \d x\right)^\frac{r-3}{r-1}\\
		&\leq \frac{\beta \mu }{2}\int_{\T}|\u(x)|^{r-1}|\nabla \u(x)|^2 \d x+ \frac{r-3}{r-1}\left(\frac{4}{\beta \mu (r-1)}\right)^\frac{2}{r-3}\int_{\T}|\nabla \u(x)|^2 \d x.
	\end{align*}
		Plugging the above estimate in \eqref{n21} results to 
	\begin{align}\label{n22}
		|	\langle(\u \cdot \nabla) \u,-  \Delta \u \rangle|
		&\leq  \frac{\mu}{4}\|\Delta\u\|_{\H}^2+\frac{\beta }{2}\big\||\u|^\frac{r-1}{2}|\nabla\u| \big\|_{\H}^2+\frac{r-3}{\mu(r-1)}\left(\frac{4}{\beta \mu (r-1)}\right)^\frac{2}{r-3}\|\nabla\u\|_{\H}^2.
	\end{align}
	Substituting the  estimates \eqref{n16}, \eqref{n20} and \eqref{n22} in \eqref{n15} and integrating it from $t_1$ to $t$ leads to
	\begin{align}\label{n23} 
		\nonumber&\|\nabla\u(t)\|_{\H}^2+\big(2\alpha- \eta\big) \int_{t_1}^t\|\nabla\u(s)\|_{\H}^2\d s +\beta \int_{t_1}^t \big\||\u(s)|^\frac{r-1}{2}|\nabla\u(s)|\big\|_{\H}^2 \d s
		\nonumber\\&\leq  \|\nabla\u(t_1)\|_{\H}^2+\frac{2(t-t_1)}{\mu}\|g\|_0^2\|\f\|_{\H}^2,
	\end{align}
	for all $t\in[t_1,T]$, where $\eta=\frac{2(r-3)}{\mu(r-1)}\left(\frac{4}{\beta \mu  (r-1)}\right)^\frac{2}{r-3}$. For $\eta <2\alpha$,  \eqref{n13} follows by using \eqref{n-3} in \eqref{n23}. 
	\vskip 0.2 cm
	\noindent\textbf{Case III:} 	\emph{$d=r=3$.} The  term $\beta  \big\||\u(t)|^\frac{r-1}{2}|\nabla\u(t)|\big\|_{\H}^2$ in \eqref{n15} becomes $\beta  \big\||\u(t)| |\nabla\u(t)|\big\|_{\H}^2$. 
	Substituting the  estimates \eqref{n16}, \eqref{n20} and \eqref{n21} in \eqref{n15}, and integrating it from $t_1$ to $t$, we arrive at  
	\begin{align}\label{n24}
		\nonumber&\|\nabla \u(t)\|_{\H}^2+\mu\int_{t_{1}}^t\|\Delta \u(s)\|_{\H}^2\d s +2 \alpha \int_{t_{1}}^t\|\nabla\u(s)\|_{\H}^2\d s  +2\bigg(\beta -\frac{1}{\mu}\bigg)\int_{t_{1}}^t  \big\||\u(s)| |\nabla\u(s)|\big\|_{\H}^2 \d s\nonumber\\&\leq  \|\nabla\u(t_1)\|_{\H}^2+\frac{2(t-t_1)}{\mu }\|g\|_0^2\|\f\|_{\H}^2,
	\end{align}
	for all $t\in[t_1,T]$  and we immediately get \eqref{n14} by using \eqref{n-3} in \eqref{n24}, provided $\beta\mu > 1$. 
\end{proof}

\begin{lemma}\label{lemmae3}
	Let $(\u(\cdot),\nabla p(\cdot))$ be the unique solution of the CBF  equations \eqref{1a}-\eqref{2a} and $\u_0 \in \H$. Let $t_1\in \big(\frac{T}{8},\frac{2T}{8}\big)$ be the time obtained in Lemma \ref{lemmae1}. Then, for all $t \in [t_1,T]$,
	\begin{itemize}
		\item [$(i)$] for $d=2$ and $r \in [1,3]$, we have
		\begin{align}\label{n25}
			\int_{t_1}^\frac{3T}{8}\|\u_t(t)\|_{\H}^2\d t & \leq \bigg(\frac{K_{11}}{ T}+K_{12}\bigg)\|\u_0\|_{\H}^2 +\bigg(\frac{3T}{4}+K_{13}\bigg)\|g\|_0^2\|\f\|_{\H}^2 \no \\& \quad+C \bigg[\|\u_0\|_{\H}+\frac{1}{\alpha}\|g\|_0\|\f\|_{\H}\bigg]^\frac{2}{3}
			\no \\&\quad\times
			\bigg[ \frac{4}{\mu } \left\{ \left(\frac{1}{T}+\frac{\alpha}{8}\right)\|\u_0\|_{\H}^2+\frac{1}{\alpha}\|g\|_0^2\|\f\|_{\H}^2\right\} +\frac{1}{2\mu\alpha}\|g\|_0^2\|\f\|_{\H}^2\bigg]^\frac{4}{3}\no \\&\quad\times \bigg[ \frac{4}{\mu^2 } \left\{ \left(\frac{1}{T}+\frac{\alpha}{8}\right)\|\u_0\|_{\H}^2+\frac{1}{\alpha}\|g\|_0^2\|\f\|_{\H}^2\right\}+\frac{3T}{8\mu^2}\|g\|_0^2\|\f\|_{\H}^2\bigg],
		\end{align}
	where
	\begin{align}\label{321}
		K_{11}=4+\frac{8}{r+1}, \ \ \
		K_{12}=\frac{5 \alpha}{2}+\frac{\alpha}{r+1}, \  \ \ \text{and}  \ \ \ K_{13}=\frac{6}{\alpha}+\frac{8}{(r+1)\alpha},
	\end{align}

		\item  [$(ii)$] for $d=2,3$ and $r >3$, we have
	\begin{align}\label{n26}
	\int_{t_{1}}^\frac{3T}{8}\| \u_t(t)\|_{\H}^2\d t &\leq \bigg(\frac{K_{21}}{T}+K_{22}T+K_{23}\bigg)\|\u_0\|_{\H}^2+\bigg(K_{24}T+K_{25}\bigg)\|g\|_0^2\|\f\|_{\H}^2\no\\&\quad+\frac{3T}{8\alpha }\|g_t\|_0^2\|\f\|_{\H}^2,
	\end{align}
		where 
		\begin{equation}\label{322}
			\left\{
		\begin{aligned}
		&\gamma=\frac{r-3}{r-1}\bigg(\frac{2}{\beta(r-1)}\bigg)^\frac{2}{r-3}, \ \ \	K_{21}=4+\frac{4}{\mu }+\frac{8}{r+1},  \ \ \   K_{22}=\frac{3\alpha}{4}+\frac{3\alpha \gamma}{32\mu}, \\
			&K_{23}=\bigg(\frac{7}{2}+\frac{1}{r+1}+\frac{1}{2\mu}\bigg)\alpha+\frac{\gamma}{2 \mu}, \ \ \ K_{24}=\frac{3}{4 \alpha}+\frac{3}{4 \mu}+\frac{3\gamma}{4\mu\alpha}, \\
			& \text{and} \ \ \ \ \ 
				K_{25}=\bigg( 7+\frac{4}{r+1}+\frac{2}{\mu}\bigg)\frac{2}{\alpha},
		\end{aligned}
		\right. 
\end{equation}
		\item  [$(iii)$] for $d=r=3$ with $\beta \mu >1$, we have
	\begin{align}\label{n29}
		 \ \	\int_{t_{1}}^\frac{3T}{8}\| \u_t(t)\|_{\H}^2\d t &\leq \bigg(\frac{K_{31}}{T}+K_{32}T+K_{33}\bigg)\|\u_0\|_{\H}^2+\bigg(K_{34}T+K_{35}\bigg)\|g\|_0^2\|\f\|_{\H}^2\no \\&\quad+\frac{3T}{8\alpha}\|g_t\|_0^2\|\f\|_{\H}^2,
	\end{align}
		where 
		\begin{equation}\label{325}
			\left\{
		\begin{aligned}
			&K_{31}=6+\frac{2 }{\beta \mu-1},   \ \ \ K_{32}=\frac{3\alpha}{4}, \ \ \
			K_{33}=\bigg({15}+\frac{1 }{(\beta \mu-1)}\bigg) \frac{ \alpha}{4}, \\ 
		&	K_{34}=\frac{3}{4 \alpha}+\frac{3}{8(\beta \mu-1)},  \ \ \
			\text{and} \ \ \	K_{35}=\bigg( 8+\frac{ 1}{(\beta \mu-1)}\bigg)\frac{2}{ \alpha}.
		\end{aligned}
	\right. 
\end{equation}
	\end{itemize} 
	Moreover, there exists a time $t_2 \in \big(\frac{2T}{8},\frac{3T}{8}\big)$ such that
	\begin{itemize}
		\item  [$(iv)$] for $d=2$ and $r \in [1,3]$, we have
		\begin{align}\label{n32}
			\|\u_t(t_2)\|_{\H}^2
			&\leq  \frac{8}{T}\bigg[ \bigg(\frac{K_{11}}{ T}+K_{12}\bigg)\|\u_0\|_{\H}^2 +\bigg(\frac{3T}{4}+K_{13}\bigg)\|g\|_0^2\|\f\|_{\H}^2 \no \\& \quad+C \bigg[\|\u_0\|_{\H}+\frac{1}{\alpha}\|g\|_0\|\f\|_{\H}\bigg]^\frac{2}{3}
			\no \\&\quad\times
			\bigg[ \frac{4}{\mu } \left\{ \left(\frac{1}{T}+\frac{\alpha}{8}\right)\|\u_0\|_{\H}^2+\frac{1}{\alpha}\|g\|_0^2\|\f\|_{\H}^2\right\} +\frac{1}{2\mu\alpha}\|g\|_0^2\|\f\|_{\H}^2\bigg]^\frac{4}{3}\no \\&\quad\times \bigg[ \frac{4}{\mu^2 } \left\{ \left(\frac{1}{T}+\frac{\alpha}{8}\right)\|\u_0\|_{\H}^2+\frac{1}{\alpha}\|g\|_0^2\|\f\|_{\H}^2\right\}+\frac{3T}{8\mu^2}\|g\|_0^2\|\f\|_{\H}^2\bigg]\bigg],
		\end{align}
		\item  [$(v)$] for $d=2,3$ and $r >3$, we have
		\begin{align}\label{n33}
			\|\u_t(t_2)\|_{\H}^2
			&\leq  \frac{8}{T}\bigg\{\bigg(\frac{K_{21}}{T}+K_{22}T+K_{23}\bigg)\|\u_0\|_{\H}^2+\bigg(K_{24}T+K_{25}\bigg)\|g\|_0^2\|\f\|_{\H}^2\no\\&\quad+\frac{3T}{8\alpha}\|g_t\|_0^2\|\f\|_{\H}^2\bigg\},
		\end{align}
		\item  [$(vi)$] for $d=r=3$ with $\beta \mu >1$, we have
		\begin{align}\label{n34}
			\|\u_t(t_2)\|_{\H}^2
			&\leq  \frac{8}{T} \bigg\{\bigg(\frac{K_{31}}{T}+K_{32}T+K_{33}\bigg)\|\u_0\|_{\H}^2+\bigg(K_{34}T+K_{35}\bigg)\|g\|_0^2\|\f\|_{\H}^2\no\\&\quad+\frac{3T}{8\alpha}\|g_t\|_0^2\|\f\|_{\H}^2\bigg\}.
		\end{align}
	\end{itemize}
\end{lemma}
\begin{proof}
	Taking the inner product with $\u_t(\cdot)$ in \eqref{1a} and integrating the resulting equation over $\T$, we obtain 
	\begin{align}\label{n35}
		\nonumber&\|\u_t(t)\|_{\H}^2+\frac{\mu}{2} \frac{\d}{\d t}\|\nabla\u(t)\|_{\H}^2+\frac{\alpha }{2}\frac{\d}{\d t}\|\u(t)\|_{\H}^2+\frac{\beta}{r+1}\frac{\d}{\d t}\|\u(t)\|_{\widetilde{\L}^{r+1}}^{r+1}\nonumber\\&=(\f g(t),\u_t(t))-((\u(t) \cdot\nabla)\u(t),\u_t(t)),
	\end{align}
	where we have used the fact that $\langle\nabla p,\u_t\rangle=0$. For $r\geq 1$, it can be easily seen that
\begin{align*}
	\left(|\u(t)|^{r-1}\u(t),\u_t(t)\right)=\frac{1}{r+1}\frac{\d}{\d t}\|\u(t)\|_{\widetilde{\L}^{r+1}}^{r+1}. 
\end{align*}
	\vskip 0.2 cm
	\noindent\textbf{Case I:} 	\emph{$d=2$ and $r \in [1,3]$.} We estimate the term $|(\f g,\u_t)|$ using H\"older's and Young's inequalities as
	\begin{align}\label{n36}
		|(\f g,\u_t)|\leq\|g\|_{\mathrm{L}^{\infty}}\|\f\|_{\H}\|\u_t\|_{\H}\leq\frac{1}{4}\|\u_t\|_{\H}^2+\|g\|_{\mathrm{L}^{\infty}}^2\|\f\|_{\H}^2. 
	\end{align}
	Making use of H\"older's, Sobolev's and Young's inequalities, we estimate the term  $|((\u \cdot\nabla)\u,\u_t)|$ as
	\begin{align}\label{n37}
		|((\u \cdot\nabla)\u,\u_t)|&\leq\|\u\|_{\widetilde{\L}^3}\|\nabla \u\|_{\widetilde{\L}^6}\|\u_t\|_{\H}\leq C \| \u\|_{\H}^\frac{1}{3}\|\nabla \u\|_{\H}^\frac{2}{3}\|\Delta \u\|_{\H}\|\u_t\|_{\H}\no\\& \leq \frac{1}{4}\|\u_t\|_{\H}^2+C\| \u\|_{\H}^\frac{2}{3}\|\nabla \u\|_{\H}^\frac{4}{3}\|\Delta \u\|_{\H}^2.
	\end{align}
	Substituting the estimates \eqref{n36} and \eqref{n37} in \eqref{n35}, and integrating the resulting estimate from $t_1$ to $\frac{3T}{8}$, we arrive at  
	\begin{align}\label{n38}
	\nonumber&\int_{t_{1}}^\frac{3T}{8}\|\u_t(s)\|_{\H}^2\d s+\mu\|\nabla \u(3T/8)\|_{\H}^2+\alpha\|\u(3T/8)\|_{\H}^2+\frac{2\beta}{r+1}\|\u(3T/8)\|_{\widetilde{\L}^{r+1}}^{r+1}\nonumber\\&\leq \mu \|\nabla\u(t_1)\|_{\H}^2+\alpha\|\u(t_1)\|_{\H}^2+\frac{2\beta}{r+1}\|\u(t_1)\|_{\widetilde{\L}^{r+1}}^{r+1}\nonumber\\&\quad+2\bigg(\frac{3T}{8}-t_1\bigg)\|g\|_0^2\|\f\|_{\H}^2+C\int_{t_1}^\frac{3T}{8}\| \u(s)\|_{\H}^\frac{2}{3}\|\nabla \u(s)\|_{\H}^\frac{4}{3}\|\Delta\u(s)\|_{\H}^2 \d s
		\no \\& \leq \mu \|\nabla\u(t_1)\|_{\H}^2+\alpha\sup_{t\in[0,T]}\|\u(t)\|_{\H}^2+\frac{2\beta}{r+1}\|\u(t_1)\|_{\widetilde{\L}^{r+1}}^{r+1}\no \\&\quad+\frac{3T}{4}\|g\|_0^2\|\f\|_{\H}^2
		+C \sup_{t\in[t_1,3T/8]}\left(\| \u(t)\|_{\H}^\frac{2}{3}\|\nabla \u(t)\|_{\H}^\frac{4}{3}\right)\int_{t_1}^\frac{3T}{8}\|\Delta\u(t)\|_{\H}^2 \d t,
	\end{align}
	for all $t\in[t_1,T]$. 
	Thus, by using the estimates \eqref{n-1}, \eqref{n-3}, \eqref{n-4}, \eqref{n111} and \eqref{n12} in \eqref{n38}, we immediately get \eqref{n25}. 
	\vskip 0.2 cm
	\noindent\textbf{Case II:} 	\emph{$d=2,3$ and $r>3$.}
	We have
	\begin{align}\label{n39}
		(\f g(t),\u_t(t))=\frac{\d}{\d t}(\f g(t),\u(t))-(\f g_{t}(t),\u(t)), 
	\end{align}
	and estimate the term $|(\f g_t,\u)|$ as
	\begin{align}\label{n40}
		|(\f g_t,\u)| \leq \|g_t\|_{\mathrm{L}^\infty}\|\f\|_{\H} \|\u\|_{\H} \leq \frac{\alpha}{2} \|\u\|_{\H}^2+\frac{1}{2 \alpha}\|g_t\|_{\mathrm{L}^\infty}^2\|\f\|_{\H}^2.
	\end{align}
	An estimate similar to \eqref{n22} yields
	\begin{align}\label{n41}
		|	\langle(\u \cdot \nabla) \u, \u_t \rangle|
		\leq  \frac{1}{2}\|\u_t\|_{\H}^2+\frac{\beta }{2}\big\||\u|^\frac{r-1}{2}|\nabla\u| \big\|_{\H}^2+\frac{\gamma}{2}\|\nabla\u\|_{\H}^2,
	\end{align}
	where $\gamma=\frac{r-3}{r-1}\bigg(\frac{2}{\beta(r-1)}\bigg)^\frac{2}{r-3}$. Substituting the  estimates \eqref{n39}-\eqref{n41} in \eqref{n35} and integrating it from $t_1$ to $\frac{3T}{8}$, we obtain
	\begin{align}\label{n42}
		\nonumber&\int_{t_{1}}^\frac{3T}{8}\|\u_t(s)\|_{\H}^2\d s+\mu\|\nabla \u(3T/8)\|_{\H}^2+\alpha\|\u(3T/8)\|_{\H}^2+\frac{2\beta}{r+1}\|\u(3T/8)\|_{\widetilde{\L}^{r+1}}^{r+1}\nonumber\\&\leq \mu \|\nabla\u(t_1)\|_{\H}^2+\alpha\|\u(t_1)\|_{\H}^2+\frac{2\beta}{r+1}\|\u(t_1)\|_{\widetilde{\L}^{r+1}}^{r+1} +\alpha \int_{t_{1}}^\frac{3T}{8} \| \u(s)\|_{\H}^2 \d s\no \\&\quad
		+\frac{\frac{3T}{8}-t_1}{\alpha}\|g_t\|_0^2\|\f\|_{\H}^2 +2 (\f g(t),\u(t)) -2 (\f g(t_{1}),\u(t_1))\no \\& \quad+\beta\int_{t_{1}}^\frac{3T}{8}\big\||\u(s)|^{\frac{r-1}{2}}|\nabla \u(s)|\big\|_{\H}^2 \d s+\gamma\int_{t_{1}}^\frac{3T}{8} \|\nabla \u(s)\|_{\H}^2 \d s
		\nonumber\\&\leq \mu \|\nabla\u(t_1)\|_{\H}^2+\bigg(\alpha+\frac{3 \alpha T}{8}\bigg)\sup_{t\in[0,T]}\|\u(t)\|_{\H}^2+\frac{2\beta}{r+1}\|\u(t_1)\|_{\widetilde{\L}^{r+1}}^{r+1} +\frac{3T}{8\alpha}\|g_t\|_0^2\|\f\|_{\H}^2 
	\no \\&\quad	 +4\|g\|_0\|\f\|_{\H}\|\u_0\|_{\H}+\frac{4}{ \alpha}\|g\|_0^2\|\f\|_{\H}^2 +\beta\int_{t_{1}}^\frac{3T}{8}\big\||\u(s)|^{\frac{r-1}{2}}|\nabla \u(s)|\big\|_{\H}^2 \d s\no \\&\quad+\frac{\gamma}{\mu}\bigg(\frac{1}{2}\|\u_0\|_{\H}^2+\frac{3T}{8}\|g\|_0\|\f\|_{\H}\|\u_0\|_{\H}+\frac{3T}{8}\|g\|_0^2\|\f\|_{\H}^2\bigg)
		\nonumber\\&\leq \mu \|\nabla\u(t_1)\|_{\H}^2+\bigg(\alpha+\frac{3 \alpha T}{8}\bigg)\sup_{t\in[0,T]}\|\u(t)\|_{\H}^2+\frac{2\beta}{r+1}\|\u(t_1)\|_{\widetilde{\L}^{r+1}}^{r+1}
	\no \\&\quad	+\frac{3T}{8\alpha}\|g_t\|_0^2\|\f\|_{\H}^2 
		+\alpha \|\u_0\|_{\H}^2+\frac{8}{\alpha}\|g\|_0^2\|\f\|_{\H}^2 +\beta\int_{t_{1}}^\frac{3T}{8}\big\||\u(s)|^{\frac{r-1}{2}}|\nabla \u(s)|\big\|_{\H}^2 \d s
		 \no \\& \quad+\frac{\gamma}{\mu}\bigg\{\bigg(\frac{1}{2}+\frac{3 \alpha T}{32}\bigg)\|\u_0\|_{\H}^2+\frac{3T}{4\alpha}\|g\|_0^2\|\f\|_{\H}^2\bigg\},
	\end{align}
	for all $t\in[t_1,T]$ and we have used H\"older's and Young's inequalities. One can obtain \eqref{n26} by using the estimates \eqref{n-1}-\eqref{n-4} and \eqref{n13}.
	\vskip 0.2 cm
	\noindent\textbf{Case III:} 	\emph{$d=r=3$.}
	We estimate $|	\langle(\u \cdot \nabla) \u, \u_t \rangle|$ using H\"older's and Young's inequalities as
	\begin{align}\label{n43}
		|	\langle(\u \cdot \nabla) \u, \u_t \rangle|
		\leq  \frac{1}{2}\|\u_t\|_{\H}^2+\frac{1 }{2}\big\||\u||\nabla\u| \big\|_{\H}^2.
	\end{align}
	Substituting the  estimates \eqref{n39}, \eqref{n40} and \eqref{n43} in \eqref{n35} and integrating it from $t_1$ to $\frac{3T}{8}$ results to (cf. \eqref{n42})
	\begin{align}\label{n44}
		\nonumber&\int_{t_{1}}^\frac{3T}{8}\|\u_t(s)\|_{\H}^2\d s+\mu\|\nabla \u(3T/8)\|_{\H}^2+\alpha\|\u(3T/8)\|_{\H}^2+\frac{\beta}{2}\|\u(3T/8)\|_{\widetilde{\L}^{4}}^{4}\nonumber\\&\leq \mu \|\nabla\u(t_1)\|_{\H}^2+\bigg(\alpha +\frac{3T}{8}\bigg)\sup_{t\in[0,T]}\|\u(t)\|_{\H}^2+\frac{\beta}{2}\|\u(\cdot,t_1)\|_{\widetilde{\L}^{4}}^{4} 
		\no \\&\quad+\frac{3T}{8\alpha}\|g_t\|_0^2\|\f\|_{\H}^2 
		+\alpha \|\u_0\|_{\H}^2+\frac{8}{\alpha} \|g\|_0^2\|\f\|_{\H}^2 +\int_{t_{1}}^\frac{3T}{8}\big\||\u(s)||\nabla \u(s)|\big\|_{\H}^2 \d s,
	\end{align}
	for all $t\in[t_1,T]$ and \eqref{n29} follows by using \eqref{n-1}-\eqref{n-4} and \eqref{n14}.
	
	Recalling that $\u \in \C^1\big([T/8,T];\V\big)$ implies \eqref{n32}-\eqref{n34}. By the mean value theorem, there exists a time $t_2 \in \big(\frac{2T}{8},\frac{3T}{8}\big)$ such that
	\begin{align*}
		\|\u_t(t_2)\|_{\H}^2&=\frac{1}{\frac{3T}{8}-\frac{2T}{8}}\int_{2T/8}^{3T/8}	\|\u_t(t)\|_{\H}^2 \d t\leq \frac{8}{T} \int_{t_1}^{3T/8} \|\u_t(t)\|_{\H}^2 \d t,
	\end{align*}
and \eqref{n32} follows by using \eqref{n25}. Applying the similar arguments as above, we get
	\begin{align*}
		\|\u_t(t_2)\|_{\H}^2&=\frac{1}{\frac{3T}{8}-\frac{2T}{8}}\int_{2T/8}^{3T/8}	\|\u_t(t)\|_{\H}^2 \d t\leq \frac{8}{T} \int_{t_1}^{3T/8} \|\u_t(t)\|_{\H}^2 \d t,
	\end{align*}
	which leads to \eqref{n33} and \eqref{n34} from  \eqref{n26} and \eqref{n29}, respectively.
\end{proof}

\begin{lemma}\label{lemmae4}
	Let $(\u(\cdot),\nabla p(\cdot))$ be the unique solution of the CBF  equations \eqref{1a}-\eqref{2a} and $\u_0 \in \H$. Let $t_2 \in \big(\frac{2T}{8},\frac{3T}{8}\big)$ be the same time obtained in Lemma \ref{lemmabe2}. Then
	\begin{itemize}
		\item [$(i)$]
		for  $d=2$ and $r \in[1,3]$, we have 
		\begin{align}\label{n45}
			\sup_{t\in[t_2,T]}\|\u_t(t)\|_{\H}^2
			&\leq \bigg[ \bigg(\frac{K_{11}}{ T}+K_{12}\bigg)\|\u_0\|_{\H}^2 +\bigg(\frac{3T}{4}+K_{13}\bigg)\|g\|_0^2\|\f\|_{\H}^2 \no \\& \quad+C \bigg[\|\u_0\|_{\H}+\frac{1}{\alpha}\|g\|_0\|\f\|_{\H}\bigg]^\frac{2}{3}
			\no \\&\quad\times
			\bigg[ \frac{4}{\mu } \left\{ \left(\frac{1}{T}+\frac{\alpha}{8}\right)\|\u_0\|_{\H}^2+\frac{1}{\alpha}\|g\|_0^2\|\f\|_{\H}^2\right\} +\frac{1}{2\mu\alpha}\|g\|_0^2\|\f\|_{\H}^2\bigg]^\frac{4}{3}\no \\&\quad\times \bigg[ \frac{4}{\mu^2 } \left\{ \left(\frac{1}{T}+\frac{\alpha}{8}\right)\|\u_0\|_{\H}^2+\frac{1}{\alpha}\|g\|_0^2\|\f\|_{\H}^2\right\}+\frac{3T}{8\mu^2}\|g\|_0^2\|\f\|_{\H}^2\bigg]\bigg]\no \\& \quad\times\bigg[\frac{8}{T}+ \frac{8}{\mu^2 } \left\{ \left(\frac{1}{T}+\frac{\alpha}{8}\right)\|\u_0\|_{\H}^2+\frac{1}{ \alpha}\|g\|_0^2\|\f\|_{\H}^2\right\}+\frac{1}{\mu^2\alpha}\|g\|_0^2\|\f\|_{\H}^2\bigg]\no \\& \quad+ \frac{T}{ \alpha}\|g_t\|_{0}^2\|\f\|_{\H}^2,
		\end{align}
		where $K_{1i}, \ i=1,2,3,$  are defined in Lemma \ref{lemmae3},
		\item [$(ii)$]
		for $d=2,3$ and $ r > 3$ with $\eta^*=\frac{r-3}{ \mu (r-1)}\left(\frac{2}{\beta \mu (r-1)}\right)^\frac{2}{r-3}<\mu \lambda_1+\alpha$, we have 
		\iffalse
		\begin{align}\label{n461}
			&\sup_{t\in[t_2,T]}\|\u_t(t)\|_{\H}^2+\beta \int_{t_2}^T\big\||\u(t)|^\frac{r-1}{2}\u_t(t)\big\|_{\H}^2 \d t
			\no\\&\quad+\bigg(\mu \lambda_1+\alpha-\frac{r-3}{ \mu (r-1)}\left(\frac{2}{\beta \mu (r-1)}\right)^\frac{2}{r-3}\bigg) \int_{t_2}^t\|\u_t(t)\|_{\H}^2\d t \no\\&\leq \frac{16}{T}\bigg[K_{31}\|\u_0\|_{\H}^2+K_{32}\|g\|_0^2\|\f\|_{\H}^2+\frac{T}{\mu \lambda_1}\|g_t\|_0^2\|\f\|_{\H}^2\no \\& \quad+\bigg\{\bigg(1+\frac{\alpha}{\mu \lambda_1}+\frac{1}{\mu}+\frac{2}{r+1}\bigg)\frac{8}{ \alpha T(2 \mu \lambda_1+\alpha)}\no\\&\quad+\bigg(2+(\mu+\gamma)\frac{1}{2 \alpha}\bigg)\frac{T}{\mu }\bigg\}\|g\|_0^2\|\f\|_{\H}^2\bigg]+ \frac{T}{ \alpha}\|g_t\|_{0}^2\|\f\|_{\H}^2,
		\end{align}
	\fi
		\begin{align}\label{n46}
			\sup_{t\in[t_2,T]}\|\u_t(t)\|_{\H}^2&\leq \frac{8}{T}\bigg\{ \bigg(\frac{K_{21}}{T}+K_{22}T+K_{23}\bigg)\|\u_0\|_{\H}^2+\bigg(K_{24}T+K_{25}\bigg)\|g\|_0^2\|\f\|_{\H}^2\no \\&\quad+\frac{3T}{8\alpha }\|g_t\|_0^2\|\f\|_{\H}^2\bigg\}+ \frac{1}{ \alpha(\alpha-\eta^*)}\|g_t\|_{0}^2\|\f\|_{\H}^2,
		\end{align}
		where $\gamma$ and $K_{2i}, \ i=1,\ldots,5,$  are defined in Lemma \ref{lemmae3},
		\item [$(iii)$]
		for $d=r=3$ with $\beta \mu > 1$, we have 
		\iffalse
		\begin{align}\label{n471}
			&\sup_{t\in[t_2,T]}\|\u_t(t)\|_{\H}^2+\mu\int_{t_2}^T\|\nabla\u_t(t)\|_{\H}^2\d t +2\bigg(\beta -\frac{1}{2\mu}\bigg)\int_{t_2}^T\left\||\u(t)|\u_t(t)\right\|_{\H}^2  \d t 	\no\\&\leq  \frac{16}{T} \bigg[K_{41}\|\u_0\|_{\H}^2+K_{42}\|g\|_0^2\|\f\|_{\H}^2+\frac{T}{\mu \lambda_1}\|g_t\|_0^2\|\f\|_{\H}^2\no \\& \quad+\bigg\{\bigg(3+\frac{2\alpha}{\mu \lambda_1}+\frac{ \beta }{(\beta \mu-1)}\bigg) \frac{4}{ \alpha T(2 \mu \lambda_1+\alpha)}\no\\&\quad+\bigg(\frac{ \beta  }{(\beta \mu-1)}+\frac{1}{\alpha}\bigg)\frac{T}{2}\bigg\}\|g\|_0^2\|\f\|_{\H}^2\bigg]+ \frac{T}{ \alpha}\|g_t\|_{0}^2\|\f\|_{\H}^2,
		\end{align}
	\fi
		\begin{align}\label{n47}
			\sup_{t\in[t_2,T]}\|\u_t(t)\|_{\H}^2	&\leq  \frac{8}{T} \bigg\{\bigg(\frac{K_{31}}{T}+K_{32}T+K_{33}\bigg)\|\u_0\|_{\H}^2+\bigg(K_{34}T+K_{35}\bigg)\|g\|_0^2\|\f\|_{\H}^2\no\\&\quad+\frac{3T}{8\alpha}\|g_t\|_0^2\|\f\|_{\H}^2\bigg\}+ \frac{1}{ \alpha^2}\|g_t\|_{0}^2\|\f\|_{\H}^2,
		\end{align}
		where $K_{3i}, \ i=1,\ldots,5,$  are defined in Lemma \ref{lemmae3}. 
	\end{itemize}
\end{lemma}
\begin{proof} 
	Applying $\partial / \partial t$ to the equation \eqref{1a}, we find
	\begin{align}\label{n48}
		\u_{tt}-\mu \Delta \u_t+(\u \cdot \nabla)\u_t+(\u_t \cdot \nabla)\u+\alpha \u_t+\beta  \mathcal{C}'(\u)\u_t+\nabla p_t=\f g_t,
	\end{align}
	where $\mathcal{C}'(\cdot)$ is defined in \eqref{29}.	Multiplying both sides of \eqref{n48} by  $\u_t(\cdot)$ and integrating over $\T$, we deduce
	\begin{align}\label{n49}
		&\frac{1}{2}\frac{\d}{\d t}\|\u_t(t)\|_{\H}^2+\mu\|\nabla\u_t(t)\|_{\H}^2+\alpha \|\u_t(t)\|_{\H}^2+\beta( \mathcal{C}'(\u)\u_t(t),\u_t(t))\nonumber\\&=(\f g_t(t),\u_t(t))-\big((\u_t(t) \cdot \nabla)\u(t),\u_t(t)\big),
	\end{align}
	where we have used the fact that $\langle(\u \cdot \nabla)\u_t,\u_t\rangle=0$ and $\langle\nabla p_t,\u_t\rangle=0$. Applying H\"older's and Young's inequalities, we obtain
\begin{align}\label{n50}
	|(\f g_t,\u_t)| 
	\leq \frac{\alpha}{2}\|\u_t\|_{\H}^2+\frac{1}{2 \alpha}\|g_t\|_{\mathrm{L}^{\infty}}^2\|\f\|_{\H}^2.
\end{align}
	\vskip 0.2 cm
	\noindent\textbf{Case I:} \emph{$d=2$ and $r \in [1,3]$.} 
		Making use of H\"older's, Ladyzhenskaya's and Young's inequalities, we find
	\begin{align}\label{n51}
		\big|	\big((\u_t \cdot \nabla)\u,\u_t\big)\big| &\leq \|\u_t\|_{\widetilde{\L}^4}^2 \|\nabla\u\|_{\H} \leq \sqrt{2}\|\u_t\|_{\H}\|\nabla\u_t\|_{\H}\|\nabla\u\|_{\H}\no\\& \leq \frac{\mu}{2}\|\nabla\u_t\|_{\H}^2+\frac{1}{\mu}\|\u_t\|_{\H}^2\|\nabla\u\|_{\H}^2.
	\end{align}
	From \eqref{29}, we get 
	\begin{align}\label{n52}
		\big(\mathcal{C}'(\u)\u_t,\u_t\big)\no &= \left(|\u|^{r-1}\u_t+(r-1)\frac{\u}{|\u|^{3-r}}(\u\cdot\u_t),\u_t\right) \\&= \big\||\u|^\frac{r-1}{2}|\u_t|\big\|_{\H}^2+ (r-1)\bigg\|\frac{1}{|\u|^\frac{3-r}{2}}(\u \cdot \u_t)\bigg\|_{\H}^2,
	\end{align}  
	where in the final term, the norm is zero whenever $\u=\mathbf{0}$.	
	Plugging the relations \eqref{n50}-\eqref{n52} in \eqref{n49}, and  integrating the resulting relation from $t_2$ to $t$ results to
	\begin{align}\label{n53}
		&\|\u_t(t)\|_{\H}^2+\mu\int_{t_2}^t\|\nabla\u_t(s)\|_{\H}^2\d s+\alpha \int_{t_2}^t\|\u_t(s)\|_{\H}^2\d s\no\\&\quad +2\beta \int_{t_2}^t\big\||\u(s)|^\frac{r-1}{2}|\u_t(s)|\big\|_{\H}^2 \d s+2\beta(r-1)\int_{t_2}^t\bigg\|\frac{1}{|\u(s)|^\frac{3-r}{2}}\big(\u(s) \cdot \u_t(s)\big)\bigg\|_{\H}^2\d s\nonumber\\&
		\leq \|\u_t(t_2)\|_{\H}^2+ \frac{t-t_2}{ \alpha}\|g_t\|_{0}^2\|\f\|_{\H}^2+\frac{2}{\mu}\int_{t_2}^t\|\u_t(s)\|_{\H}^2\|\nabla\u(s)\|_{\H}^2\d s
		\nonumber\\&
		\leq \|\u_t(t_2)\|_{\H}^2+ \frac{T}{ \alpha}\|g_t\|_{0}^2\|\f\|_{\H}^2+\frac{2}{\mu}\sup_{t\in[t_2,T]}\|\nabla\u(t)\|_{\H}^2 \int_{t_1}^\frac{3T}{8}\|\u_t(s)\|_{\H}^2 \d s,
	\end{align} 
	for all $t\in[t_2,T]$.  Thus, from \eqref{n53}, one reaches at \eqref{n45} by using the estimates \eqref{n12}, \eqref{n25} and \eqref{n32}.
	\vskip 0.2 cm
	\noindent\textbf{Case II:} 	\emph{$d=2,3$ and $r > 3$.} 
		An estimate similar to \eqref{n22} gives the estimate
	\begin{align}\label{n54}
		\big|	\big((\u_t \cdot \nabla)\u,\u_t\big)\big|&	\leq \frac{\mu}{2}\|\nabla \u_t\|_{\H}^2 +\frac{\beta }{2}\big\||\u|^\frac{r-1}{2}|\u_t|\big\|_{\H}^2+ \frac{r-3}{2\mu(r-1)}\left(\frac{2}{\beta \mu (r-1)}\right)^\frac{2}{r-3}\|\u_t\|_{\H}^2.
	\end{align}
 From \eqref{29}, we get 
	\begin{align}\label{n55}
		\big(\mathcal{C}'(\u)\u_t,\u_t\big)&=\big(|\u|^{r-1}\u_t+(r-1)\u|\u|^{r-3}(\u\cdot\u_t),\u_t\big)\no\\& =  \big\||\u|^\frac{r-1}{2}|\u_t|\big\|_{\H}^2+ (r-1)\big\||\u|^\frac{r-3}{2}(\u \cdot \u_t)\big\|_{\H}^2.
	\end{align}
	Substituting the estimates \eqref{n50}, \eqref{n54} and \eqref{n55} in \eqref{n49}, we deduce 
	\begin{align}\label{n56}
		&\frac{\d}{\d t} \|\u_t(t)\|_{\H}^2+\mu \|\nabla\u_t(t)\|_{\H}^2+\big(\alpha-\eta^*\big)\|\u_t(t)\|_{\H}^2+\beta\big\||\u(t)|^\frac{r-1}{2}|\u_t(t)|\big\|_{\H}^2\no \\& \quad+2\beta(r-1)\big\||\u|^\frac{r-3}{2}(\u \cdot \u_t)\big\|_{\H}^2 \leq \frac{1}{ \alpha}\|g_t\|_{\mathrm{L}^\infty}^2\|\f\|_{\H}^2,
	\end{align}
	where $\eta^*=\frac{r-3}{ \mu (r-1)}\left(\frac{2}{\beta \mu (r-1)}\right)^\frac{2}{r-3}$.
	\iffalse
	 Integrate \eqref{n56} from $t_2$ to $t$ results in
\begin{align}\label{n571}
		&\|\u_t(t)\|_{\H}^2+\big(\mu \lambda_1+\alpha-\eta^*\big) \int_{t_2}^t\|\u_t(s)\|_{\H}^2\d s\no\\&\quad+\beta \int_{t_2}^t\big\||\u(s)|^\frac{r-1}{2}|\u_t(s)|\big\|_{\H}^2 \d s +2 \beta(r-1)\int_{t_2}^t\big\||\u|^\frac{r-3}{2}(\u \cdot \u_t)\big\|_{\H}^2 \d s\nonumber\\&
		\leq \|\u_t(\cdot,t_2)\|_{\H}^2+ \frac{t-t_2}{ \alpha}\|g_t\|_{0}^2\|\f\|_{\H}^2,
	\end{align} 
	for all $t\in[t_2,T]$. Thus, from the relation \eqref{n57}, 
	one can easily get \eqref{n46} by using \eqref{n33}. 
	\fi
	An application of the variation of constants formula in \eqref{n56} gives for all $t \in[t_2,T]$
	\begin{align}\label{n57}
		\|\u_t(t)\|_{\H}^2 &\leq e^{-(\alpha-\eta^* )(t-t_2)} \|\u_t(t_2)\|_{\H}^2  +\frac{1- e^{-(\alpha-\eta^* )(t-t_2)}}{\alpha (\alpha-\eta^*)}\|g_t\|_{0}^2\|\f\|_{\H}^2\no \\&
		\leq \|\u_t(t_2)\|_{\H}^2  +\frac{1}{\alpha (\alpha-\eta^*)}\|g_t\|_{0}^2\|\f\|_{\H}^2,
	\end{align}
	which leads to \eqref{n46} by using \eqref{n33}, provided $\eta^*<\alpha$.
	\vskip 0.2 cm
	\noindent\textbf{Case III:} \emph{$d=r=3$ and $\beta \mu > 1$.} Applying H\"older's and Young's inequalities, we obtain
	\begin{align}\label{n58}
		\big|	\big((\u_t \cdot \nabla)\u,\u_t\big)\big| \leq \frac{\mu}{2}\|\nabla \u_t\|_{\H}^2 + \frac{1}{2\mu}\big\||\u||\u_t|\big\|_{\H}^2.
	\end{align}
	For $r=3$, the equality \eqref{n55} becomes 
	\begin{align}\label{n59}
		\big(\mathcal{C}'(\u)\u_t,\u_t\big)=  \big\||\u||\u_t|\big\|_{\H}^2+ 2\big\|(\u \cdot \u_t)\big\|_{\H}^2.
	\end{align}
	Plugging the relations \eqref{n50}, \eqref{n58} and \eqref{n59} in \eqref{n49}, we have 
	\begin{align}\label{n60}
		&\frac{\d}{\d t}\|\u_t(t)\|_{\H}^2+\alpha \|\u_t(t)\|_{\H}^2+2\bigg(\beta -\frac{1}{2\mu}\bigg)\big\||\u(t)||\u_t(t)|\big\|_{\H}^2  
		\leq  \frac{1}{ \alpha}\|g_t\|_{0}^2\|\f\|_{\H}^2,
	\end{align} 
	for all $t\in[t_2,T]$. Applying the variation of constants formula in \eqref{n60}, we get
	\begin{align}\label{n61}
		\|\u_t(t)\|_{\H}^2 &\leq e^{-\alpha(t-t_2)} \|\u_t(t_2)\|_{\H}^2  +\frac{1- e^{-\alpha (t-t_2)}}{\alpha^2}\|g_t\|_{0}^2\|\f\|_{\H}^2\no \\&
		\leq \|\u_t(t_2)\|_{\H}^2  +\frac{1}{\alpha^2}\|g_t\|_{0}^2\|\f\|_{\H}^2,
	\end{align}
 for all $t \in[t_2,T]$ and \eqref{n47} follows by using \eqref{n34} in \eqref{n61}, provided $\beta\mu > 1$.
\end{proof}

\section{Proof of Theorem \ref{thm2} $(i)$}\label{sec4}\setcounter{equation}{0}
The energy estimates mentioned in Section \ref{sec3} and Appendix \ref{sec6} allow us to prove the existence and uniqueness of a solution to the inverse problem \eqref{1a}-\eqref{1d} as well as  to establish the stability of the solution. To show the existence of a solution to the inverse problem \eqref{1a}-\eqref{1d}, we use  Theorem \ref{thm1}  to prove that the nonlinear operator $\mathcal{B}$ has a fixed point in $\mathcal{D}$, which follows from an application of the well-known  Tikhonov fixed point theorem.

	Subsequently, arguments for the existence of the solution of the inverse problem \eqref{1a}-\eqref{1d} are based on the works \cite{JF,PV1}, where the existence of the solution of the inverse problem for 2D NSE has been examined by exploiting Tikhonov's fixed point theorem.
	\begin{theorem}[Tikhonov's fixed point theorem, \cite{NS}]\label{thmS}
	Let $\mathcal{D}$ be a non-empty bounded closed convex subset of a separable reflexive Banach space $\X$ and let $\mathcal{B} : \mathcal{D} \to \mathcal{D}$ be a weakly continuous mapping (that is, if $\u_n \in \mathcal{D}, \ \u_n \rightharpoonup \u$ weakly in $\X$, then $\mathcal{B}\u_n \rightharpoonup\mathcal{B} \u$ weakly in $\X$ as well). Then, $\mathcal{B}$ has at least one fixed point in $\mathcal{D}$.
\end{theorem}
\subsection{Existence}
We commence proving Theorem \ref{thm2} (i) by ensuring that the nonlinear operator $\mathcal{B}$ defined in \eqref{1i} fulfills all the assumptions given in Theorem \ref{thmS}. The following remark is crucial in this work and the subsequent lemma demonstrates that the operator $\mathcal{B}$ maps $\mathcal{D}$ into itself. 

	\begin{remark}\label{rem4.2}
	Let us specify the choice of $M$ for defining the closed ball $\mathcal{D}$. For $d=2,3$ and $r >3$,  we define $M$ by
\begin{align*}
			&M=M(\alpha,\mu)\nonumber\\&:= \frac{\bigg(\frac{8K_{21}}{T^2}+8K_{22}+\frac{8K_{23}}{T}\bigg)^\frac{1}{2}\|\u_0\|_{\H}+\big\|(\boldsymbol{\varphi}\cdot \nabla)\boldsymbol{\varphi}+\nabla \psi-\mu \Delta\boldsymbol{\varphi}+\alpha \boldsymbol{\varphi}+ \beta |\boldsymbol{\varphi}|^{r-1}\boldsymbol{\varphi}\big\|_{\H}} {g_T-\bigg\{\bigg(8K_{24}+\frac{8K_{25}}{T}\bigg)^\frac{1}{2}\|g\|_0+\bigg(\frac{3}{\alpha }+ \frac{1}{ \alpha(\alpha-\eta^*)}\bigg)^\frac{1}{2}\|g_t\|_{0}\bigg\}},
		\end{align*}
where $K_{2i}, \ i=1,\ldots,5,$  are defined in Lemma \ref{lemmae3} (see \eqref{322}), and $\eta^*=\frac{r-3}{\mu(r-1)}\left(\frac{2}{\beta \mu  (r-1)}\right)^\frac{2}{r-3}$, and $\alpha,\beta $ and $\mu$ is taken sufficiently large such that 
	\begin{align}\label{41}
		g_T>\bigg(8K_{24}+\frac{8K_{25}}{T}\bigg)^\frac{1}{2}\|g\|_0+\bigg(\frac{3}{\alpha}+ \frac{1}{ \alpha(\alpha-\eta^*)}\bigg)^\frac{1}{2}\|g_t\|_{0}.
	\end{align}
Similarly, for $d=r=3$ with $ \beta \mu > 1$, we can define $M$ by 
\begin{align*}
&M=M(\alpha,\mu)\nonumber\\&:= \frac{\bigg(\frac{8K_{31}}{T^2}+8K_{32}+\frac{8K_{33}}{T}\bigg)^\frac{1}{2}\|\u_0\|_{\H}+\big\|(\boldsymbol{\varphi}\cdot \nabla)\boldsymbol{\varphi}+\nabla \psi-\mu \Delta\boldsymbol{\varphi}+\alpha \boldsymbol{\varphi}+ \beta |\boldsymbol{\varphi}|^{r-1}\boldsymbol{\varphi}\big\|_{\H}} {g_T-\bigg\{\bigg(8K_{34}+\frac{8K_{35}}{T}\bigg)^\frac{1}{2}\|g\|_0+\bigg(\frac{3}{\alpha }+ \frac{1}{ \alpha^2}\bigg)^\frac{1}{2}\|g_t\|_{0}\bigg\}},
\end{align*}
where $K_{3i}, \ i=1,\ldots,5,$  are defined in Lemma \ref{lemmae3} (see \eqref{325}),  and $\alpha,\beta $ and $\mu$ are taken sufficiently large such that
\begin{align*}
	g_{T}> \bigg(8K_{34}+\frac{8K_{35}}{T}\bigg)^\frac{1}{2}\|g\|_0+\bigg(\frac{3}{\alpha}+ \frac{1}{ \alpha^2}\bigg)^\frac{1}{2}\|g_t\|_{0}.
\end{align*} 
It is worth emphasizing that $M$ is large if $\alpha, \beta$ and $\mu$ are  large, which indicates that the unknown part $\f(\cdot)$ of the source could be large. A similar choice of $M$ has to be made in the case of $d=2$ and $r\in[1,3]$ also (cf. \eqref{4.7} below). 
\end{remark}
\begin{lemma}\label{lemma4}
		Let $\u_0 \in \H, \ \boldsymbol{\varphi} \in \H^2(\T) \cap \V $, $\nabla \psi \in \G(\T)$, $g, g_t \in \C(\T \times [0,T])$ satisfy assumption \eqref{1e}. If $\alpha,\beta$ and $\mu$ are  sufficiently large so that  the conditions given in Remark \ref{rem4.2} (cf. \eqref{41}) are  satisfied, then the operator $\mathcal{B}$ maps $\mathcal{D}$ into itself. 
\end{lemma}
\begin{proof}
	Let $\f \in \mathcal{D}$. From \eqref{1i}, we have
	\begin{align}\label{B111}
		\| \mathcal{B}\f\|_{\H}\leq \frac{1}{g_T}\bigg\{\|\u_t(T)\|_{\H}+\|(\boldsymbol{\varphi}\cdot \nabla)\boldsymbol{\varphi}+\nabla \psi-\mu \Delta\boldsymbol{\varphi}+\alpha \boldsymbol{\varphi}+ \beta |\boldsymbol{\varphi}|^{r-1}\boldsymbol{\varphi}\|_{\H}\bigg\}.
	\end{align}
\vskip 0.2 cm 
\noindent\textbf{Case I:} \emph{$d=2$ and $r \in [1,3]$.} At the final time $t=T$, from  \eqref{n45}, we obtain
\begin{align}\label{B114}
	\|\u_t(T)\|_{\H}&\leq\bigg[ \bigg(\frac{K_{11}}{ T}+K_{12}\bigg)\|\u_0\|_{\H}^2 +\bigg(\frac{3T}{4}+K_{13}\bigg)\|g\|_0^2\|\f\|_{\H}^2 \no \\& \quad+C \bigg\{\|\u_0\|_{\H}+\frac{1}{\alpha}\|g\|_0\|\f\|_{\H}\bigg\}^\frac{2}{3}
	\no \\&\quad\times
	\bigg\{ \frac{4}{\mu } \left\{ \left(\frac{1}{T}+\frac{\alpha}{8}\right)\|\u_0\|_{\H}^2+\frac{1}{\alpha}\|g\|_0^2\|\f\|_{\H}^2\right\} +\frac{1}{2\mu\alpha}\|g\|_0^2\|\f\|_{\H}^2\bigg\}^\frac{4}{3}\no \\&\quad\times \bigg\{ \frac{4}{\mu^2 } \left\{ \left(\frac{1}{T}+\frac{\alpha}{8}\right)\|\u_0\|_{\H}^2+\frac{1}{\alpha}\|g\|_0^2\|\f\|_{\H}^2\right\}+\frac{3T}{8\mu^2}\|g\|_0^2\|\f\|_{\H}^2\bigg\}\bigg]^\frac{1}{2}\no \\& \quad\times\bigg[\frac{8}{T}+ \frac{8}{\mu^2 } \left\{ \left(\frac{1}{T}+\frac{\alpha}{8}\right)\|\u_0\|_{\H}^2+\frac{1}{ \alpha}\|g\|_0^2\|\f\|_{\H}^2\right\}+\frac{1}{\mu^2\alpha}\|g\|_0^2\|\f\|_{\H}^2\bigg]^\frac{1}{2}\no \\& \quad+ \sqrt{\frac{T}{ \alpha}}\|g_t\|_{0}\|\f\|_{\H},
\end{align}
where $K_{1i}, \ i=1,2,3,$  are defined in Lemma \ref{lemmae3} (see \eqref{321}).
 Plugging the relation \eqref{B114} in \eqref{B111} gives
\begin{align}\label{4.7}
	\|\mathcal{B} \f\|_{\H}&\leq \frac{1}{g_T} \bigg[ \bigg[ \bigg(\frac{K_{11}}{ T}+K_{12}\bigg)\|\u_0\|_{\H}^2 +\bigg(\frac{3T}{4}+K_{13}\bigg)\|g\|_0^2\|\f\|_{\H}^2 \no \\& \quad+\bigg\{\|\u_0\|_{\H}+\frac{1}{\alpha}\|g\|_0\|\f\|_{\H}\bigg\}^\frac{2}{3}
	\no \\&\quad\times
	\bigg\{ \frac{4}{\mu } \left\{ \left(\frac{1}{T}+\frac{\alpha}{8}\right)\|\u_0\|_{\H}^2+\frac{1}{\alpha}\|g\|_0^2\|\f\|_{\H}^2\right\} +\frac{1}{2\mu\alpha}\|g\|_0^2\|\f\|_{\H}^2\bigg\}^\frac{4}{3}\no \\&\quad\times \bigg\{ \frac{4}{\mu^2 } \left\{ \left(\frac{1}{T}+\frac{\alpha}{8}\right)\|\u_0\|_{\H}^2+\frac{1}{\alpha}\|g\|_0^2\|\f\|_{\H}^2\right\}+\frac{3T}{8\mu^2}\|g\|_0^2\|\f\|_{\H}^2\bigg\}\bigg]^\frac{1}{2}\no \\& \quad\times\bigg[\frac{8}{T}+ \frac{8}{\mu^2 } \left\{ \left(\frac{1}{T}+\frac{\alpha}{8}\right)\|\u_0\|_{\H}^2+\frac{1}{ \alpha}\|g\|_0^2\|\f\|_{\H}^2\right\}+\frac{1}{\mu^2\alpha}\|g\|_0^2\|\f\|_{\H}^2\bigg]^\frac{1}{2}\no \\& \quad+ \sqrt{\frac{T}{ \alpha}}\|g_t\|_{0}\|\f\|_{\H}+\big\|(\boldsymbol{\varphi}\cdot \nabla)\boldsymbol{\varphi}+\nabla \psi-\mu \Delta\boldsymbol{\varphi}+\alpha \boldsymbol{\varphi}+ \beta |\boldsymbol{\varphi}|^{r-1}\boldsymbol{\varphi}\big\|_{\H}\bigg] \leq M.
\end{align}
	\vskip 0.2 cm 
\noindent\textbf{Case II:} \emph{$d=2,3$ and $r >3$.} 
	At the final time $t=T$, from  \eqref{n46}, we deduce 
\begin{align}\label{B112}
	\|\u_t(T)\|_{\H}&\leq   \bigg(\frac{8K_{21}}{T^2}+8K_{22}+\frac{8K_{23}}{T}\bigg)^\frac{1}{2}\|\u_0\|_{\H}+\bigg(8K_{24}+\frac{8K_{25}}{T}\bigg)^\frac{1}{2}\|g\|_0\|\f\|_{\H}\no \\&\quad+\bigg(\frac{3}{\alpha }+ \frac{1}{ \alpha(\alpha-\eta^*)}\bigg)^\frac{1}{2}\|g_t\|_{0}\|\f\|_{\H},
\end{align}
where $K_{2i}, \ i=1,\ldots,5,$  are defined in Lemma \ref{lemmae3} (see \eqref{322}) and  $ \gamma=\frac{r-3}{r-1}\bigg(\frac{2}{\beta(r-1)}\bigg)^\frac{2}{r-3}.$
Substituting \eqref{B112} in \eqref{B111}, we infer 
\begin{align}\label{m1}
	\|\mathcal{B} \f\|_{\H}&\leq \frac{1}{g_T} \bigg[ \bigg(\frac{8K_{21}}{T^2}+8K_{22}+\frac{8K_{23}}{T}\bigg)^\frac{1}{2}\|\u_0\|_{\H}+\bigg(8K_{24}+\frac{8K_{25}}{T}\bigg)^\frac{1}{2}\|g\|_0\|\f\|_{\H}\no \\&\quad+\bigg(\frac{3}{\alpha}+ \frac{1}{ \alpha(\alpha-\eta^*)}\bigg)^\frac{1}{2}\|g_t\|_{0}\|\f\|_{\H} \no \\&\quad+\big\|(\boldsymbol{\varphi}\cdot \nabla)\boldsymbol{\varphi}+\nabla \psi-\mu \Delta\boldsymbol{\varphi}+\alpha \boldsymbol{\varphi}+ \beta |\boldsymbol{\varphi}|^{r-1}\boldsymbol{\varphi}\big\|_{\H}\bigg] \leq M.
\end{align}
\vskip 0.2 cm 
\noindent\textbf{Case III:} \emph{$d=r=3$ with $\beta \mu >1$.} 
At the final time $t=T$ in \eqref{n47}, from \eqref{B111}, we deduce 
\begin{align}\label{B113}
\|\mathcal{B} \f\|_{\H}&\leq \frac{1}{g_T} \bigg[ \bigg(\frac{8K_{31}}{T^2}+8K_{32}+\frac{8K_{33}}{T}\bigg)^\frac{1}{2}\|\u_0\|_{\H}+\bigg(8K_{34}+\frac{8K_{35}}{T}\bigg)^\frac{1}{2}\|g\|_0\|\f\|_{\H}\no \\&\quad+\bigg(\frac{3}{\alpha }+ \frac{1}{ \alpha^2}\bigg)^\frac{1}{2}\|g_t\|_{0}\|\f\|_{\H} +\big\|(\boldsymbol{\varphi}\cdot \nabla)\boldsymbol{\varphi}+\nabla \psi-\mu \Delta\boldsymbol{\varphi}+\alpha \boldsymbol{\varphi}+ \beta |\boldsymbol{\varphi}|^{r-1}\boldsymbol{\varphi}\big\|_{\H}\bigg] \no\\& \leq M,
\end{align}
where $K_{3i}, \ i=1,\ldots, 5,$  are defined in Lemma \ref{lemmae3} (see \eqref{325}) and this proves Lemma \ref{lemma4}.
\end{proof}
	The following lemma proves that there is a solution to the inverse problem \eqref{1a}-\eqref{1d}.
\begin{lemma}\label{lem4.4}
	Let the assumptions of Lemma \ref{lemma4} hold true. Then, $\mathcal{B}$ is weakly continuous from $\mathcal{D}$ into itself. 
\end{lemma}
\begin{proof}
	It is sufficient to prove that the nonlinear operator $\mathcal{A}: \mathcal{D} \to \H$ is weakly continuous. Let $\{\u_k\}_{k \in \N}$ be the sequences of the solutions to the direct problem \eqref{1a}-\eqref{2a} corresponding to the external forcing $\{\f_kg\}_{k \in \N}$. At the final time $t=T$, from \eqref{1a}, we have
	\begin{align}\label{A}
		(\u_k)_t( T)=\boldsymbol{f}_kg(T)+\mu \Delta \boldsymbol{\varphi}-(\boldsymbol{\varphi}\cdot\nabla)\boldsymbol{\varphi}-\alpha \boldsymbol{\varphi}- \beta|\boldsymbol{\varphi}|^{r-1}\boldsymbol{\varphi}-\nabla \psi.
	\end{align}
	  Let the sequence $\{\f_k\}_{k \in \N}$  in $\mathcal{D}$ be such that  $	\f_k \rightharpoonup  \f \text{ weakly in } \H$ (that is, $(\f_k-\f, \w) \rightarrow 0 \ \ \text{in} \ \  \H$-norm for all $\w \in \H$).
We know that the spaces $\H, \V$ and $\wi{\L}^{r+1}$ are reflexive.
Making use of the energy estimates (cf. \eqref{bb1}, \eqref{bb17}) and the Banach-Alaoglu theorem, we can extract subsequences $\{\u_{k_l}\}$ of $\{\u_k\}$,  such that (for simplicity, we denote the index $k_l$ by $k$) $ \u_k \rightharpoonup \u\textrm{ weakly in
		} \H,$  for all $t\in[0,T]$. From equation \eqref{A}, for all $\w \in \H$, we have
		\begin{align*}
	&\big|((\u_k)_t(T), \w)| \\&\leq \bigg\{\|g\|_{0}\|\f_k\|_{\H}+\big\|(\boldsymbol{\varphi}\cdot \nabla)\boldsymbol{\varphi}+\nabla \psi-\mu \Delta\boldsymbol{\varphi}+\alpha \boldsymbol{\varphi}+ \beta |\boldsymbol{\varphi}|^{r-1}\boldsymbol{\varphi}\big\|_{\H}\bigg\}\|\w\|_{\H} <\infty.
\end{align*}
Thus, it is immediate that (cf. \eqref{bb22}  and Remark \ref{remb7} below)
\begin{align*}
	(\u_k)_t(T) \rightharpoonup \u_t(T)\textrm{ weakly in
	} \H.
\end{align*}
On passing the limit $k \rightarrow \infty$ in \eqref{A}, we have
\begin{align*}
	\mathcal{A} \f_k=(\u_k)_t(T) \rightharpoonup \u_t(T)=\mathcal{A} \f \textrm{ weakly in } \H.
\end{align*}
So, the operator $\mathcal{A}$ is weakly continuous. Thus, the operator $\mathcal{B}$ is weakly continuous. 
\end{proof}

\begin{proof}[Proof of Theorem \ref{thm2} $(i)$]
From Lemma \ref{lem4.4}, we know that  the operator $\mathcal{B}$ is weakly continuous.	Hence, $\mathcal{B}$ has a fixed point in $\mathcal{D}$. From Theorem \ref{thm1}, we infer that the inverse problem \eqref{1a}-\eqref{1d} has a solution.
	\end{proof}
\begin{example}\label{ex4.5}
	For $g\equiv 1$, $\u_0=\mathbf{0}$, and
	\begin{align*} 
		&\big\|(\boldsymbol{\varphi}\cdot \nabla)\boldsymbol{\varphi}+\nabla \psi-\mu \Delta\boldsymbol{\varphi}+\alpha \boldsymbol{\varphi}+ \beta |\boldsymbol{\varphi}|^{r-1}\boldsymbol{\varphi}\big\|_{\H} <(1-8K_{24})^\frac{1}{2}M,
	\end{align*}
 with $8K_{24}<1$,  that is, $\frac{1}{\alpha}+\frac{\gamma}{\mu\alpha}+\frac{1}{ \mu}<\frac{1}{6}$, then
 the condition \eqref{m1} becomes $$\frac{1}{ T} \leq \frac{\big(1- 8K_{24}\big)M^2-\big\|(\boldsymbol{\varphi}\cdot \nabla)\boldsymbol{\varphi}+\nabla \psi-\mu \Delta\boldsymbol{\varphi}+\alpha \boldsymbol{\varphi}+ \beta |\boldsymbol{\varphi}|^{r-1}\boldsymbol{\varphi}\big\|_{\H}^2}{8M^2K_{25}} ,$$
	 where $\ K_{24}$ and $K_{25}$ are defined in Lemma \ref{lemmae3} (see \eqref{322}). 
	 Therefore,  for any 
	\begin{align*} 
		T\geq \frac{8M^2K_{25}}{\big(1- 8K_{24}\big)M^2-\big\|(\boldsymbol{\varphi}\cdot \nabla)\boldsymbol{\varphi}+\nabla \psi-\mu \Delta\boldsymbol{\varphi}+\alpha \boldsymbol{\varphi}+ \beta |\boldsymbol{\varphi}|^{r-1}\boldsymbol{\varphi}\big\|_{\H}^2},
	\end{align*}
and under the condition \eqref{1g2} for $d=2,3$ and $r\in(3,\infty)$,  there exists a solution of the inverse problem \eqref{1a}-\eqref{1d}, for $d=2,3$ and $r\in (3,\infty)$.

Similarly, for $d=r=3$,  $g\equiv 1$, $\u_0=\mathbf{0}$, and
\begin{align*}
	\big\|(\boldsymbol{\varphi}\cdot \nabla)\boldsymbol{\varphi}+\nabla \psi-\mu \Delta\boldsymbol{\varphi}+\alpha \boldsymbol{\varphi}+ \beta |\boldsymbol{\varphi}|^{r-1}\boldsymbol{\varphi}\big\|_{\H}<(1-8K_{34})^\frac{1}{2}M, 
\end{align*}
 with $8K_{34}<1$, that is, $\frac{1}{ \alpha}+\frac{1}{2(\beta \mu-1)}<\frac{1}{6}$,  the condition \eqref{B113} becomes   $$\frac{1}{ T} \leq \frac{\big(1- 8K_{34}\big)M^2-\big\|(\boldsymbol{\varphi}\cdot \nabla)\boldsymbol{\varphi}+\nabla \psi-\mu \Delta\boldsymbol{\varphi}+\alpha \boldsymbol{\varphi}+ \beta |\boldsymbol{\varphi}|^{r-1}\boldsymbol{\varphi}\big\|_{\H}^2}{8M^2K_{35}} $$
 where $ K_{34}$ and $K_{35}$ are defined in Lemma \ref{lemmae3} (see \eqref{325}). 
 Therefore,  for any 
 \begin{align*} 
 T\geq \frac{8M^2K_{35}}{\big(1- 8K_{34}\big)M^2-\big\|(\boldsymbol{\varphi}\cdot \nabla)\boldsymbol{\varphi}+\nabla \psi-\mu \Delta\boldsymbol{\varphi}+\alpha \boldsymbol{\varphi}+ \beta |\boldsymbol{\varphi}|^{r-1}\boldsymbol{\varphi}\big\|_{\H}^2},
 \end{align*}
  there exists a solution of the inverse problem \eqref{1a}-\eqref{1d}, for $d=r=3$.
\end{example}

\section{Proof of Theorem \ref{thm2} $(ii)$}\label{sec4b}\setcounter{equation}{0}
%	\subsection{Uniqueness and stability}
	In the previous section, we have proved the existence of a solution  $\{\u,\nabla p,\f\}$ to the inverse problem \eqref{1a}-\eqref{1d}. To obtain results on the uniqueness and stability, we first provide some supporting lemmas. Let $\{\u_i,\nabla p_i,\f_i\}$ $(i=1,2)$ be the solutions of the inverse problem \eqref{1a}-\eqref{1d} corresponding to the given data $\{\u_{0i},\boldsymbol{\varphi}_i,\nabla\psi_i,g_i\}$ $(i=1,2)$ and set 
\begin{align*}
	& \ \ \ \  \ \ \  \ \ \ \u:=\u_1-\u_2, \ \ \ \ \ \ \ \nabla p:=\nabla (p_1-p_2), \ \ \ \  \ \ \f:=\f_1-\f_2, \\& \u_0:=\u_{01}-\u_{02}, \ \ \ \boldsymbol{\boldsymbol{\varphi}}:=\boldsymbol{\varphi}_1-\boldsymbol{\varphi}_2, \ \ \ \nabla\psi:=\nabla(\psi_1-\psi_2), \ \ \  g:=g_1-g_2.
\end{align*}

Later discussions of the uniqueness and stability of the solution to the inverse problem \eqref{1a}-\eqref{1d} are based on the paper \cite{JF}, in which the authors demonstrate  the uniqueness and stability of the solution to the inverse problem for NSE in two dimensions.
	
%The stability of the velocity $\u(\cdot)$ of the solution to the inverse problem  \eqref{1a}-\eqref{1d} is established in the following lemma.
 In this section, we assume that $\alpha,\beta$ and $\mu$ are  sufficiently large so that  the conditions given in Remark \ref{rem4.2} (cf. \eqref{41}) are  satisfied.

\begin{lemma}\label{lemmaS1}
	Let	$\u_{0i} \in \H$, $g_i, (g_i)_t \in \C(\T\times [0,T])$ satisfy assumption \eqref{1e} and $\f _i\in \H$, for $i=1,2$. For $r \geq 1, \ \text{in}\ 2D$, and for $ r \geq 3, \ \text{in} \ 3D$ ($\beta\mu\geq 1$ for $r=3$), the following estimate holds:
	\begin{align}\label{S1}
		&\sup_{t\in[0,T]}\|\u(t)\|_{\H}^2+\mu \int_0^T\|\nabla\u(t)\|_{\H}^2\d t+\alpha \int_0^T\|\u(t)\|_{\H}^2\d t +\frac{\beta}{2^{r-1}}\int_0^T\|\u(t)\|_{\widetilde{\L}^{r+1}}^{r+1}\d t \nonumber\\&\leq C \big\{\|\u_0\|_{\H}^2+ \|\f\|_{\H}^2+\|g\|_0^2\big\}.
	\end{align}
Moreover, there exists a time $t_3 \in \big(\frac{3T}{8},\frac{4T}{8}\big)$ such that
\begin{align}\label{S2}
\| \nabla \u(t_3)\|_{\H}^2 	\leq  C \big\{\|\u_0\|_{\H}^2+ \|\f\|_{\H}^2+\|g\|_0^2\big\}
\end{align}
holds and $C$ depends on the input data,  $\mu,\alpha,\beta,r$ and $T$.
\end{lemma}
\begin{proof}
	Subtracting the equations for $\{\u_i,\nabla p_i, \f_i\} \ (i=1,2)$, we find 
	\begin{align}\label{S3}
		&\u_t-\mu \Delta \u+(\u_1 \cdot \nabla)\u\nonumber+(\u \cdot \nabla)\u_2+\alpha \u\\&\quad+\beta \left(|\u_1|^{r-1}\u_1-|\u_2|^{r-1}\u_2\right)+\nabla p=\f g_1+\f_2 g,
	\end{align}
for a.e. $t\in[0,T]$ in $\H$. Multiplying both sides of \eqref{S3} by $\u(\cdot)$ and integrating the resulting equation over $\T$, we obtain
	\begin{align}\label{S4}
		&\frac{1}{2}\frac{\d}{\d t}\|\u(t)\|_{\H}^2+\mu \|\nabla \u(t)\|_{\H}^2+\alpha\|\u(t)\|_{\H}^2\nonumber\\&=\big((\f g_1(t)+\f_2 g(t),\u(t)\big) -\big((\u(t) \cdot \nabla)\u_2(t),\u(t)\big)\nonumber\\&\quad-\beta \left(|\u_1(t)|^{r-1}\u_1(t)-|\u_2(t)|^{r-1}\u_2(t),\u(t)\right).
	\end{align}
 We estimate $|(\f g_1+\f_2 g,\u)|$ using  H\"older's and Young's inequalities as
	\begin{align}\label{S5}
		|(\f g_1+\f_2 g,\u)|\leq\frac{1}{2 \alpha}\big(\|\f\|_{\H}\|g_1\|_{\mathrm{L}^\infty}+\|\f_2\|_{\H}\|g\|_{\mathrm{L}^\infty}\big)^2+\frac{\alpha}{2}\|\u\|_{\H}^2.
	\end{align}
For $r\geq 1$, we have (see Sec. 2.4, \cite{MTM4,MTM6})
\begin{align}\label{C3}
	&\beta\left(|\u_1|^{r-1}\u_1-|\u_2|^{r-1}\u_2,\u_1-\u_2\right)\no\\&\geq \frac{\beta}{2}\big\||\u_1|^{\frac{r-1}{2}}|\u_1-\u_2|\big\|_{\H}^2+\frac{\beta}{2}\big\||\u_2|^{\frac{r-1}{2}}|\u_1-\u_2|\big\|_{\H}^2.
\end{align}
 It is important to note that
\begin{align*}
	\|\u_1-\u_2\|_{\widetilde{\L}^{r+1}}^{r+1}&=\int_{\T}|\u_1(x)-\u_2(x)|^{r-1} |\u_1(x)-\u_2(x)|^{2} \d x \nonumber\\&\leq 2^{r-2} \int_{\T}\left(|\u_1(x)|^{r-1}+|\u_2(x)|^{r-1}\right) |\u_1(x)-\u_2(x)|^{2} \d x
	\nonumber\\& \leq  2^{r-2}\left( \big\||\u_1|^{\frac{r-1}{2}}|\u_1-\u_2|\big\|_{\H}^2+\big\||\u_2|^{\frac{r-1}{2}}|\u_1-\u_2|\big\|_{\H}^2\right).
\end{align*}
From the above inequality, we have
\begin{align}\label{C41}
	\frac{2^{2-r}\beta}{4}\|\u_1-\u_2\|_{\widetilde{\L}^{r+1}}^{r+1}\leq \frac{\beta}{4}\||\u_1|^{\frac{r-1}{2}}|\u_1-\u_2|\|_{\H}^2+\frac{\beta}{4}\||\u_2|^{\frac{r-1}{2}}|\u_1-\u_2|\|_{\H}^2.
\end{align}
Gathering both the estimates \eqref{C3} and \eqref{C41}, we obtain
\begin{align}\label{C611}
	&\beta\left(|\u_1|^{r-1}\u_1-|\u_2|^{r-1}\u_2,\u_1-\u_2\right) \geq \frac{\beta}{2^r}\|\u_1-\u_2\|_{\widetilde{\L}^{r+1}}^{r+1}.
	\end{align}
	\vskip 0.2 cm 
\noindent\textbf{Case I:} \emph{$d=2$ and $r \in [1,3]$.} 
Applying H\"older's, Ladyzhenskaya's and Young's inequalities, one gets
\begin{align}\label{S6}
	|((\u \cdot \nabla)\u_2,\u)| &\leq\|\u\|_{\widetilde{\L}^{4}}^2\|\nabla\u_2\|_{\H}\leq\sqrt{2}\|\u\|_{\H}\|\nabla\u\|_{\H}\|\nabla\u_2\|_{\H}\no \\&\leq\frac{\mu}{2}\|\nabla \u\|_{\H}^2+\frac{1}{\mu}\|\u\|_{\H}^2\|\nabla\u_2\|_{\H}^2.
\end{align}
Plugging the estimates \eqref{S5}, \eqref{C611} and \eqref{S6} in \eqref{S4}, and integrating the resulting estimate from $0$ to $t$, we obtain
\begin{align}\label{S8}
	&\|\u(t)\|_{\H}^2+\mu \int_0^t\|\nabla\u(s)\|_{\H}^2\d  s+\alpha\int_0^t\|\u(s)\|_{\H}^2\d s +\frac{\beta}{2^{r-1}}\int_0^t\|\u(s)\|_{\widetilde{\L}^{r+1}}^{r+1}\d s \nonumber\\& \leq \|\u_0\|_{\H}^2+ \frac{t}{\alpha}\big(\|\f\|_{\H}\|g_1\|_0+\|\f_2\|_{\H}\|g\|_0\big)^2+\frac{2}{\mu}\int_0^t\|\nabla\u_2(s)\|_{\H}^2 \|\u(s)\|_{\H}^2 \d s,
\end{align} 
for all $t \in [0,T]$. An application of Gronwall's inequality in \eqref{S8} gives 
\begin{align*}
	\|\u(t)\|_{\H}^2\leq \exp \left(\frac{2}{\mu}\int_0^T\|\nabla\u_2(t)\|_{\H}^2\d t\right)\left\{\|\u_0\|_{\H}^2+ \frac{T}{\alpha}\big(\|\f\|_{\H}\|g_1\|_0+\|\f_2\|_{\H}\|g\|_0\big)^2\right\},
\end{align*}
for all $t \in [0,T]$. Thus, from \eqref{S8}, it is immediate that
\begin{align*}
	&\sup_{t\in[0,T]}\|\u(t)\|_{\H}^2+\mu \int_0^T\|\nabla\u(t)\|_{\H}^2\d t +\alpha \int_0^T\|\u(t)\|_{\H}^2\d t +\frac{\beta}{2^{r-1}}\int_0^T\|\u(t)\|_{\widetilde{\L}^{r+1}}^{r+1}\d t  \no \\&\leq \exp \left\{\frac{1}{\mu^2}\left(\|\u_{02}\|_{\H}^2+\frac{T}{\alpha}\|\f_2\|_{\H}^2\|g_2\|_{0}^2\right)\right\}\left\{\|\u_0\|_{\H}^2+ \frac{T}{\alpha}\big(\|\f\|_{\H}\|g_1\|_0+\|\f_2\|_{\H}\|g\|_0\big)^2\right\} \\ &
	\leq C\big(\mu,\alpha, \beta,\|\u_{02}\|_{\H},\|\f_2\|_{\H},\|g_1\|_{0},\|g_2\|_{0},T\big)\left\{\|\u_0\|_{\H}^2+ \|\f\|_{\H}^2+\|g\|_0^2\right\},
\end{align*} 
and \eqref{S1} follows.
	\vskip 0.2 cm 
\noindent\textbf{Case II:} \emph{$d=2,3$ and $r >3$.} 
 An estimate similar to \eqref{n22} yields 
\begin{align}\label{C4}
	|\left((\u \cdot \nabla)\u_2,\u \right)|&\leq
	\|\nabla\u\|_{\H}\||\u||\u_2|\|_{\H} \leq \frac{\mu }{2}\|\nabla\u\|_{\H}^2+\frac{1}{2 \mu}\||\u||\u_2|\|_{\H}^2
\no \\&	\leq \frac{\mu }{2}\|\nabla\u\|_{\H}^2+\frac{\beta}{4}\big\||\u_2|^{\frac{r-1}{2}}|\u|\big\|_{\H}^2+\frac{r-3}{2\mu(r-1)}\left(\frac{4}{\beta\mu (r-1)}\right)^{\frac{2}{r-3}}\|\u\|_{\H}^2,
\end{align}
for $r>3$. Combining \eqref{C3}, \eqref{C41} and \eqref{C4}, we get
\begin{align}\label{C5}
	&\beta\left(|\u_1|^{r-1}\u_1-|\u_2|^{r-1}\u_2,\u \right) +\left((\u \cdot \nabla)\u_2,\u \right)\nonumber\\&\geq \frac{\beta}{2}\big\||\u_1|^{\frac{r-1}{2}}|\u|\big\|_{\H}^2+\frac{\beta}{4}\big\||\u_2|^{\frac{r-1}{2}}|\u|\big\|_{\H}^2-\frac{r-3}{2\mu(r-1)}\left(\frac{4}{\beta\mu (r-1)}\right)^{\frac{2}{r-3}}\|\u\|_{\H}^2-\frac{\mu }{2}\|\nabla\u\|_{\H}^2 \no\\& \geq
	\frac{\beta}{2^r}\|\u\|_{\widetilde{\L}^{r+1}}^{r+1}
	-\frac{r-3}{2\mu(r-1)}\left(\frac{4}{\beta\mu (r-1)}\right)^{\frac{2}{r-3}}\|\u\|_{\H}^2-\frac{\mu }{2}\|\nabla\u\|_{\H}^2.
\end{align}
Plugging the estimates \eqref{S5} and \eqref{C5} in \eqref{S4}, and integrating it from $0$ to $t$, we deduce 
\begin{align}\label{C6}
	&\|\u(t)\|_{\H}^2+\mu \int_0^t\|\nabla\u(s)\|_{\H}^2\d  s+\alpha\int_0^t\|\u(s)\|_{\H}^2\d s +\frac{\beta}{2^{r-1}}\int_0^t\|\u(s)\|_{\widetilde{\L}^{r+1}}^{r+1}\d s \nonumber\\& \leq \|\u_0\|_{\H}^2+ \frac{t}{\alpha}\big(\|\f\|_{\H}\|g_1\|_0+\|\f_2\|_{\H}\|g\|_0\big)^2\no \\&\quad+\frac{r-3}{\mu(r-1)}\left(\frac{4}{\beta\mu (r-1)}\right)^{\frac{2}{r-3}}\int_0^t \|\u(s)\|_{\H}^2 \d s,
\end{align} 
for all $t \in [0,T]$.	By the virtue of Gronwall's inequality in \eqref{C6}, followed by taking supremum on both sides
over time from $0$ to $T$
\iffalse
\begin{align*}
&\sup_{t\in[0,T]}\|\u(t)\|_{\H}^2+\mu \int_0^T\|\u(t)\|_{\V}^2\d  t+\alpha\int_0^T\|\u(t)\|_{\H}^2\d t +\frac{\beta}{2^{r-1}}\int_0^T\|\u(t)\|_{\widetilde{\L}^{r+1}}^{r+1}\d t \nonumber\\& \leq \bigg\{\|\u_0\|_{\H}^2+ \frac{T}{\alpha}\big(\|\f\|_{\H}\|g_1\|_0+\|\f_2\|_{\L^2}\|g\|_0\big)^2\bigg\} \exp{(\frac{r-3}{2\mu(r-1)}\left(\frac{4}{\beta\mu (r-1)}\right)^{\frac{2}{r-3}} T)},
\end{align*}
\fi
 easily leads to \eqref{S1}. \
\vskip 0.2 cm 
\noindent\textbf{Case III:} \emph{$d=r=3$ with $ \beta \mu \geq 1$.} From \eqref{C3}, we get
	\begin{align}\label{C7}
	\beta\left(|\u_1|^2\u_1-|\u_2|^2\u_2,\u\right)\geq \frac{\beta}{2}\big\||\u_1||\u|\big\|_{\H}^2+\frac{\beta}{2}\big\||\u_2||\u|\big\|_{\H}^2.
\end{align}
We use the Cauchy-Schwarz and Young's inequalities to obtain
\begin{align}\label{C8}
	|\left((\u \cdot \nabla)\u_2,\u \right)|\leq \|\nabla\u\|_{\H}\big\||\u_2||\u|\big\|_{\H} \leq \frac{\mu}{2} \|\nabla\u\|_{\H}^2+\frac{1}{2\mu }\big\||\u_2||\u|\big\|_{\H}^2.
\end{align}
  Gathering \eqref{C7} and \eqref{C8}, we have
 \begin{align}\label{C9}
 	&\beta\left(|\u_1|^2\u_1-|\u_2|^2\u_2,\u\right)+\left((\u \cdot \nabla)\u_2,\u \right)\no\\&\geq \frac{\beta}{2}\big\||\u_1||\u| \big\|_{\H}^2+\bigg(\frac{\beta}{2}-\frac{1}{2\mu }\bigg)\big\||\u_2||\u| \big\|_{\H}^2-\frac{\mu}{2} \|\nabla\u\|_{\H}^2
 	\no\\& \geq \frac{1}{2}\left(\beta-\frac{1}{\mu }\right)\|\u\|_{\widetilde{\L}^4}^4-\frac{\mu}{2} \|\nabla\u\|_{\H}^2.
 \end{align}
	Plugging the estimates \eqref{S5} and \eqref{C9} in \eqref{S4}, and integrating it from $0$ to $t$, we infer  
\begin{align*}
	&\|\u(t)\|_{\H}^2+\mu\int_0^t\|\nabla\u(s)\|_{\H}^2\d  s+\alpha\int_0^t\|\u(s)\|_{\H}^2\d s +\left(\beta-\frac{1}{\mu}\right)\int_0^t\|\u(s)\|_{\widetilde{\L}^{4}}^{4}\d s \\&\leq \|\u_0\|_{\H}^2+ \frac{t}{\alpha}\big(\|\f\|_{\H}\|g_1\|_0+\|\f_2\|_{\H}\|g\|_0\big)^2,
\end{align*} 
and \eqref{S1} follows, provided $\beta \mu >1$.\

Since $\u \in \C^1([T/8,T];\V)$ (see Remark \ref{remb7}), by the mean value theorem, there exists a time $t_3 \in \big(\frac{3T}{8},\frac{4T}{8}\big)$ such that
\begin{align*}
	\|\nabla\u(t_3)\|_{\H}^2=\frac{1}{\frac{4T}{8}-\frac{3T}{8}}\int_{3T/8}^{4T/8}	\|\nabla\u(t)\|_{\H}^2 \d t \leq \frac{8}{T} \int_0^T \|\nabla\u(t)\|_{\H}^2 \d t,
\end{align*}
which leads to \eqref{S2} from \eqref{S1} and it completes the proof.
\end{proof}

	\begin{lemma}\label{3.5}
	Let	the assumptions of Lemma \ref{lemmaS1} hold true and let $t_3 \in \big(\frac{3T}{8},\frac{4T}{8}\big)$ be the same time obtained in Lemma \ref{lemmaS1}. For $r \geq 1 $, in 2D, and for $ r \geq 3$, in 3D, we have  
	\begin{align}\label{s1}
		&\sup_{t\in[t_3,T]}\|\nabla\u(t)\|_{\H}^2+\int_{t_3}^T\|\u_t(t)\|_{\H}^2\d t +\mu^2 \int_{t_3}^T\|\Delta\u(t)\|_{\H}^{2} \d t \no\\&\leq  C\big\{\|\u_0\|_{\H}^2+ \|\f\|_{\H}^2+\|g\|_0^2\big\}.
	\end{align}
Moreover, there exists a time $t_4 \in \big(\frac{4T}{8},\frac{5T}{8}\big)$ such that
\begin{align}\label{s2}
	\|\u_t(t_4)\|_{\H}^2 \leq C
\big\{\|\u_0\|_{\H}^2+ \|\f\|_{\H}^2+\|g\|_0^2\big\}
\end{align}
 holds and $C$ depends on the input data,  $\mu,\alpha,\beta,r$ and $T$.
\end{lemma}
\begin{proof}
	Taking the inner product with $\u_t(\cdot)-\mu \Delta \u(\cdot)$ in the equation \eqref{S3} and integrating over $\T$, we obtain
	\begin{align}\label{s3}
		&\|\u_t(t)\|_{\H}^2+ \mu\frac{\d}{\d t} \|\nabla\u(t)\|_{\H}^2+\mu^2\|\Delta\u(t)\|_{\H}^2+\frac{\alpha}{2}\frac{\d}{\d t} \|\u(t)\|_{\H}^2+\alpha \mu \|\nabla\u(t)\|_{\H}^2\no\\&=\big(\f g_1(t)+\f_2 g(t),\u_t(t)\big) + \big(\f g_1(t)+\f_2 g(t),-\mu\Delta\u(t)\big)\nonumber\\&\quad-\big((\u(t) \cdot \nabla)\u_2(t),\u_t(t)\big)-\big((\u(t) \cdot \nabla)\u_2(t),-\mu\Delta\u(t)\big)\nonumber\\&\quad-\big((\u_1(t) \cdot \nabla)\u(t),\u_t(t)\big)-\big((\u_1(t) \cdot \nabla)\u(t),-\mu\Delta\u(t)\big)\nonumber\\&\quad-\beta \left(|\u_1(t)|^{r-1}\u_1(t)-|\u_2(t)|^{r-1}\u_2(t),\u_t(t)\right)\nonumber\\&\quad-\beta   \left(|\u_1(t)|^{r-1}\u_1(t)-|\u_2(t)|^{r-1}\u_2(t),-\mu\Delta\u(t) \right) =:\sum_{j=1}^8I_j,
	\end{align}
	where we have used the fact $\langle\nabla p,  \u_t\rangle=0$ and $\langle\nabla p, \Delta \u\rangle=0$. Next, we estimate each $I_j$'s $(j=1,\ldots,8)$ separately as follows:
	Making use of the Cauchy-Schwarz and Young's inequalities, we estimate $I_1$ and $I_2$ as
	\begin{align}
		I_1&\leq2\big(\|\f\|_{\H}\|g_1\|_{\mathrm{L}^\infty}+\|\f_2\|_{\H}\|g\|_{\mathrm{L}^\infty}\big)^2+\frac{1}{8}\|\u_t\|_{\H}^2, \label{s4}\\
		I_2&\leq3\big(\|\f\|_{\H}\|g_1\|_{\mathrm{L}^\infty}+\|\f_2\|_{\H}\|g\|_{\mathrm{L}^\infty}\big)^2+\frac{\mu^2}{12}\|\Delta\u\|_{\H}^2.\label{s5}
\end{align}
\vskip 0.2 cm
\noindent\textbf{Case I:} \emph{$d=2$ and $ r \geq 1$.}
 We apply H\"older's, Gagliardo-Nirenberg's and Young's inequalities to estimate $I_3$ and $I_4$ as
	 \begin{align}
	 	I_3 &\leq \|\u\|_{\widetilde{\L}^{3}} \|\nabla\u_2\|_{\widetilde{\L}^{6}}\|\u_t\|_{\H}\leq C\|\u\|_{\H}^\frac{2}{3}\|\nabla\u\|_{\H}^\frac{1}{3}\|\Delta\u_2\|_{\H}\|\u_t\|_{\H}\no\\&\leq \frac{1}{8}\|\u_t\|_{\H}^2+C \|\u\|_{\H}^\frac{4}{3}\|\nabla\u\|_{\H}^\frac{2}{3} \|\Delta\u_2\|_{\H}^2, \label{s6}\\
	 		I_4 &\leq \mu \|\u\|_{\widetilde{\L}^{3}} \|\nabla\u_2\|_{\widetilde{\L}^{6}}\|\Delta\u\|_{\H}\leq C\mu\|\u\|_{\H}^\frac{2}{3}\|\nabla\u\|_{\H}^\frac{1}{3}\|\Delta\u_2\|_{\H}\|\Delta\u\|_{\H}\no\\& \leq
	 	\frac{\mu^2}{12}	\|\Delta\u\|_{\H}^2+C \|\u\|_{\H}^\frac{4}{3}\|\nabla\u\|_{\H}^\frac{2}{3} \|\Delta\u_2\|_{\H}^2.\label{s7}
	 \end{align}
Similarly, we estimate $I_5$ and $I_6$ as
	\begin{align}
		I_5 &\leq \|\u_1\|_{\widetilde{\L}^{4}} \|\nabla\u\|_{\widetilde{\L}^{4}}\|\u_t\|_{\H}\leq C\|\u_1\|_{\H}^\frac{1}{2}\|\nabla\u_1\|_{\H}^\frac{1}{2}\|\u\|_{\H}^\frac{1}{4}\|\u\|_{\H_p^2}^\frac{3}{4} \|\u_t\|_{\H}\nonumber\\&\leq \frac{1}{8}\|\u_t\|_{\H}^2+C \|\u_1\|_{\H}\|\nabla\u_1\|_{\H}\|\u\|_{\H}^\frac{1}{2}\left(\|\u\|_{\H}^\frac{3}{2}+\|\Delta\u\|_{\H}^\frac{3}{2}\right)
		\nonumber\\&\leq \frac{1}{8}\|\u_t\|_{\H}^2+\frac{\mu^2}{12}\|\Delta\u\|_{\H}^{2}+C\|\u_1\|_{\H}\|\nabla\u_1\|_{\H}\|\u\|_{\H}^2+C\|\u_1\|_{\H}^4\|\nabla\u_1\|_{\H}^4\|\u\|_{\H}^2, \label{s8}\\
		I_6 &\leq \mu \|\u_1\|_{\widetilde{\L}^{4}} \|\nabla\u\|_{\widetilde{\L}^{4}}\|\Delta\u\|_{\H}\leq C\mu\|\u_1\|_{\H}^\frac{1}{2}\|\nabla\u_1\|_{\H}^\frac{1}{2}\|\u\|_{\H}^\frac{1}{4}\left(\|\u\|_{\H}^\frac{3}{4}+\|\Delta\u\|_{\H}^\frac{3}{4}\right)\|\Delta\u\|_{\H}\no\\& \leq
	\frac{\mu^2}{12}	\|\Delta\u\|_{\H}^2+C \|\u_1\|_{\H}\|\nabla\u_1\|_{\H}\|\u\|_{\H}^\frac{1}{2}\left(\|\u\|_{\H}^\frac{3}{2}+\|\Delta\u\|_{\H}^\frac{3}{2}\right)
	\no \\& \leq 	\frac{\mu^2}{6}	\|\Delta\u\|_{\H}^2+C\|\u_1\|_{\H}\|\nabla\u_1\|_{\H}\|\u\|_{\H}^2+C\|\u_1\|_{\H}^4\|\nabla\u_1\|_{\H}^4\|\u\|_{\H}^2
	.\label{s9} 	
	\end{align}
 Using Taylor's formula (Theorem 7.9.1, \cite{PGC}), Agmon's and Young's inequalities, we obtain (cf. \cite{MTM6})
		\begin{align}\label{s10}
		I_7&\leq\beta \left\|\mathcal{C}(\u_1)-\mathcal{C}(\u_2)\right\|_{\H}\|\u_t\|_{\H}\nonumber\\& \leq\beta\bigg\|\int_0^1\mathcal{C}'\big(\theta\u_1+(1-\theta)\u_2\big)(\u_1-\u_2)\d \theta\bigg\|_{\H}\|\u_t\|_{\H}\nonumber\\&\leq C \sup_{0<\theta<1}\big\|(\u_1-\u_2)|\theta\u_1+(1-\theta)\u_2|^{r-1} \no\\&\quad+(r-1)\big(\theta\u_1+(1-\theta)\u_2\big)\big|\theta\u_1+(1-\theta)\u_2\big|^{r-3} \no\\& \quad\times\big((\theta\u_1+(1-\theta)\u_2) \cdot (\u_1-\u_2)\big) \big\|_{\H} \|\u_t\|_{\H} \nonumber\\&\leq C\left(\|\u_1\|_{\widetilde{\L}^\infty}+\|\u_2\|_{\widetilde{\L}^\infty}\right)^{r-1}\|\u\|_{\H} \|\u_t\|_{\H}
		\nonumber\\&\leq C\left(\|\u_1\|_{\H}+\|\u_1\|_{\H}^\frac{1}{2}\|\Delta\u_1\|_{\H}^\frac{1}{2}+\|\u_2\|_{\H}+\|\u_2\|_{\H}^\frac{1}{2}\|\Delta\u_2\|_{\H}^\frac{1}{2}\right)^{r-1}\|\u\|_{\H}\|\u_t\|_{\H}
		\nonumber\\&\leq C \left(\|\u_1\|_{\H}+\|\u_1\|_{\H}^\frac{1}{2}\|\Delta\u_1\|_{\H}^\frac{1}{2}+\|\u_2\|_{\H}+\|\u_2\|_{\H}^\frac{1}{2}\|\Delta\u_2\|_{\H}^\frac{1}{2}\right)^{2(r-1)}\|\u\|_{\H}^2+\frac{1}{8}\|\u_t\|_{\H}^2.
	\end{align}
 Using the similar arguments as above, it can be easily deduced that
	\begin{align}\label{s11}
		I_8&\leq\beta \mu \left\|\mathcal{C}(\u_1)-\mathcal{C}(\u_2)\right\|_{\H}\|\Delta\u\|_{\H}\nonumber\\&\leq C\sup_{0<\theta<1} \big\||\theta\u_1+(1-\theta)\u_2|^{r-1}\big\|_{\widetilde{\L}^\infty}\|\u_1-\u_2\|_{\H}\|\Delta\u\|_{\H}
		\no\\&\leq \frac{\mu^2}{12}\|\Delta\u\|_{\H}^2+ C\left(\|\u_1\|_{\H}+\|\u_1\|_{\H}^\frac{1}{2}\|\Delta\u_1\|_{\H}^\frac{1}{2}+\|\u_2\|_{\H}+\|\u_2\|_{\H}^\frac{1}{2}\|\Delta\u_2\|_{\H}^\frac{1}{2}\right)^{2(r-1)}\|\u\|_{\H}^2.
		\end{align}
Substituting the estimates \eqref{s4}-\eqref{s11} in \eqref{s3} and integrating the resulting estimate from $t_3$ to $t$, we find 
	\begin{align}\label{s14}
	&\int_{t_3}^t\|\u_t(s)\|_{\H}^2 \d s+ 2\mu \|\nabla\u(t)\|_{\H}^2+\mu^2\int_{t_3}^t\|\Delta\u(s)\|_{\H}^2 \d s \no\\&\leq
	2\mu \|\nabla\u(t_3)\|_{\H}^2+\alpha \|\u(t_3)\|_{\H}^2 +20(t-t_3)\big(\|\f\|_{\H}^2\|g_1\|_0^2+\|\f_2\|_{\H}^2\|g\|_0^2\big) 
	\no \\& \quad+ C (t-t_3)^\frac{2}{3}\sup_{t\in[t_3,T]}\left(\|\u(t)\|_{\H}^\frac{4}{3} \|\Delta\u_2(t)\|_{\H}^2\right)\bigg( \int_{t_3}^t\|\nabla\u(s)\|_{\H}^2 \d s\bigg)^\frac{1}{3}
		\no\\&\quad+C (t-t_3)\sup_{t\in[t_3,T]}\bigg(\|\u_1(t)\|_{\H}\|\nabla\u_1(t)\|_{\H}+\|\u_1(t)\|_{\H}^4\|\nabla\u_1(t)\|_{\H}^4 \bigg) \sup_{t\in[t_3,T]}\|\u(t)\|_{\H}^2
	\no	\\& \quad +C \sup_{t\in[t_3,T]} \left(\|\u_1(t)\|_{\H}+\|\u_1(t)\|_{\H}^\frac{1}{2}\|\Delta\u_1(t)\|_{\H}^\frac{1}{2}+\|\u_2(t)\|_{\H}+\|\u_2(t)\|_{\H}^\frac{1}{2}\|\Delta\u_2(t)\|_{\H}^\frac{1}{2}\right)^{2(r-1)}
	\no\\& \quad \times (t-t_3)\sup_{t\in[t_3,T]} \|\u(t)\|_{\H}^2
		\end{align}
for all $t \in [t_3,T]$.  Using the estimates obtained in Lemmas \ref{lemmabe2}, \ref{lemb5} and \ref{lemmaS1}, from \eqref{s14}, we immediately obtain \eqref{s1}.
	 \vskip 0.2 cm 
	  \noindent\textbf{Case II:} \emph{$d=3$ and $r \geq3$.}  We estimate $I_3$ and $I_4$ 
	  using H\"older's, Gagliardo-Nirenberg's and Young's inequalities as
	  \begin{align}
	  	I_3 &\leq \|\u\|_{\widetilde{\L}^{3}} \|\nabla\u_2\|_{\widetilde{\L}^{6}}\|\u_t\|_{\H}\leq C\|\u\|_{\H}^\frac{1}{2}\|\nabla\u\|_{\H}^\frac{1}{2}\|\Delta\u_2\|_{\H}\|\u_t\|_{\H}\no\\&\leq \frac{1}{8}\|\u_t\|_{\H}^2+C \|\u\|_{\H}\|\nabla\u\|_{\H} \|\Delta\u_2\|_{\H}^2, \label{s20}\\
	  	I_4 &\leq \mu \|\u\|_{\widetilde{\L}^{3}} \|\nabla\u_2\|_{\widetilde{\L}^{6}}\|\Delta\u\|_{\H}
	  	\leq C\mu\|\u\|_{\H}^\frac{1}{2}\|\nabla\u\|_{\H}^\frac{1}{2}\|\Delta\u_2\|_{\H}\|\Delta\u\|_{\H}\no\\& \leq
	  	\frac{\mu^2}{12}	\|\Delta\u\|_{\H}^2+C\|\u\|_{\H}\|\nabla\u\|_{\H} \|\Delta\u_2\|_{\H}^2
	  	.\label{21}
	  \end{align}
	 Similarly, we compute
	\begin{align}
		I_5 &\leq \|\u_1\|_{\widetilde{\L}^{4}} \|\nabla\u\|_{\widetilde{\L}^{4}}\|\u_t\|_{\H}\leq C\|\u_1\|_{\H}^\frac{1}{4}\|\nabla\u_1\|_{\H}^\frac{3}{4}\|\u\|_{\H}^\frac{1}{8}\|\u\|_{\H_p^2}^\frac{7}{8} \|\u_t\|_{\H}\nonumber\\&\leq \frac{1}{8}\|\u_t\|_{\H}^2+C \|\u_1\|_{\H}^\frac{1}{2}\|\nabla\u_1\|_{\H}^\frac{3}{2}\|\u\|_{\H}^\frac{1}{4}\left(\|\u\|_{\H}^\frac{7}{4}+\|\Delta\u\|_{\H}^\frac{7}{4}\right)
		\nonumber\\&\leq \frac{1}{8}\|\u_t\|_{\H}^2+\frac{\mu^2}{12}\|\Delta\u\|_{\H}^{2}+C\|\u_1\|_{\H}^\frac{1}{2}\|\nabla\u_1\|_{\H}^\frac{3}{2}\|\u\|_{\H}^2+C\|\u_1\|_{\H}^4\|\nabla\u_1\|_{\H}^{12} \|\u\|_{\H}^2, \label{s22}\\
		I_6 &\leq \mu \|\u_1\|_{\widetilde{\L}^{4}} \|\nabla\u\|_{\widetilde{\L}^{4}}\|\Delta\u\|_{\H}\leq C\mu\|\u_1\|_{\H}^\frac{1}{4}\|\nabla\u_1\|_{\H}^\frac{3}{4}\|\u\|_{\H}^\frac{1}{8}\left(\|\u\|_{\H}^\frac{7}{8}+\|\Delta\u\|_{\H}^\frac{7}{8}\right)\|\Delta\u\|_{\H}
		\no\\& \leq
		\frac{\mu^2}{6}	\|\Delta\u\|_{\H}^2+C\|\u_1\|_{\H}^\frac{1}{2}\|\nabla\u_1\|_{\H}^\frac{3}{2}\|\u\|_{\H}^2+C\|\u_1\|_{\H}^4\|\nabla\u_1\|_{\H}^{12} \|\u\|_{\H}^2.\label{s23} 	
	\end{align}
 Using the similar arguments as  \eqref{s10}, we deduce 
\begin{align}\label{s24}
	I_7&\leq\beta \left\|\mathcal{C}(\u_1)-\mathcal{C}(\u_2)\right\|_{\H}\|\u_t\|_{\H}\no\\&\leq C\left(\|\u_1\|_{\widetilde{\L}^\infty}+\|\u_2\|_{\widetilde{\L}^\infty}\right)^{r-1}\|\u\|_{\H} \|\u_t\|_{\H}
	\nonumber\\&\leq C\left(\|\u_1\|_{\H}^\frac{1}{4}\|\u_1\|_{\H_p^2}^\frac{3}{4}+\|\u_2\|_{\H}^\frac{1}{4}\|\u_2\|_{\H_p^2}^\frac{3}{4}\right)^{r-1}\|\u\|_{\H}\|\u_t\|_{\H}
	\nonumber\\&\leq C \left(\|\u_1\|_{\H}+\|\u_1\|_{\H}^\frac{1}{4}\|\Delta\u_1\|_{\H}^\frac{3}{4}+\|\u_2\|_{\H}+\|\u_2\|_{\H}^\frac{1}{4}\|\Delta\u_2\|_{\H}^\frac{3}{4}\right)^{2(r-1)}\|\u\|_{\H}^2+\frac{1}{8}\|\u_t\|_{\H}^2, 
\end{align}
and similarly, we obtain
\begin{align}\label{s25}
	I_8&\leq\beta \mu \left\|\mathcal{C}(\u_1)-\mathcal{C}(\u_2)\right\|_{\H}\|\Delta\u\|_{\H}\nonumber\\&\leq C\sup_{0<\theta<1} \big\||\theta\u_1+(1-\theta)\u_2|^{r-1}\big\|_{\widetilde{\L}^\infty}\|\u_1-\u_2\|_{\H}\|\Delta\u\|_{\H}
	\no\\&\leq \frac{\mu^2}{12}\|\Delta\u\|_{\H}^2+ C \left(\|\u_1\|_{\H}+\|\u_1\|_{\H}^\frac{1}{4}\|\Delta\u_1\|_{\H}^\frac{3}{4}+\|\u_2\|_{\H}+\|\u_2\|_{\H}^\frac{1}{4}\|\Delta\u_2\|_{\H}^\frac{3}{4}\right)^{2(r-1)}\|\u\|_{\H}^2.
\end{align}
  Substituting the estimates \eqref{s4}, \eqref{s5} and \eqref{s20}-\eqref{s25} in \eqref{s3} and integrating it from $t_3$ to $t$, we find 
  \begin{align}\label{s26}
  	&\int_{t_3}^t\|\u_t(s)\|_{\H}^2 \d s+ 2\mu \|\nabla\u(t)\|_{\H}^2+\mu^2\int_{t_3}^t\|\Delta\u(s)\|_{\H}^2 \d s \no\\&\leq
  	2\mu \|\nabla \u(t_3)\|_{\H}^2+\alpha \|\u(t_3)\|_{\H}^2+ 20(t-t_3)\big(\|\f\|_{\H}^2\|g_1\|_0^2+\|\f_2\|_{\H}^2\|g\|_0^2\big) \no\\&\quad+C (t-t)^\frac{1}{2}\sup_{t\in[t_3,T]} \big(\|\u(t)\|_{\H} \|\Delta\u_2(t)\|_{\H}^2\big)\bigg(\int_{t_3}^t\|\nabla \u(s)\|_{\H}^2\d s\bigg)^\frac{1}{2}
  	\no\\&\quad+C(t-t_3)\sup_{t\in[t_3,T]}\bigg(\|\u_1(t)\|_{\H}^\frac{1}{2}\|\nabla\u_1(t)\|_{\H}^\frac{3}{2}+\|\u_1(t)\|_{\H}^{4}\|\nabla\u_1(t)\|_{\H}^{12}\bigg)\sup_{t\in[t_3,T]}\|\u(t)\|_{\H}^{2}
  	\no \\& \quad+ C \sup_{t\in[t_3,T]} \left(\|\u_1(t)\|_{\H}+\|\u_1(t)\|_{\H}^\frac{1}{4}\|\Delta\u_1(t)\|_{\H}^\frac{3}{4}+\|\u_2(t)(t)\|_{\H}+\|u_2(t)\|_{\H}^\frac{1}{4}\|\Delta\u_2(t)\|_{\H}^\frac{3}{4}\right)^{2(r-1)}
  	\no \\& \quad \times (t-t_3)\sup_{t\in[t_3,T]}\|\u(t)\|_{\H}^2
  \end{align}
  for all $t \in [t_3,T]$. Thus, from \eqref{s26}, we immediately have  \eqref{s1} by using the estimates obtained in Lemmas \ref{lemmabe2}, \ref{lemb5} and \ref{lemmaS1}.
 
 Recalling that $\u \in \C^1([T/8,T],\V)$, there exists a time $t_4 \in \big(\frac{4T}{8},\frac{5T}{8}\big)$ such that
	 \begin{align*}
	 	\|\u_t(t_4)\|_{\H}^2=\frac{1}{\frac{5T}{8}-\frac{4T}{8}} \int_{4T/8}^{5T/8}\|\u_t(t)\|_{\H}^2 \d t 
	 	\leq \frac{8}{T}\int_{t_3}^T\|\u_t(t)\|_{\H}^2 \d t,
	 \end{align*}
 which implies \eqref{s2} by \eqref{s1} and completes the proof.
\end{proof}
	The following lemma proves the stability of the vector-valued function $\f$ of the solution of the inverse problem.
\begin{lemma}\label{lemmaSf}
	Let $\u_{0i} \in \H, \ \boldsymbol{\varphi}_{i} \in \H^2(\T) \cap \V $, $\nabla \psi_{i} \in \G(\T)$, $g_i, (g_{i})_t \in \C\left(\T \times [0,T]\right)$ satisfy the assumption \eqref{1e}, and $\f_{i} \in \H$, for $i=1,2$. For $r \geq 1 $, in 2D, and for $r \geq 3$, in 3D, we have  
	\begin{align}\label{S21}
		\|\f\|_{\H}\leq C\big(\|\u_0\|_{\H}+\|g\|_0+\|g_t\|_0+\|\nabla \boldsymbol{\varphi}\|_{\H}+\|\nabla \psi-\mu \Delta \boldsymbol{\varphi}\|_{\H}\big),
	\end{align}
	where $C$ depends on the input data,  $\mu,\alpha,\beta,r$ and $T$.
\end{lemma}

\begin{proof}
	Differentiating \eqref{S3} with respect to time $t$ and then multiplying by $\u_t(\cdot)$ to the resulting equation, we obtain
	\begin{align}\label{S22}
		&\frac{1}{2}\frac{\d}{\d t}\|\u_t(t)\|_{\H}^2+ \mu \|\nabla\u_t(t)\|_{\H}^2+\alpha\|\u_t(t)\|_{\H}^2\nonumber\\&=((\f g_{1t}(t)+\f_2 g_t(t),\u_t(t)) -((\u_{1t}(t) \cdot \nabla)\u(t),\u_t(t))\nonumber\\&\quad-((\u_t(t) \cdot \nabla)\u_2(t),\u_t(t))-((\u(t) \cdot \nabla)\u_{2t}(t),\u_t(t))\nonumber\\&\quad-\beta  \big(\mathcal{C}'(\u_1)\u_{1t}(t)-\mathcal{C}'(\u_2)\u_{2t}(t),\u_t (t)\big)=:\sum_{j=1}^5E_j.
	\end{align}
	Next, we estimate each $E_j$'s $(j=1,\ldots,5)$ separately as follows:
	Using H\"older's and Young's inequalities, we get
	\begin{align}\label{S23}
		E_1\leq\frac{1}{2 \alpha}\left(\|\f\|_{\H}\|g_{1t}\|_{\mathrm{L}^\infty}+\|\f_2\|_{\H}\|g_t\|_{\mathrm{L}^\infty}\right)^2+\frac{\alpha}{2}\|\u_t\|_{\H}^2.
	\end{align}
\vskip 0.2 cm
\noindent\textbf{Case I:} \emph{$d=2$ and $r \geq 1$.}
Making use of H\"older's,  Ladyzhenskaya's and Young's inequalities, we obtain the following estimate:
\begin{align}
	E_2 \nonumber&\leq \|\u_{1t}\|_{\widetilde{\L}^4} \|\nabla\u_t\|_{\H}\|\u\|_{\widetilde{\L}^4}\leq \sqrt{2} \|\u_{1t}\|_{\H}^\frac{1}{2} \|\nabla\u_{1t}\|_{\H}^\frac{1}{2}\|\nabla\u_t\|_{\H}\|\u\|_{\H}^\frac{1}{2} \|\nabla\u\|_{\H}^\frac{1}{2}\\&\leq \frac{\mu}{6}\|\nabla\u_t\|_{\H}^2+ C \|\u_{1t}\|_{\H} \|\nabla\u_{1t}\|_{\H}\|\u\|_{\H}\|\nabla\u\|_{\H},\label{S24}\\
	E_3 &\leq  \|\nabla\u_2\|_{\H}\|\u_t\|_{\widetilde{\L}^4}^2\leq \sqrt{2} \|\nabla\u_2\|_{\H}\|\u_t\|_{\H}\|\nabla\u_t\|_{\H} \leq \frac{\mu}{6}\|\nabla\u_t\|_{\H}^2+\frac{3}{\mu} \|\nabla\u_2\|_{\H}^2\|\u_t\|_{\H}^2, \label{S27} \\
	E_4& \leq \|\u\|_{\widetilde{\L}^4} \|\nabla\u_{t}\|_{\H}\|\u_{2t}\|_{\widetilde{\L}^4}\leq \frac{\mu}{6}\|\nabla\u_t\|_{\H}^2+ C  \|\u_{2t}\|_{\H} \|\nabla\u_{2t}\|_{\H}\|\u\|_{\H}\|\nabla\u\|_{\H}.	\label{S25}
\end{align}
	Applying Taylor's formula and H\"older's, Agmon's and Young's  inequalities, one can deduce 
\begin{align}\label{S26}
	E_5	&=-\beta\left(\mathcal{C}'(\u_1)\u_{1t}-\mathcal{C}'(\u_2)\u_{2t},\u_{1t}-\u_{2t} \right)\nonumber\\&=-\beta \big\{(\mathcal{C}'(\u_1)(\u_{1t}-\u_{2t}),\u_{1t}-\u_{2t})+\big((\mathcal{C}'(\u_1)-\mathcal{C}'(\u_2))\u_{2t},\u_{1t}-\u_{2t}\big)\big\}
	\nonumber\\&\leq-\beta \big((\mathcal{C}'(\u_1)-\mathcal{C}'(\u_2))\u_{2t},\u_{1t}-\u_{2t}\big)
	\nonumber\\& \leq \beta \sup_{0<\theta<1}\| \mathcal{C}''(\theta\u_1+(1-\theta)\u_2) (\u \otimes\u_{2t})\|_{\H}\|\u_{1t}-\u_{2t}\|_{\H}
	\nonumber\\& \leq C
	\left(\|\u_1\|_{\widetilde{\L}^\infty}+\|\u_2\|_{\widetilde{\L}^\infty}\right)^{r-2}\|\u\|_{\widetilde{\L}^4}\|\u_{2t}\|_{\widetilde{\L}^4}\|\u_{1t}-\u_{2t}\|_{\H}
	\nonumber\\& \leq C
	\left(\|\u_1\|_{\H}^\frac{1}{2}\|\u_1\|_{\H_p^2}^\frac{1}{2}+\|\u_2\|_{\H}^\frac{1}{2}\|\u_2\|_{\H_p^2}^\frac{1}{2}\right)^{r-2}\|\u\|_{\H}^\frac{1}{2}\|\nabla\u\|_{\H}^\frac{1}{2}\|\u_{2t}\|_{\H}^\frac{1}{2}\|\nabla\u_{2t}\|_{\H}^\frac{1}{2}\|\u_{1t}-\u_{2t}\|_{\H}
\nonumber\\&
	\leq  C\left(\|\u_1\|_{\H}+\|\u_1\|_{\H}^\frac{1}{2}\|\Delta\u_1\|_{\H}^\frac{1}{2}+\|\u_2\|_{\H}+\|\u_2\|_{\H}^\frac{1}{2}\|\Delta\u_2\|_{\H}^\frac{1}{2}\right)^{2(r-2)}\|\u\|_{\H}\|\nabla\u\|_{\H}
	\no \\& \quad+\|\u_{2t}\|_{\H}\|\nabla\u_{2t}\|_{\H}
	\|\u_{t}\|_{\H}^2,
\end{align}
for $r\geq 3$. Note that for the case $r=2$, \eqref{S26} holds. For $r=2$, using H\"older's and Young's inequalities, we estimate $E_5$ as
\begin{align}\label{S271}
	E_5 &=-\beta\left(\mathcal{C}'(\u_1)\u_{1t}-\mathcal{C}'(\u_2)\u_{2t},\u_{1t}-\u_{2t} \right)
	\nonumber\\&\leq-\beta \big((\mathcal{C}'(\u_1)-\mathcal{C}'(\u_2))\u_{2t},\u_{1t}-\u_{2t}\big)
	\nonumber\\& \leq -\beta \langle(|\u_1|-|\u_2|)\u_{2t},\u_t \rangle-\beta \left\langle\frac{\u_1}{|\u_1|}(\u_1 \cdot \u_{2t})-\frac{\u_2}{|\u_2|}(\u_2 \cdot \u_{2t}),\u_t \right \rangle
	\no\\& \leq 2\beta \|\u_1-\u_2\|_{\widetilde{\L}^4} \|\u_{2t}\|_{\widetilde{\L}^4}\|\u_t\|_{\H}
	\no \\& \leq C \|\u_{2t}\|_{\H}\|\nabla\u_{2t}\|_{\H}
	\|\u_{t}\|_{\H}^2+C\|\u\|_{\H}\|\nabla\u\|_{\H}.
\end{align}
 Plugging the estimates \eqref{S23}-\eqref{S271} in \eqref{S22}, and integrating the resulting estimate from $t_4$ to $t$, we find 
	\begin{align}\label{S28}
		&\|\u_t(t) \|_{\H}^2+\mu \int_{t_4}^t \|\nabla\u_t (s)\|_{\H}^2\d s+\alpha\int_{t_4}^t \|\u_t (s)\|_{\H}^2\d s \nonumber\\&\leq \|\u_t(t_4) \|_{\H}^2+\frac{(t-t_4)}{\alpha}\big(\|\f\|_{\H}\|g_{1t}\|_0+\|\f_2\|_{\H}\|g_t\|_0\big)^2 \nonumber\\&\quad+C\int_{t_4}^t \big(\|\u_{1t}(s)\|_{\H} \|\nabla\u_{1t}(s)\|_{\H}+\|\u_{2t}(s)\|_{\H} \|\nabla\u_{2t}(s)\|_{\H}\big)\|\u(s)\|_{\H}\|\nabla\u(s)\|_{\H} \d s\no \\&\quad+\frac{6}{\mu}\int_{t_4}^t\|\nabla\u_2(s)\|_{\H}^2\|\u_t(s)\|_{\H}^2 \d s+C \int_{t_4}^t\|\u_{2t}(s)\|_{\H}\|\nabla\u_{2t}(s)\|_{\H}
		\|\u_{t}(s)\|_{\H}^2 \d s\nonumber\\&\quad+C\int_{t_4}^t   \left(\|\u_1(s)\|_{\H}+\|\u_1(s)\|_{\H}^\frac{1}{2}\|\Delta\u_1(s)\|_{\H}^\frac{1}{2}+\|\u_2(s)\|_{\H}+\|\u_2(s)\|_{\H}^\frac{1}{2}\|\Delta\u_2(s)\|_{\H}^\frac{1}{2}\right)^{2(r-2)}
		\nonumber\\&\quad \times\|\u(s)\|_{\H}\|\nabla\u(s)\|_{\H}\d s
 \nonumber\\&\leq \|\u_t(t_4) \|_{\H}^2+\frac{t-t_4}{ \alpha }\big(\|\f\|_{\H}\|g_{1t}\|_0+\|\f_2\|_{\H}\|g_t\|_0\big)^2 +C\sup_{t\in[t_4,T]}\big(\|\u(t)\|_{\H}\|\nabla\u(t)\|_{\H}\big) \nonumber\\&\quad \times (t-t_4)\sup_{t\in[t_4,T]}\big(\|\u_{1t}(t)\|_{\H} \|\nabla\u_{1t}(t)\|_{\H}+\|\u_{2t}(t)\|_{\H} \|\nabla\u_{2t}(t)\|_{\H} \big) \nonumber\\&\quad +C \sup_{t\in[t_4,T]}\big(\|\nabla\u_2(t)\|_{\H}^2+\|\u_{2t}(t)\|_{\H}\|\nabla\u_{2t}(t)\|_{\H}\big)\int_{t_4}^t\|\u_{t}(s)\|_{\H}^2\d s\nonumber\\&\quad+C  \sup_{t\in[t_4,T]}\left(\|\u_1(t)\|_{\H}+\|\u_1(t)\|_{\H}^\frac{1}{2}\|\Delta\u_1(t)\|_{\H}^\frac{1}{2}+\|\u_2(t)\|_{\H}+\|\u_2(t)\|_{\H}^\frac{1}{2}\|\Delta\u_2(t)\|_{\H}^\frac{1}{2}\right)^{2(r-2)}
 \no \\& \quad \times (t-t_4) \sup_{t\in[t_4,T]}\big(\|\u(t)\|_{\H}\|\nabla\u(t)\|_{\H}\big),
	\end{align}
	for all $t\in[t_4,T]$. Using the energy estimates obtained in Lemmas \ref{lemb5}, \ref{lemb6}, \ref{lemmabe2} and \ref{3.5}  in the inequality \eqref{S28}, it can be easily seen that 
	\begin{align*}
		\sup_{t\in[t_4,T]}\|\u_t(t) \|_{\H}^2 \leq C \left( \|\u_0\|_{\H}^2+\|\f\|_{\H}^2+\|g\|_0^2+\|g_t\|_0^2 \right),
	\end{align*}
	and as a result, we get
	\begin{align}\label{S29}
		\|\u_t(T) \|_{\H}\leq C\big(\|\u_0\|_{\H}+\|\f\|_{\H}+\|g\|_0+\|g_t\|_0\big).
	\end{align}
	Using the final overdetermination data in \eqref{S3}, one can easily deduce 
	\begin{align*}
		\f g_1+\f_2 g&=\u_t(T)+(\boldsymbol{\varphi}_1 \cdot \nabla)\boldsymbol{\varphi}\nonumber+(\boldsymbol{\varphi} \cdot \nabla)\boldsymbol{\varphi}_2-\mu \Delta \boldsymbol{\varphi}+\nabla \psi\\&\quad+\beta \left(|\boldsymbol{\varphi}_1|^{r-1}\boldsymbol{\varphi}_1-|\boldsymbol{\varphi}_2|^{r-1}\boldsymbol{\varphi}_2\right),
	\end{align*}
	which leads to
	\begin{align}\label{S30}
		g_T\|\f\|_{\H}\leq\|\f g_1\|_{\H}&\leq \|\f_2\|_{\H}\|g\|_0+ \|\u_t(T)\|_{\H}+\|\boldsymbol{\varphi}_1\|_{\widetilde{\L}^\infty}\|\nabla \boldsymbol{\varphi}\|_{\H}+\|\boldsymbol{\varphi}\|_{\widetilde{\L}^3}\|\nabla \boldsymbol{\varphi}_2\|_{\widetilde{\L}^6}\nonumber\\&\quad+\|\nabla \psi-\mu \Delta \boldsymbol{\varphi}\|_{\L^2}+C(\|\boldsymbol{\varphi}_1\|_{\widetilde{\L}^{\infty}}+\|\boldsymbol{\varphi}_2\|_{\widetilde{\L}^{\infty}})^{r-1}\|\boldsymbol{\varphi}\|_{\H}
		\nonumber\\&\leq \|\f_2\|_{\H}\|g\|_0+ \|\u_t(T)\|_{\H}+\left(\|\boldsymbol{\varphi}_1\|_{\H}+\|\boldsymbol{\varphi}_1\|_{\H}^\frac{1}{2}\|\Delta\boldsymbol{\varphi}_1\|_{\H}^\frac{1}{2}\right)\|\nabla \boldsymbol{\varphi}\|_{\H}\nonumber\\&\quad+\|\boldsymbol{\varphi}\|_{\H}^\frac{2}{3}\|\nabla\boldsymbol{\varphi}\|_{\H}^\frac{1}{3}\|\Delta \boldsymbol{\varphi}_2\|_{\H}+\|\nabla \psi-\mu \Delta \boldsymbol{\varphi}\|_{\L^2}\no \\& \quad+C\left(\|\boldsymbol{\varphi}_1\|_{\H}+\|\boldsymbol{\varphi}_1\|_{\H}^\frac{1}{2}\|\Delta\boldsymbol{\varphi}_2\|_{\H}^\frac{1}{2}+\|\boldsymbol{\varphi}_1\|_{\H}+\|\boldsymbol{\varphi}_2\|_{\H}^\frac{1}{2}\|\Delta\boldsymbol{\varphi}_2\|_{\H}^\frac{1}{2}\right)^{r-1}\|\boldsymbol{\varphi}\|_{\H}
		\nonumber\\&\leq
		\|\u_t(T)\|_{\H}+C\big(\| \nabla\boldsymbol{\varphi}\|_{\H}+\|\nabla \psi-\mu \Delta \boldsymbol{\varphi}\|_{\L^2}+\|g\|_0\big),
	\end{align}
where we have used Sobolev's inequality in the above relation.	Plugging the relation \eqref{S29} in \eqref{S30}, we get
	\begin{align*}
		\|\f\|_{\H}\leq C\left(\|\u_0\|_{\H}+\|g\|_0+\|g_t\|_0+\|\nabla \boldsymbol{\varphi}\|_{\H}+\|\nabla \psi-\mu \Delta \boldsymbol{\varphi}\|_{\L^2}\right),
	\end{align*}
	which is \eqref{S21}. For $r=1$, a calculation similar to \eqref{S271} and the use of the same arguments as above lead to the required result. 
		\vskip 0.2 cm
	\noindent\textbf{Case II:} \emph{$d=3$ and $r \geq 3$.} Using H\"older's, Gagliardo-Nirenberg's and Young's inequalities, we obtain the following estimate:
	\begin{align}
		E_2 \nonumber&\leq \|\u_{1t}\|_{\widetilde{\L}^4} \|\nabla\u_t\|_{\H}\|\u\|_{\widetilde{\L}^4}\leq C \|\u_{1t}\|_{\H}^\frac{1}{4} \|\nabla\u_{1t}\|_{\H}^\frac{3}{4}\|\nabla\u_t\|_{\H}\|\u\|_{\H}^\frac{1}{4} \|\nabla\u\|_{\H}^\frac{3}{4}\\&\leq \frac{\mu}{6}\|\nabla\u_t\|_{\H}^2+ C \|\u_{1t}\|_{\H}^\frac{1}{2} \|\nabla\u_{1t}\|_{\H}^\frac{3}{2}\|\u\|_{\H}^\frac{1}{2}\|\nabla\u\|_{\H}^\frac{3}{2},\label{S241}\\
		E_3 &\leq  \|\nabla\u_2\|_{\H}\|\u_t\|_{\widetilde{\L}^4}^2\leq C \|\nabla\u_2\|_{\H}\|\u_t\|_{\H}^\frac{1}{2}\|\nabla\u_t\|_{\H}^\frac{3}{2} \leq \frac{\mu}{6}\|\nabla\u_t\|_{\H}^2+C \|\nabla\u_2\|_{\H}^4\|\u_t\|_{\H}^2, \label{S2711} \\
		E_4& \leq \|\u\|_{\widetilde{\L}^4} \|\nabla\u_{t}\|_{\H}\|\u_{2t}\|_{\widetilde{\L}^4}\leq \frac{\mu}{6}\|\nabla\u_t\|_{\H}^2+  C \|\u_{2t}\|_{\H}^\frac{1}{2} \|\nabla\u_{2t}\|_{\H}^\frac{3}{2}\|\u\|_{\H}^\frac{1}{2}\|\nabla\u\|_{\H}^\frac{3}{2}.	\label{S251}
	\end{align}
A calculation similar to \eqref{S26} gives
\begin{align}\label{S261}
	E_5	&=-\beta\left(\mathcal{C}'(\u_1)\u_{1t}-\mathcal{C}'(\u_2)\u_{2t},\u_{1t}-\u_{2t} \right)
	\nonumber\\& \leq C
	\left(\|\u_1\|_{\widetilde{\L}^\infty}+\|\u_2\|_{\widetilde{\L}^\infty}\right)^{r-2}\|\u\|_{\widetilde{\L}^4}\|\u_{2t}\|_{\widetilde{\L}^4}\|\u_{1t}-\u_{2t}\|_{\H}
	\nonumber\\& \leq C
	\left(\|\u_1\|_{\H}^\frac{1}{4}\|\u_1\|_{\H_p^2}^\frac{3}{4}+\|\u_1\|_{\H}^\frac{1}{4}\|\u_1\|_{\H_p^2}^\frac{3}{4}\right)^{r-2}\|\u\|_{\H}^\frac{1}{4}\|\nabla\u\|_{\H}^\frac{3}{4}\|\u_{2t}\|_{\H}^\frac{1}{4}\|\nabla\u_{2t}\|_{\H}^\frac{3}{4}\|\u_{1t}-\u_{2t}\|_{\H}
	\nonumber\\&
	\leq C\left(\|\u_1\|_{\H}+\|\u_1\|_{\H}^\frac{1}{4}\|\Delta\u_1\|_{\H}^\frac{3}{4}+\|\u_2\|_{\H}+\|\u_2\|_{\H}^\frac{1}{4}\|\Delta\u_2\|_{\H}^\frac{3}{4}\right)^{2(r-2)}\|\u\|_{\H}^\frac{1}{2}\|\nabla\u\|_{\H}^\frac{3}{2}
	\no \\& \quad+C \|\u_{2t}\|_{\H}^\frac{1}{2}\|\nabla\u_{2t}\|_{\H}^\frac{3}{2}
	\|\u_{t}\|_{\H}^2,
\end{align}
for $r\geq 3$.
Plugging the estimates \eqref{S23}, \eqref{S241}-\eqref{S261} in \eqref{S22}, and integrating it from $t_4$ to $t$, we find
	\begin{align}\label{S36}
		&\|\u_t(t) \|_{\H}^2+\mu \int_{t_4}^t \|\u_t (s)\|_{\V}^2\d s+\alpha\int_{t_4}^t \|\u_t (s)\|_{\H}^2\d s 
		\nonumber\\&\leq \|\u_t(t_4) \|_{\H}^2+\frac{t-t_4}{ \alpha }\big(\|\f\|_{\H}\|g_{1t}\|_0+\|\f_2\|_{\H}\|g_t\|_0\big)^2 +C\sup_{t\in[t_4,T]}\left(\|\u(t)\|_{\H}^\frac{1}{2}\|\nabla\u(t)\|_{\H}^\frac{3}{2}\right)
		\nonumber\\&\quad \times (t-t_4)
		\sup_{t\in[t_4,T]}\left(\|\u_{1t}(t)\|_{\H}^\frac{1}{2} \|\nabla\u_{1t}(t)\|_{\H}^\frac{3}{2}+\|\u_{2t}(t)\|_{\H}^\frac{1}{2} \|\nabla\u_{2t}(t)\|_{\H}^\frac{3}{2} \right) \nonumber\\&\quad +C \sup_{t\in[t_4,T]}\big(\|\nabla\u_2(t)\|_{\H}^4+\|\u_{2t}(t)\|_{\H}^\frac{1}{2}\|\nabla\u_{2t}(t)\|_{\H}^\frac{3}{2}\big)\int_{t_4}^t\|\u_{t}(s)\|_{\H}^2\d s\nonumber\\&\quad+C  \sup_{t\in[t_4,T]}\left(\|\u_1(t)\|_{\H}+\|\u_1(t)\|_{\H}^\frac{1}{4}\|\Delta\u_1(t)\|_{\H}^\frac{3}{4}+\|\u_2(t)\|_{\H}+\|\u_2(t)\|_{\H}^\frac{1}{4}\|\Delta\u_2(t)\|_{\H}^\frac{3}{4}\right)^{2(r-2)} \no\\& \quad \times (t-t_4) \sup_{t\in[t_4,T]}\left(\|\u(t)\|_{\H}^\frac{1}{2}\|\nabla\u(t)\|_{\H}^\frac{3}{2}\right),
	\end{align}
	for all $t \in [t_4,T]$. 
	Using the energy estimates obtained in Lemmas \ref{lemb5}-\ref{lemb6}, \ref{lemmabe2} and \ref{3.5} in the above inequality, one can easily deduce 
	\begin{align*}
		\sup_{t\in[t_4,T]}\|\u_t(t) \|_{\H}^2 \leq C \left( \|\u_0\|_{\H}^2+\|\f\|_{\H}^2+\|g\|_0^2+\|g_t\|_0^2 \right),
	\end{align*}
	and thus, we obtain
	\begin{align}\label{S37}
		\|\u_t(T) \|_{\H}\leq C\big(\|\u_0\|_{\H}+\|\f\|_{\H}+\|g\|_0+\|g_t\|_0\big).
	\end{align}
 By using the similar arguments as $d=2$ and $r \geq 1$, we finally obtain the relation \eqref{S21} (cf. \eqref{S30}).
	\end{proof}

\begin{proof}[Proof of part (ii) of Theorem \ref{thm2}] {
		\emph{Stability of the pressure gradient $\nabla p$:} The equation \eqref{S3} can be used to establish the stability of the pressure gradient $\nabla p$.
			Taking divergence on both the sides of \eqref{S3} and making the use of divergence free condition on $\f_{i}$ for $i=1,2$, we get 
		\begin{align}\label{d}
			p &=(-\Delta)^{-1}\big[\nabla \cdot \big\{(\u_1 \cdot \nabla)\u+(\u \cdot \nabla)\u_2+\beta \left(|\u_1|^{r-1}\u_1-|\u_2|^{r-1}\u_2\right)\big\} \big],
		\end{align}
		in the weak sense. 
			\vskip 0.2 cm
		\noindent\textbf{Case II:} \emph{$d=2$ and $r \geq 1$.}
		Taking gradient on both sides in the equation \eqref{d} and then applying H\"older's inequality and Taylor's formula, we deduce 
		\begin{align*}
			\|\nabla p\|_{\H^{-1}} &\leq  C\left( \| (\u_1 \cdot \nabla)\u\|_{\V'}+\| (\u \cdot \nabla)\u_2\|_{\V'}+\left\| |\u_1|^{r-1}\u_1-|\u_2|^{r-1}\u_2\right\|_{\V'}\right) \\& \leq C\left( \| \u_1 \otimes\u\|_{\H}+\| \u \otimes \u_2\|_{\H}+\left\| |\u_1|^{r-1}\u_1-|\u_2|^{r-1}\u_2\right\|_{\wi{\L}^\frac{r+1}{r}}\right) \\& \leq C\left( \| \u_1\|_{\wi{\L}^4}+\| \u_2\|_{\wi{\L}^4}\right)\|\u\|_{\wi{\L}^4}+\|\u_1-\u_2\|_{\wi{\L}^{r+1}}\left(\| \u_1\|_{\wi{\L}^{r+1}}^{r-1}+\| \u_1\|_{\wi{\L}^{r+1}}^{r-1}\right),
		\end{align*}
		where we have used  the embedding $\V \subset \wi{\L}^{r+1} \subset \H \equiv \H' \subset \wi{\L}^\frac{r+1}{r} \subset \V'$. Taking the $(\frac{r+1}{r})^\text{th}$ power on both sides in the above relation and then integrating it from $0$ to $T$, followed by using H\"older's inequality, we arrive at
		\begin{align*}
		&	\int_{0}^{T}\|\nabla p(t)\|_{\H^{-1}}^\frac{r+1}{r} \d t \nonumber\\&\leq C\int_0^T\| \u_1(t)\|_{\widetilde{\L}^4}^\frac{r+1}{r}\|\u(t)\|_{\widetilde{\L}^4}^\frac{r+1}{r}\d t+C \int_0^T\| \u_2(t)\|_{\wi{\L}^4}^\frac{r+1}{r}\|\u(t)\|_{\wi{\L}^4}^\frac{r+1}{r}\d t\\&\quad+C\int_0^T\|\u_1(t)\|_{\wi{\L}^{r+1}}^\frac{(r-1)(r+1)}{r} \|\u(t)\|_{\wi{\L}^{r+1}}^\frac{r+1}{r}\d t\\&\quad+C\int_0^T\|\u_2(t)\|_{\wi{\L}^{r+1}}^\frac{(r-1)(r+1)}{r} \|\u(t)\|_{\wi{\L}^{r+1}}^\frac{r+1}{r}\d t
			\\& \leq CT^\frac{r-1}{2r}\bigg(\int_0^T\| \u(t)\|_{\widetilde{\L}^4}^4\d t\bigg)^\frac{r+1}{4r}\bigg[\bigg(\int_0^T\| \u_1(t)\|_{\widetilde{\L}^4}^{4}\d t\bigg)^\frac{r+1}{4r}
			+\bigg(\int_0^T\| \u_2(t)\|_{\widetilde{\L}^4}^4\d t\bigg)^\frac{r+1}{4r}\bigg]
			\\&\quad+C\bigg(\int_0^T\| \u(t)\|_{\widetilde{\L}^{r+1}}^{r+1}\d t\bigg)^\frac{1}{r}\bigg[\bigg(\int_0^T\| \u_1(t)\|_{\widetilde{\L}^{r+1}}^{r+1}\d t\bigg)^\frac{r-1}{r}+\bigg(\int_0^T\| \u_2(t)\|_{\widetilde{\L}^{r+1}}^{r+1}\d t\bigg)^\frac{r-1}{r}\bigg].
		\end{align*}
		Using the energy estimates obtained in Lemmas \ref{lemb1} and \ref{lemmaS1} in the above inequality, we obtain
		\begin{align*}
		\bigg(	\int_0^T\| \nabla p(t) \|_{\H^{-1}}^\frac{r+1}{r} \d t \bigg)^\frac{r}{r+1}	& \leq C\left(\|\u_0\|_{\H}+ \|g\|_0+\|\f\|_{\H} \right)+C\left(\|\u_0\|_{\H}^{\frac{2}{r+1}}+ \|g\|_0^{\frac{2}{r+1}}+\|\f\|_{\H}^{\frac{2}{r+1}} \right)
			\\& \leq C\left(\|\u_0\|_{\H}^{\frac{2}{r+1}}+ \|g\|_0^{\frac{2}{r+1}}+	\|g_t\|_0^{\frac{2}{r+1}}+\|\nabla \boldsymbol{\varphi}\|_{\H}^{\frac{2}{r+1}}+\|\nabla \psi-\mu \Delta \boldsymbol{\varphi}\|_{\L^2}^{\frac{2}{r+1}}\right),
		\end{align*}
		where $C$ depends on the input data,  $\mu,\alpha,\beta,r$ and $T$, and we have used Lemma  \ref{lemmaSf}.
%		Finally, we have
%		\begin{align*}
%			\| \nabla p \|_{\mathrm{L}^\frac{r+1}{r}(0,T;\H^{-1})}\leq C\big(\|\u_0\|_{\H}+ \|g\|_0+	\|g_t\|_0+\|\nabla \boldsymbol{\varphi}\|_{\H}+\|\nabla \psi-\mu \Delta \boldsymbol{\varphi}\|_{\L^2}\big).
%		\end{align*}
		\vskip 0.2 cm
		\noindent\textbf{Case II:} \emph{$d=3$ and $r \geq 3$.}
	The stability of the pressure gradient can be justified using  Taylor's formula, H\"older's and interpolation inequalities in \eqref{d}, and then energy estimates obtained in the Lemmas \ref{lemb1}, \ref{lemmaS1} and \ref{lemmaSf} as follows: 	For all $\Psi \in \mathrm{L}^2(0,T;\V) \cap\mathrm{L}^{r+1}(0,T;\wi{\L}^{r+1}) $ 
		\begin{align*}
		&	\bigg|\int_0^T \langle\nabla p(t),\Psi (t) \rangle \d t \bigg| \\&\leq \bigg|\int_0^T \langle (\u_1(t) \cdot\nabla) \u(t),\Psi (t) \rangle \d t \bigg|+\bigg|\int_0^T \langle (\u(t) \cdot\nabla) \u_2(t),\Psi (t) \rangle \d t \bigg| \\&\quad+\bigg|\int_0^T \big\langle (|\u_1(t)|^{r-1}\u_1(t)-|\u_2(t)|^{r-1}\u_2(t)),\Psi (t) \big\rangle \d t \bigg|
			\\& \leq \int_0^T \left(\| \u_1(t)\|_{\wi{\L}^\frac{2(r+1)}{r-1}}+\| \u_2(t)\|_{\wi{\L}^\frac{2(r+1)}{r-1}}\right)\|\u(t)\|_{\wi{\L}^{r+1}}\|\nabla\Psi(t)\|_{\H}\d t \\&\quad+ C\int_0^T\|(\u_1-\u_2)(t)\|_{\wi{\L}^{r+1}}\bigg(\| \u_1(t)\|_{\wi{\L}^{r+1}}+\| \u_2(t)\|_{\wi{\L}^{r+1}}\bigg)^{r-1}\| \Psi(t)\|_{\wi{\L}^{r+1}}
			\\& \leq \int_0^T \left(\| \u_1(t)\|_{\H}^\frac{r-3}{r-1}\| \u_1(t)\|_{\wi{\L}^{r+1}}^\frac{2}{r-1}+\| \u_2(t)\|_{\H}^\frac{r-3}{r-1}\| \u_2(t)\|_{\wi{\L}^{r+1}}^\frac{2}{r-1}\right)\|\u(t)\|_{\wi{\L}^{r+1}}\|\nabla\Psi(t)\|_{\H}\d t \\&\quad+ C\int_0^T\|(\u_1-\u_2)(t)\|_{\wi{\L}^{r+1}}\left(\| \u_1(t)\|_{\wi{\L}^{r+1}}^{r-1}+\| \u_2(t)\|_{\wi{\L}^{r+1}}^{r-1}\right)\| \Psi(t)\|_{\wi{\L}^{r+1}}
			\\& \leq C\bigg\{\bigg(\int_0^T\|\u_1(t)\|_{\H}^2 \d t\bigg)^\frac{r-3}{2(r-1)} \bigg(\int_0^T \big(\| \u_1(t)\|_{\wi{\L}^{r+1}}^{r+1} \d t\bigg)^\frac{2}{r^{2}-1} \bigg(\int_0^T \| \u(t)\|_{\wi{\L}^{r+1}}^{r+1}\big) \d t \bigg)^\frac{1}{r+1}
			\\&\quad+\bigg(\int_0^T\|\u_2(t)\|_{\H}^2 \d t\bigg)^\frac{r-3}{2(r-1)} \bigg(\int_0^T \big(\| \u_2(t)\|_{\wi{\L}^{r+1}}^{r+1} \d t\bigg)^\frac{2}{r^{2}-1} \bigg(\int_0^T \| \u(t)\|_{\wi{\L}^{r+1}}^{r+1}\big) \d t \bigg)^\frac{1}{r+1}\bigg\}
			\\&\qquad \times
			\bigg(\int_0^T \| \Psi(t)\|_{\V}^2 \d t \bigg)^\frac{1}{2}
			\\& \quad+C \bigg(\int_0^T \big(\| \u_1(t)\|_{\wi{\L}^{r+1}}^{r+1}+\| \u_2(t)\|_{\wi{\L}^{r+1}}^{r+1}\big) \d t \bigg)^\frac{r-1}{r+1} \bigg(\int_0^T \| \u(t)\|_{\wi{\L}^{r+1}}^{r+1}\big) \d t \bigg)^\frac{1}{r+1}\\&\quad\times \bigg(\int_0^T \| \Psi(t)\|_{\wi{\L}^{r+1}}^{r+1} \d t \bigg)^\frac{1}{r+1}
			\\& \leq C\bigg(\mu,\alpha, \beta,r,T,\|\u_{0i}\|_{\H},\|\f_i\|_{\H},\|g_i\|_{0},T\bigg)\bigg(\int_0^T \| \u(t)\|_{\wi{\L}^{r+1}}^{r+1}\big) \d t \bigg)^\frac{1}{r+1}
				\\& \qquad \times
			\bigg\{\bigg(\int_0^T \| \Psi(t)\|_{\V}^2 \d t \bigg)^\frac{1}{2}+\bigg(\int_0^T \| \Psi(t)\|_{\wi{\L}^{r+1}}^{r+1} \d t \bigg)^\frac{1}{r+1}\bigg\},
		\end{align*}
for $i=1,2$. Thus, it is immediate that $$\nabla p \in \mathrm{L}^2(0,T;\H^{-1})+\mathrm{L}^\frac{r+1}{r}\big(0,T;{\L}^\frac{r+1}{r}\big)\subset\mathrm{L}^\frac{r+1}{r}\big(0,T;\H^{-1}+{\L}^\frac{r+1}{r}\big).$$
	and 	\begin{align*}
		&\|\nabla p \|_{\mathrm{L}^\frac{r+1}{r}\big(0,T;\H^{-1}+{\L}^\frac{r+1}{r}\big)}\leq C\|\nabla( p_1-p_2) \|_{\mathrm{L}^2(0,T;\H^{-1})+\mathrm{L}^\frac{r+1}{r}\big(0,T;{\L}^\frac{r+1}{r}\big)}\nonumber\\&\leq C\left(\|\u_{01}-\u_{02}\|_{\H}^{\frac{2}{r+1}}+\|g_1-g_2\|_0^{\frac{2}{r+1}}+\|(g_1-g_2)_t\|_0^{\frac{2}{r+1}}\right.\nonumber\\&\quad\left.+\| \nabla (\boldsymbol{\varphi}_1-\boldsymbol{\varphi}_2)\|_{\H}^{\frac{2}{r+1}}+\| \nabla (\psi_1-\psi_2)-\mu \Delta(\boldsymbol{\varphi}_1-\boldsymbol{\varphi}_2)\|_{\L^2}^{\frac{2}{r+1}}\right).
	\end{align*}

From the stability estimate of the pressure gradient $\nabla p$ and Lemmas \ref{lemmaS1}-\ref{lemmaSf}, we can see that the solution depends continuously on the data. The H\"older type stability of the solution $\{\u,\nabla  p, \f\}$ given in \eqref{1k1} and \eqref{1k} can be established in a similar way. }
%\begin{itemize}
%		\item [(1)] for $d=2$ and $r \geq 1$, 
%	\begin{align*}
%		&\nonumber\|\u_1-\u_2\|_{\mathrm{L}^\infty(0,T;\H )}+	\|\u_1-\u_2\|_{\mathrm{L}^2(0,T;\V)}+\|\u_1-\u_2\|_{\mathrm{L}^{r+1}(0,T;\widetilde{\L}^{r+1})}\nonumber\\&\quad+\|\nabla( p_1-p_2) \|_{\mathrm{L}^\frac{r+1}{r}(0,T;\H^{-1})}+\|\f_1-\f_2\|_{\L^2}\nonumber\\&\leq C\big(\|\u_{01}-\u_{02}\|_{\H}+\|g_1-g_2\|_0+\|(g_1-g_2)_t\|_0\nonumber\\&\quad+\| \nabla (\boldsymbol{\varphi}_1-\boldsymbol{\varphi}_2)\|_{\H}+\| \nabla (\psi_1-\psi_2)-\mu \Delta(\boldsymbol{\varphi}_1-\boldsymbol{\varphi}_2)\|_{\L^2}\big),
%	\end{align*} 
%	\item [(2)] for $d=3$ and $r \geq 3$
%\begin{align*}
%	&\nonumber\|\u_1-\u_2\|_{\mathrm{L}^\infty(0,T;\H )}+	\|\u_1-\u_2\|_{\mathrm{L}^2(0,T;\V)}+\|\u_1-\u_2\|_{\mathrm{L}^{r+1}(0,T;\widetilde{\L}^{r+1})}\nonumber\\&\quad+\|\f_1-\f_2\|_{\L^2}+\|\nabla( p_1-p_2) \|_{\mathrm{L}^2(0,T;\H^{-1})+\mathrm{L}^\frac{r+1}{r}\big(0,T;{\L}^\frac{r+1}{r}\big)}\nonumber\\&\leq C\big(\|\u_{01}-\u_{02}\|_{\H}+\|g_1-g_2\|_0+\|(g_1-g_2)_t\|_0\nonumber\\&\quad+\| \nabla (\boldsymbol{\varphi}_1-\boldsymbol{\varphi}_2)\|_{\H}+\| \nabla (\psi_1-\psi_2)-\mu \Delta(\boldsymbol{\varphi}_1-\boldsymbol{\varphi}_2)\|_{\L^2}\big),
%\end{align*} 
%\end{itemize}
%which completes the proof of part (ii) of Theorem \ref{thm2}.
\end{proof}

	\begin{appendix}
	\renewcommand{\thesection}{\Alph{section}}
	\numberwithin{equation}{section}
		\section{Proof of Theorem \ref{thm1}}\label{sec5}\setcounter{equation}{0}
	\begin{proof}[Proof of Theorem \ref{thm1}]
		Let us assume that the nonlinear operator equation \eqref{1i} has a solution within the domain $\mathcal{D}$, denoted as $\f$. By substituting $\f$ into equation \eqref{1a}, we employ the system \eqref{1a}-\eqref{2a} to find a pair of functions $(\u,\nabla p)$ as the solution to the direct problem corresponding to the divergence-free external forcing function $\boldsymbol{F}(x,t)=\f(x)g(x,t)$.
		
		To establish that $\u$ and $\nabla p$ satisfy the overdetermination  condition \eqref{1d}, we consider $$\u(x,T)=\boldsymbol{\varphi}_1(x) \ \ \ \ \text{and} \ \ \ \ \nabla p(x,T)=\nabla\psi_1(x),  \ \ \ \ x \ \in \ \Omega,$$
		and
		\begin{align}\label{3a}
			\u_t(x,T)-\mu \Delta \boldsymbol{\varphi}_1+(\boldsymbol{\varphi}_1\cdot\nabla)\boldsymbol{\varphi}_1+\alpha \boldsymbol{\varphi}_1+\beta|\boldsymbol{\varphi}_1|^{r-1}\boldsymbol{\varphi}_1+\nabla \psi_1=\boldsymbol{f}g(x,T).
		\end{align}
		On the other hand, \eqref{1i} implies that
		\begin{align}\label{3b}
			\mathcal{A}\f-\mu \Delta \boldsymbol{\varphi}+(\boldsymbol{\varphi}\cdot\nabla)\boldsymbol{\varphi}+\alpha \boldsymbol{\varphi}+\beta|\boldsymbol{\varphi}|^{r-1}\boldsymbol{\varphi}+\nabla \psi=\boldsymbol{f}g(x,T).
		\end{align}
	By employing \eqref{1h} in  \eqref{3b} and then combining \eqref{3a} and \eqref{3b}, we can readily confirm that the functions $(\boldsymbol{\varphi}-\boldsymbol{\varphi}_1)$ and $\nabla(\psi-\psi_1)$ satisfy the following system of equations:
		\begin{equation}\label{3c}
			\left\{
			\begin{aligned}
				-\mu \Delta (\boldsymbol{\varphi}-\boldsymbol{\varphi}_1)+\alpha(\boldsymbol{\varphi}-\boldsymbol{\varphi}_1)+((\boldsymbol{\varphi}-\boldsymbol{\varphi}_1)\cdot\nabla)\boldsymbol{\varphi}+((\boldsymbol{\varphi}_1&\cdot\nabla)(\boldsymbol{\varphi}-\boldsymbol{\varphi}_1)\\+\beta(|\boldsymbol{\varphi}|^{r-1}\boldsymbol{\varphi}-|\boldsymbol{\varphi}_1|^{r-1}\boldsymbol{\varphi}_1)+\nabla (\psi-\psi_1)&=\boldsymbol{0}, \ \ \text{in} \ \Omega, \\ \nabla \cdot  (\boldsymbol{\varphi}-\boldsymbol{\varphi}_1)&=0, \ \  \text{in} \ \Omega, \\
				\boldsymbol{\varphi}-\boldsymbol{\varphi}_1&=\boldsymbol{0}, \ \  \text{on} \ \partial\Omega.
			\end{aligned}
			\right.
		\end{equation}
		Multiplying with $(\boldsymbol{\varphi}-\boldsymbol{\varphi}_1)$ to the first equation in \eqref{3c}, we get
		\begin{align}\label{3d}
			\mu\|\nabla (\boldsymbol{\varphi}-\boldsymbol{\varphi}_1)\|_{\H}^2+\alpha \| \boldsymbol{\varphi}-\boldsymbol{\varphi}_1\|_{\H}^2&=-(((\boldsymbol{\varphi}-\boldsymbol{\varphi}_1)\cdot\nabla)\boldsymbol{\varphi},\boldsymbol{\varphi}-\boldsymbol{\varphi}_1)\nonumber\\&\quad-\beta(|\boldsymbol{\varphi}|^{r-1}\boldsymbol{\varphi}-|\boldsymbol{\varphi}_1|^{r-1}\boldsymbol{\varphi}_1,\boldsymbol{\varphi}-\boldsymbol{\varphi}_1).
		\end{align}
		For $r\geq 1$, we have (see Section 2.4 \cite{MTM4,MTM6})
		\begin{align}
			\beta\left(|\boldsymbol{\varphi}|^{r-1}\boldsymbol{\varphi}-|\boldsymbol{\varphi}_1|^{r-1}\boldsymbol{\varphi}_1,\boldsymbol{\varphi}-\boldsymbol{\varphi}_1\right)&\geq \frac{\beta}{2}\||\boldsymbol{\varphi}|^\frac{r-1}{2}(\boldsymbol{\varphi}-\boldsymbol{\varphi}_1)\|_{\H}^2+\frac{\beta}{2}\||\boldsymbol{\varphi}_1|^\frac{r-1}{2}(\boldsymbol{\varphi}-\boldsymbol{\varphi}_1)\|_{\H}^2  \label{3e} \\&\geq\frac{\beta}{2}\||\boldsymbol{\varphi}|^\frac{r-1}{2}(\boldsymbol{\varphi}-\boldsymbol{\varphi}_1)\|_{\H}^2 \geq0.\label{3e1}
		\end{align}
		\vskip 0.2 cm
		\noindent\textbf{Case I:} \emph{$d=2$ and $r\in[1,3]$.}
		Applying H\"older's, Ladyzhenskaya's and Young's inequalities, we obtain
		\begin{align}\label{3f}
			\nonumber	\big|\big(((\boldsymbol{\varphi}-\boldsymbol{\varphi}_1)\cdot\nabla)\boldsymbol{\varphi},\boldsymbol{\varphi}-\boldsymbol{\varphi}_1\big)\big| &\leq\|\nabla (\boldsymbol{\varphi}-\boldsymbol{\varphi}_1)\|_{\H}\| \boldsymbol{\varphi}-\boldsymbol{\varphi}_1\|_{\widetilde{\L}^4}\| \boldsymbol{\varphi}\|_{\widetilde{\L}^4} \nonumber\\&\leq (2)^\frac{1}{4}\|\nabla (\boldsymbol{\varphi}-\boldsymbol{\varphi}_1)\|_{\H}^\frac{3}{2}\| \boldsymbol{\varphi}-\boldsymbol{\varphi}_1\|_{\H}^\frac{1}{2}\| \boldsymbol{\varphi}\|_{\widetilde{\L}^4}
			\nonumber\\&\leq \frac{\alpha}{2}\| \boldsymbol{\varphi}-\boldsymbol{\varphi}_1\|_{\H}^2+\frac{3}{4 \alpha^\frac{1}{3}}\|\nabla (\boldsymbol{\varphi}-\boldsymbol{\varphi}_1)\|_{\H}^2\| \boldsymbol{\varphi}\|_{\widetilde{\L}^4}^\frac{4}{3}.
		\end{align}
		Substituting \eqref{3e1} and \eqref{3f} in \eqref{3d}, we obtain
		\begin{align}\label{a8}
			\bigg(\mu-\frac{3}{4 \alpha^\frac{1}{3}}\| \boldsymbol{\varphi}\|_{\widetilde{\L}^4}^\frac{4}{3}\bigg)\|\nabla(\boldsymbol{\varphi}-\boldsymbol{\varphi}_1)\|_{\H}^2+\frac{\alpha}{2}\|\boldsymbol{\varphi}-\boldsymbol{\varphi}_1\|_{\H}^2\leq0.
		\end{align}
		If $\frac{3}{4 \alpha^\frac{1}{3}}\| \boldsymbol{\varphi}\|_{\widetilde{\L}^4}^\frac{4}{3}\leq  \mu$ and $\alpha>0$, then $\|\nabla(\boldsymbol{\varphi}-\boldsymbol{\varphi}_1)\|_{\H}=0$ and $\|\boldsymbol{\varphi}-\boldsymbol{\varphi}_1\|_{\H}=0$, so that  $\boldsymbol{\varphi}=\boldsymbol{\varphi}_1$ and $\nabla\psi=\nabla\psi_1$, by using \eqref{3c}.
		\vskip 0.2 cm
		\noindent\textbf{Case II:} 	\emph{$d=2,3$ and $r>3$.}
		A calculation similar to \eqref{n22} gives
		\begin{align}\label{3j}
			\big|\big(((\boldsymbol{\varphi}-\boldsymbol{\varphi}_1)\cdot\nabla)(\boldsymbol{\varphi}-\boldsymbol{\varphi}_1),\boldsymbol{\varphi}\big)\big|& \leq\|\nabla (\boldsymbol{\varphi}-\boldsymbol{\varphi}_1)\|_{\H}\|\boldsymbol{\varphi} (\boldsymbol{\varphi}-\boldsymbol{\varphi}_1)\|_{\H}
			\no \\& \leq \frac{\mu}{2}\|\nabla (\boldsymbol{\varphi}-\boldsymbol{\varphi}_1)\|_{\H}^2+\frac{\beta}{2}\||\boldsymbol{\varphi}|^\frac{r-1}{2}(\boldsymbol{\varphi}-\boldsymbol{\varphi}_1)\|_{\H}^2\no \\&\quad +\frac{r-3}{2 \mu(r-1)}\left(\frac{2}{\beta\mu (r-1)}\right)^{\frac{2}{r-3}}\|\boldsymbol{\varphi}-\boldsymbol{\varphi}_1\|_{\H}^2,
		\end{align}
	for $r>3$. 
		Substituting \eqref{3e1} and \eqref{3j} in \eqref{3d}, we deduce
		\begin{align*}
			&\frac{\mu }{2}\|\nabla(\boldsymbol{\varphi}-\boldsymbol{\varphi}_1)\|_{\H}^2+\bigg(\alpha-\frac{r-3}{2\mu(r-1)}\left(\frac{2}{\beta\mu (r-1)}\right)^{\frac{2}{r-3}}\bigg)\|\boldsymbol{\varphi}-\boldsymbol{\varphi}_1\|_{\H}^2 \leq 0.
		\end{align*}
		If $\frac{r-3}{2\mu(r-1)}\left(\frac{2}{\beta\mu (r-1)}\right)^{\frac{2}{r-3}}<\alpha$ and $\mu >0$, then $\|\boldsymbol{\varphi}-\boldsymbol{\varphi}_1\|_{\H}=0$, so that
		$\boldsymbol{\varphi}=\boldsymbol{\varphi}_1$ and $\nabla\psi=\nabla\psi_1$.
		\vskip 0.3cm
		\noindent
		\vskip 0.2 cm
		\noindent\textbf{Case III:} \emph{$d=r=3$.} From \eqref{3e1}, we obtain
		\begin{align}\label{3k}
			\beta(|\boldsymbol{\varphi}|^2\boldsymbol{\varphi}-|\boldsymbol{\varphi}_1|^2\boldsymbol{\varphi}_1,\boldsymbol{\varphi}-\boldsymbol{\varphi}_1)\geq \frac{\beta}{2}\||\boldsymbol{\varphi}|(\boldsymbol{\varphi}-\boldsymbol{\varphi}_1)\|_{\H}^2,
		\end{align}
		 \begin{align}\label{3l}
		\text{and} \ \ \ \ 	\nonumber\big|\big(((\boldsymbol{\varphi}-\boldsymbol{\varphi}_1)\cdot\nabla)(\boldsymbol{\varphi}-\boldsymbol{\varphi}_1),\boldsymbol{\varphi}\big)\big| &\leq\|\nabla (\boldsymbol{\varphi}-\boldsymbol{\varphi}_1)\|_{\H}\|\boldsymbol{\varphi} (\boldsymbol{\varphi}-\boldsymbol{\varphi}_1)\|_{\H}\\&\leq\frac{1}{2\beta}\|\nabla (\boldsymbol{\varphi}-\boldsymbol{\varphi}_1)\|_{\H}^2+\frac{\beta}{2}\|\boldsymbol{\varphi} (\boldsymbol{\varphi}-\boldsymbol{\varphi}_1)\|_{\H}^2.
		\end{align}
		Gathering the estimates \eqref{3k} and \eqref{3l}, and then substituting it in \eqref{3d}, we get
		\begin{align}
			\left(\mu-\frac{1}{2\beta}\right)\| \nabla(\boldsymbol{\varphi}-\boldsymbol{\varphi}_1)\|_{\H}^2 +\alpha\| \boldsymbol{\varphi}-\boldsymbol{\varphi}_1\|_{\H}^2\leq 0.
		\end{align}
		For $2\beta\mu\geq 1$ and $\alpha>0$, we get  $\|\nabla(\boldsymbol{\varphi}-\boldsymbol{\varphi}_1)\|_{\H}=0$ and $\|\boldsymbol{\varphi}-\boldsymbol{\varphi}_1\|_{\H}=0$, and hence, we find
		$\boldsymbol{\varphi}=\boldsymbol{\varphi}_1$ and $\nabla\psi=\nabla\psi_1$.
		\vskip 0.3cm
		On the other hand, let us suppose that the inverse problem \eqref{1a}-\eqref{1d} admits a solution, denoted as $\{\u,\nabla p,\f\}$. Utilizing the elliptic regularity of the solution, such as the Cattabriga regularity theorem (see \cite{LCa,Te1}, etc. or Corollary 2, \cite{KT2}) (see Appendix \ref{sec6} also for regularity results on the solution), when we evaluate the system \eqref{1a} at $t=T$, we obtain
		\begin{align}
			\nonumber &\u_t(x,T)-\mu \Delta\u(x,T)+(\u(x,T)\cdot\nabla)\u(x,T)+\alpha \u(x,T)\\&\quad+\beta|\u(x,T)|^{r-1}\u(x,T)+\nabla p(x,T)=\boldsymbol{f}(x)g(x,T).
		\end{align}
		From the given final overdetermination data \eqref{1d} and the definition of the operator $\mathcal{A}$, it follows that
		\begin{align*}
			\nonumber\mathcal{A}\f&-\mu \Delta \boldsymbol{\varphi}+(\boldsymbol{\varphi}\cdot\nabla)\boldsymbol{\varphi}+\alpha \boldsymbol{\varphi}+ \beta|\boldsymbol{\varphi}|^{r-1}\boldsymbol{\varphi}+\nabla \psi=\boldsymbol{f}g(x,T),
		\end{align*}
	and hence 
		\begin{align*}
			\f=\frac{1}{g(x,T)}\left(\mathcal{A}\f-\mu \Delta \boldsymbol{\varphi}+(\boldsymbol{\varphi}\cdot\nabla)\boldsymbol{\varphi}+\alpha \boldsymbol{\varphi}+\beta|\boldsymbol{\varphi}|^{r-1}\boldsymbol{\varphi}+\nabla \psi\right)=\mathcal{B}\f.
		\end{align*}
	Therefore, we can conclude that the nonlinear operator equation \eqref{1i} possesses a solution, thereby completing the proof.
	\end{proof}
	
	\section{A-priori Energy Estimates}\label{sec6}\setcounter{equation}{0}
	In this section, we obtain a number of \emph{a-priori} estimates  for the solutions of the system   \eqref{1a}-\eqref{2a}, since the existence, uniqueness  and regularity of generalized solutions for the system \eqref{1a}-\eqref{2a} are known (cf. \cite{KWH,KT2,PAM,MTM4}, etc.). These estimates are necessary  to prove the existence, uniqueness, and stability of the solution of our inverse problem (proof of Theorem \ref{thm2}, see sections \ref{sec4} and \ref{sec4b} above). To obtain the energy estimates for the solutions of the CBF equations \eqref{1a}-\eqref{2a}, we assume that $\u_0 \in \H,$ $g, g_t \in \C(\T \times[0,T])$ satisfy assumption \eqref{1e} and $\f \in \H$.
	
	%	First, we show that $\u_t\in\C([\epsilon,T];\V)$ for any $0<\epsilon\leq T$. For that purpose, we need some supporting Lemmas. 
	\begin{lemma}\label{lemb1}
		Let $(\u(\cdot),\nabla p(\cdot))$ be the unique solution of the CBF  equations \eqref{1a}-\eqref{2a} and $\u_0 \in \H$. Then, for $r \geq 1$, the following estimate holds for all $t \in [0,T]$:
		\begin{align}\label{bb1}
			&\sup_{s\in[0,t]}\|\u(s)\|_{\H}^2+2\mu \int_0^t\|\nabla\u(s)\|_{\H}^2\d s+ \alpha\int_0^t\|\u(s)\|_{\H}^2\d s+2\beta\int_0^t\|\u(s)\|_{\widetilde{\L}^{r+1}}^{r+1}\d s \no \\ &\leq \|\u_0\|_{\H}^2+\frac{t}{\alpha}\|g\|_0^2\|\f\|_{\H}^2.
		\end{align}
	\end{lemma} 
	\begin{proof}
Taking the inner product with $ \u(\cdot)$ in \eqref{1a} and integrating the resulting equation over $\T$, we obtain 
	\begin{align}\label{bb2}
			\frac{1}{2}\frac{\d}{\d t}\|\u(t)\|_{\H}^2+\mu \|\nabla\u(t)\|_{\H}^2+\alpha \|\u(t)\|_{\H}^2+\beta\|\u(t)\|_{\widetilde{\L}^{r+1}}^{r+1}=(\f g(t),\u(t)),
	\end{align}
	for a.e. $t\in[0,T]$. Using H\"older's and Young's inequalities, we have
	\begin{align}\label{bb3}
		|(\f g,\u)| \leq \|g\|_{\mathrm{L}^{\infty}}\|\f\|_{\H}\|\u\|_{\H}\leq \frac{\alpha}{2}\|\u\|_{\H}^2+\frac{1}{2 \alpha}\|g\|_{\mathrm{L}^{\infty}}^2\|\f\|_{\H}^2.
	\end{align}
Substituting \eqref{bb3} in \eqref{bb2} and integrating it from $0$ to $t$, we deduce the estimate \eqref{bb1}.
\iffalse
\begin{align*}
	&\|\u(t)\|_{\H}^2+2\mu \int_0^t \|\nabla\u(s)\|_{\H}^2\d s+\alpha \int_0^t \|\u(s)\|_{\H}^2 \d s+2\beta \int_0^t\|\u(s)\|_{\widetilde{\L}^{r+1}}^{r+1}\d s\\& \leq
	\|\u_0\|_{\H}^2+\frac{t}{ \alpha}\|g\|_{0}^2\|\f\|_{\H}^2,
\end{align*}
for all $t \in [0,T]$ and \eqref{bb1} follows.
\fi
	\end{proof}

	\begin{lemma}\label{lemb2}
		Let $(\u(\cdot),\nabla p(\cdot))$ be the unique solution of the CBF  equations \eqref{1a}-\eqref{2a} and $\u_0 \in \H$.  Then, for all $t\in[\epsilon,T]$ and for any $0<\epsilon<T$,
			\begin{itemize}
			\item [$(i)$] 	for $d=2$ and $r \in[1,3],$   we have
		\begin{align}\label{bb4}
			&	\sup_{s\in[\epsilon,t]}\|\nabla\u(s)\|_{\H}^2 + \mu \int_{\epsilon}^t\|\Delta\u(s)\|_{\H}^2\d s+2\beta \int_{\epsilon}^t \big\||\u(s)|^\frac{r-1}{2}|\nabla \u(s)|\big\|_{\H}^2\d s
			\nonumber\\&\leq \frac{1}{2\mu}
			\bigg\{ \frac{1}{t}\|\u_0\|_{\H}^2+\bigg(\frac{1}{\alpha}+t\bigg)\|g\|_0^2\|\f\|_{\H}^2\bigg\},
		\end{align}
		\item [$(ii)$] 	for $d=2,3$ and $r >3,$ we have
	\begin{align}\label{bb5}
		&	\sup_{s\in[\epsilon,t]}\|\nabla\u(s)\|_{\H}^2 + \mu \int_{\epsilon}^t\|\Delta\u(s)\|_{\H}^2\d s+\beta \int_{\epsilon}^t \big\||\u(s)|^\frac{r-1}{2}|\nabla \u(s)|\big\|_{\H}^2\d s
		\nonumber\\&\leq \bigg(\frac{1}{2\mu t}+\frac{\eta}{2\mu}\bigg)
	\|\u_0\|_{\H}^2+\bigg(\frac{1}{2\mu \alpha}+\frac{t}{\mu}+\frac{\eta  t}{2\mu \alpha}\bigg)\|g\|_0^2\|\f\|_{\H}^2,
	\end{align}
		\item [$(iii)$] 	for $d=r=3$ with $\beta \mu > 1,$   we have
	\begin{align}\label{bb6}
		&	\sup_{s\in[\epsilon,t]}\|\nabla\u(s)\|_{\H}^2 + \mu \int_{\epsilon}^t\|\Delta\u(s)\|_{\H}^2\d s+2\bigg(\beta-\frac{1}{\mu}\bigg) \int_{\epsilon}^t \big\||\u(s)||\nabla \u(s)|\big\|_{\H}^2\d s
		\nonumber\\&\leq \frac{1}{2\mu}
		\bigg\{ \frac{1}{t}\|\u_0\|_{\H}^2+\bigg(\frac{1}{\alpha}+2t\bigg)\|g\|_0^2\|\f\|_{\H}^2\bigg\}.
	\end{align}
\end{itemize}
	\end{lemma}
	\begin{proof} 
		Taking the inner product with $- \Delta \u(\cdot)$ in \eqref{1a} and integrating it over $\T$, we obtain 
		\begin{align}\label{bb7}
			\nonumber&\frac{1}{2} \frac{\d}{\d t}\|\nabla\u(t)\|_{\H}^2+\mu\|\Delta\u(t)\|_{\H}^2+\alpha \|\nabla\u(t)\|_{\H}^2+\beta \big\| |\u(t)|^{\frac{r-1}{2}} |\nabla \u(t)| \big\|_{\H}^2\nonumber\\&=(\f g(t),-\Delta \u(t))-((\u(t) \cdot\nabla)\u(t),-\Delta\u(t)),
		\end{align}
	for a.e. $t \in [\epsilon,T],$ for some $0<\epsilon< T$, where we have used the fact that
	$\langle \nabla p,-  \Delta \u \rangle=0$ (see Lemma \ref{lemmabe2}).
		\vskip 0.2 cm
		\noindent\textbf{Case I:} 	\emph{$d=2$ and $r \in [1,3]$.}
		Substituting the estimates \eqref{n17} and \eqref{n18} in \eqref{bb7}, and integrating the resulting estimate from $\epsilon>0$ to $t$, we arrive at  
		\begin{align}\label{bb8}
			\nonumber&\|\nabla \u(t)\|_{\H}^2+\mu\int_{\epsilon}^t\|\Delta \u(s)\|_{\H}^2\d s+2\beta  \int_{\epsilon}^t\big\||\u(s)|^\frac{r-1}{2}|\nabla \u(s)|\big\|_{\H}^2\d s\nonumber\\&\leq  \|\nabla\u(\epsilon)\|_{\H}^2+\frac{t-\epsilon}{\mu}\|g\|_0^2\|\f\|_{\H}^2,
		\end{align}
		for all $t\in[\epsilon,T]$.	From the above relation, we have
		\begin{align*}
			\|\nabla \u(t)\|_{\H}^2 \leq  \|\nabla\u(\epsilon)\|_{\H}^2+\frac{t-\epsilon}{\mu}\|g\|_0^2\|\f\|_{\H}^2.
		\end{align*}
		Integrating the above estimate over $\epsilon$ from $0$ to $t$, we deduce 	for all $t\in[\epsilon,T]$
		\begin{align*}
			\|\nabla \u(t)\|_{\H}^2
			&\leq \frac{1}{t}\bigg( \int_{0}^{t}\|\nabla\u(\epsilon)\|_{\H}^2 \d \epsilon+\|g\|_0^2\|\f\|_{\H}^2\int_{0}^t\frac{t-\epsilon}{\mu}\d \epsilon\bigg),
		\end{align*}
	which leads to \eqref{bb4} from \eqref{bb8} by using \eqref{bb1}.
		\vskip 0.2 cm
		\noindent\textbf{Case II:} 	\emph{$d=3$ and $r > 3$.} 
		Substituting the  estimates \eqref{n20} and \eqref{n22} in \eqref{bb7}, and integrating it from $\epsilon$ to $t$, we have the following estimate:  
		\begin{align}\label{bb9} 
			\nonumber&\|\nabla \u(t)\|_{\H}^2+\mu\int_{\epsilon}^t\|\Delta \u(s)\|_{\H}^2\d s+\beta  \int_{\epsilon}^t \big\||\u(s)|^\frac{r-1}{2}|\nabla\u(s)|\big\|_{\H}^2 \d  s\nonumber\\&\leq  \|\nabla\u(\epsilon)\|_{\H}^2+\frac{2(t-\epsilon)}{\mu}\|g\|_0^2\|\f\|_{\H}^2+ \eta\int_{\epsilon}^t\|\nabla \u(s)\|_{\H}^2 \d s,
		\end{align}
		for all $t\in[\epsilon,T]$,  where $\eta=\frac{2(r-3)}{\mu(r-1)}\left(\frac{4}{\beta \mu  (r-1)}\right)^\frac{2}{r-3}$. Using the estimate \eqref{bb1} in \eqref{bb9}, we get
		\begin{align*}
			\|\nabla \u(t)\|_{\H}^2 \leq  \|\nabla\u(\epsilon)\|_{\H}^2+\frac{2(t-\epsilon)}{\mu}\|g\|_0^2\|\f\|_{\H}^2+\frac{\eta}{2\mu}\bigg(\|\u_0\|_{\H}^2+\frac{t}{\alpha}\|g\|_0^2\|\f\|_{\H}^2\bigg),
		\end{align*}
		for all $t\in[\epsilon,T]$. Integrating the above estimate over $\epsilon$ from $0$ to $t$ and then using \eqref{bb1} in it, we find
		\begin{align*}
			\|\nabla \u(t)\|_{\H}^2
			&\leq \frac{1}{t}\bigg\{ \int_{0}^{t}\|\nabla\u(\epsilon)\|_{\H}^2 \d \epsilon+\frac{t^2}{\mu}\|g\|_0^2\|\f\|_{\H}^2+\frac{\eta t}{2\mu}\bigg(\|\u_0\|_{\H}^2+\frac{t}{\alpha}\|g\|_0^2\|\f\|_{\H}^2\bigg)\bigg\},
		\end{align*}
		and, from \eqref{bb9}, one can arrive at \eqref{bb5} by using \eqref{bb1}.
		\vskip 0.2 cm
		\noindent\textbf{Case III:} \emph{$d=r=3$ with $\beta \mu > 1$.} The  term $\beta  \big\||\u(t)|^\frac{r-1}{2}|\nabla\u(t)|\big\|_{\H}^2$ in \eqref{bb7} becomes $\beta  \big\||\u(t)| |\nabla\u(t)|\big\|_{\H}^2$.	Substituting the  estimates \eqref{n20}-\eqref{n21} in \eqref{bb7}, and integrating it from $\epsilon$ to $t$, we obtain  
		\begin{align}\label{bb10} 
			\nonumber&\|\nabla \u(t)\|_{\H}^2+\mu\int_{\epsilon}^t\|\Delta \u(s)\|_{\H}^2\d s +2\bigg(\beta -\frac{1}{\mu}\bigg)\int_{\epsilon}^t  \big\||\u(s)| |\nabla\u(s)|\big\|_{\H}^2 \d s\nonumber\\&\leq  \|\nabla\u(\epsilon)\|_{\H}^2+\frac{2(t-\epsilon)}{\mu }\|g\|_0^2\|\f\|_{\H}^2,
		\end{align}
		for all $t\in[\epsilon,T]$. A calculation similar to $d=2$ and $r \in [1,3]$ yields the estimate \eqref{bb6}, provided $\beta\mu > 1$.
	\end{proof}
	\begin{lemma}\label{lemb3}
		Let $(\u(\cdot),\nabla p(\cdot))$ be the unique solution of the CBF  equations \eqref{1a}-\eqref{2a} and $\u_0 \in \H$.  Then, for all $t\in[\epsilon,T]$ and for any $0<\epsilon<T$,
		\begin{itemize}
			\item [$(i)$] 	for $d=2$ and $r \in[1,3],$  we have
		\begin{align}\label{bb11}
			 \int_{\epsilon}^t\|\u_t(t)\|_{\H}^2\d t
			&\leq 
		 \frac{N_{11}}{t}\|\u_0\|_{\H}^2+(N_{12}+t)\|g\|_0^2\|\f\|_{\H}^2 \no \\&\quad+\frac{C}{2 \mu^3}\bigg\{\frac{1}{t^2}\|\u_0\|_{\H}^4+2\bigg(\frac{1}{\alpha^2}+t^2\bigg)	\|g\|_0^4\|\f\|_{\H}^4\bigg\},
		\end{align}
	where
	\begin{align*}
		N_{11}=\frac{3}{2}+\frac{1}{r+1} \ \ \ \text{and} \ \ \ N_{12}=\frac{N_{11}}{\alpha},
	\end{align*}
\item [$(ii)$] 	for $d=2,3$ and $r >3,$  we have
		\begin{align}\label{bb12}
			\int_{\epsilon}^t\|\u_t(s)\|_{\H}^2\d s \leq \bigg(\frac{N_{21}}{t}+N_{22}\bigg)\|\u_0\|_{\H}^2+\bigg(N_{23}t+N_{24}\bigg)\|g\|_0^2\|\f\|_{\H}^2,
		\end{align}
	where
	\begin{align*}
		&N_{21}=\frac{3}{2}+\frac{1}{r+1}+\frac{1}{2\mu}, \ \ \ N_{22}=\frac{\gamma^*+\eta}{2 \mu},\ \gamma^*=\frac{2(r-3)}{r-1}\left(\frac{4}{\beta(r-1)}\right)^\frac{2}{r-3}, \\&
		N_{23}=1+\frac{1}{\mu}+\frac{\gamma^*+\eta}{2\mu \alpha} \ \ \ \text{and} \ \ \ N_{24}=\frac{N_{21}}{\alpha},
	\end{align*}
\item [$(iii)$] 	for $d=r=3$ with $\beta \mu >1,$  we have
	\begin{align}\label{bb13}
	\int_{\epsilon}^t\|\u_t(s)\|_{\H}^2\d s \leq \frac{N_{31}}{t}\|\u_0\|_{\H}^2+\bigg(N_{32}t+N_{33}\bigg)\|g\|_0^2\|\f\|_{\H}^2,
\end{align}
where
\begin{align*}
	N_{31}=\frac{7}{4}+\frac{1}{2(\beta\mu-1)}, \ \ \ N_{32}=1+\frac{1}{\beta\mu-1} 
 \ \ \ \text{and} \ \ \  N_{33}=\frac{N_{31}}{\alpha}.
\end{align*}
\end{itemize}
	\end{lemma}
	\begin{proof} 
		Taking the inner product with $\u_t(\cdot)$ in \eqref{1a} and integrating the resulting equation over $\T$, we arrive at 
		\begin{align}\label{bb14}
			\nonumber&\|\u_t(t)\|_{\H}^2+\frac{\mu}{2} \frac{\d}{\d t}\|\nabla\u(t)\|_{\H}^2+\frac{\alpha }{2}\frac{\d}{\d t}\|\u(t)\|_{\H}^2+\frac{\beta}{r+1}\frac{\d}{\d t}\|\u(t)\|_{\widetilde{\L}^{r+1}}^{r+1}\nonumber\\&=(\f g(t),\u_t(t))-((\u(t) \cdot\nabla)\u(t),\u_t(t)),
		\end{align}
	for a.e. $t\in[\epsilon,T]$, where we have used the fact that $\langle\nabla p,\u_t\rangle=0$. 
	\vskip 0.2 cm
	\noindent\textbf{Case I:} 	\emph{$d=2$ and $r \in [1,3]$.} 
		Substituting the estimates \eqref{n36} and \eqref{n37} in \eqref{bb14} and integrating it from $\epsilon$ to $t$, we obtain  
		\begin{align}\label{bb15}
			\nonumber&\int_{\epsilon}^t\|\u_t(s)\|_{\H}^2\d s+\mu\|\nabla \u(t)\|_{\H}^2+\alpha\|\u(t)\|_{\H}^2+\frac{2\beta}{r+1}\|\u(t)\|_{\widetilde{\L}^{r+1}}^{r+1}
			\nonumber\\&\leq \mu \|\nabla\u(\epsilon)\|_{\H}^2+\alpha\|\u(\epsilon)\|_{\H}^2+\frac{2\beta}{r+1}\|\u(\epsilon)\|_{\widetilde{\L}^{r+1}}^{r+1}+2(t-\epsilon)\|g\|_0^2\|\f\|_{\H}^2\nonumber\\&\quad+C\sup_{s\in[\epsilon,t]}\left(\| \u(s)\|_{\H}^\frac{2}{3}\|\nabla \u(s)\|_{\H}^\frac{4}{3}\right)\int_{\epsilon}^t\|\Delta\u(s)\|_{\H}^2 \d s,
		\end{align}
		for all $t\in[\epsilon,T]$. Using the estimate \eqref{bb4} in \eqref{bb15}, we immediately get
		\begin{align*}
			\mu\|\nabla \u(t)\|_{\H}^2 	&\leq \mu \|\nabla\u(\epsilon)\|_{\H}^2+\alpha\|\u(\epsilon)\|_{\H}^2+\frac{2\beta}{r+1}\|\u(\epsilon)\|_{\widetilde{\L}^{r+1}}^{r+1}\nonumber\\&\quad+2(t-\epsilon)\|g\|_0^2\|\f\|_{\H}^2+\frac{C}{2 \mu^3}\bigg\{\frac{1}{t^2}\|\u_0\|_{\H}^4+2\bigg(\frac{1}{\alpha^2}+t^2\bigg)	\|g\|_0^4\|\f\|_{\H}^4\bigg\}.
		\end{align*}
		Integrating the above estimate over $\epsilon$ from $0$ to $t$, we obtain
		\begin{align*}
			\mu\|\nabla \u(t)\|_{\H}^2
			&	\leq \frac{1}{ t}  \bigg[\mu \int_{0}^{t}\|\nabla\u(\epsilon)\|_{\H}^2 \d \epsilon+\alpha \int_{0}^{t}\|\u(\epsilon)\|_{\H}^2 \d \epsilon+\frac{2\beta}{r+1}\int_0^t\|\u(\epsilon)\|_{\widetilde{\L}^{r+1}}^{r+1}\d \epsilon\\&\quad
			+t^2\|g\|_0^2\|\f\|_{\H}^2+\frac{C t}{2 \mu^3}\bigg\{\frac{1}{t^2}\|\u_0\|_{\H}^4+2\bigg(\frac{1}{\alpha^2}+t^2\bigg)	\|g\|_0^4\|\f\|_{\H}^4\bigg\}\bigg].
		\end{align*}
		Thus, from \eqref{bb15}, we immediately derive \eqref{bb11} by using \eqref{bb1}.
		\vskip 0.2 cm
		\noindent\textbf{Case II:} 	\emph{$d=2,3$ and $r > 3$.} Using H\"older's and Young's inequalities, we have
		\begin{align}\label{bb16}
			|(\f g,\u_t)| \leq \|g\|_{\mathrm{L}^\infty}\|\f\|_{\H} \|\u_t\|_{\H} \leq \frac{1}{4} \|\u_t\|_{\H}^2+\|g\|_{\mathrm{L}^\infty}^2\|\f\|_{\H}^2.
		\end{align}
		An estimate similar to \eqref{n22} yields
		\begin{align}\label{bb17}
			|	\langle(\u \cdot \nabla) \u, \u_t \rangle|
			\leq  \frac{1}{4}\|\u_t\|_{\H}^2+\frac{\beta }{2}\big\||\u|^\frac{r-1}{2}|\nabla\u| \big\|_{\H}^2+\frac{\gamma^*}{2}\|\nabla\u\|_{\H}^2,
		\end{align}
		where $\gamma^*=\frac{2(r-3)}{r-1}\left(\frac{4}{\beta(r-1)}\right)^\frac{2}{r-3}$. Substituting the estimates \eqref{bb16} and \eqref{bb17} in \eqref{bb14} and integrating it from $\epsilon$ to $t$, we deduce 
			\begin{align}\label{bb18}
			\nonumber&\int_{\epsilon}^t\|\u_t(s)\|_{\H}^2\d s+\mu\|\nabla \u(t)\|_{\H}^2+\alpha\|\u(t)\|_{\H}^2+\frac{2\beta}{r+1}\|\u(t)\|_{\widetilde{\L}^{r+1}}^{r+1}\nonumber\\&\leq \mu \|\nabla\u(\epsilon)\|_{\H}^2+\alpha\|\u(\epsilon)\|_{\H}^2+\frac{2\beta}{r+1}\|\u(\epsilon)\|_{\widetilde{\L}^{r+1}}^{r+1}+2(t-\epsilon)\|g\|_0^2\|\f\|_{\H}^2\nonumber\\&\quad+\beta  \int_{\epsilon}^t \big\||\u(s)|^\frac{r-1}{2}|\nabla\u(s)|\big\|_{\H}^2 \d  s +\gamma^*\int_{\epsilon}^t\|\nabla \u(s)\|_{\H}^2 \d s,
		\end{align}
		for all $t\in[\epsilon,T]$. We use the energy estimates \eqref{bb1} and \eqref{bb5} in \eqref{bb18} to obtain
		\begin{align*}
			\mu\|\nabla \u(t)\|_{\H}^2	\nonumber&\leq \mu \|\nabla\u(\epsilon)\|_{\H}^2+\alpha\|\u(\epsilon)\|_{\H}^2+\frac{2\beta}{r+1}\|\u(\epsilon)\|_{\widetilde{\L}^{r+1}}^{r+1}+2(t-\epsilon)\|g\|_0^2\|\f\|_{\H}^2 \\& \quad+\bigg(\frac{1}{2\mu t}+\frac{\gamma^*+\eta}{2\mu}\bigg)
			\|\u_0\|_{\H}^2+\bigg(\frac{1}{2\mu \alpha}+\frac{t}{\mu}+\frac{(\gamma^*+\eta) t}{2\mu \alpha}\bigg)\|g\|_0^2\|\f\|_{\H}^2.
		\end{align*}
		Integrating the above estimate over $\epsilon$ from $0$ to $t$, we get
	\begin{align*}
		\mu\|\nabla \u(t)\|_{\H}^2
		&	\leq \frac{1}{ t}  \bigg\{\mu \int_{0}^{t}\|\nabla\u(\epsilon)\|_{\H}^2 \d \epsilon+\alpha\int_{0}^{t}\|\u(\epsilon)\|_{\H}^2 \d \epsilon+\frac{2\beta}{r+1}\int_0^t\|\u(\epsilon)\|_{\widetilde{\L}^{r+1}}^{r+1}\d \epsilon\\&\quad
		+t^2\|g\|_0^2\|\f\|_{\H}^2+t\bigg(\frac{1}{2\mu t}+\frac{\gamma^*+\eta}{2\mu}\bigg)
		\|\u_0\|_{\H}^2\\&\quad+t\bigg(\frac{1}{2\mu \alpha}+\frac{t}{\mu}+\frac{(\gamma^*+\eta) t}{2\mu \alpha}\bigg)\|g\|_0^2\|\f\|_{\H}^2\bigg\}.
	\end{align*}
From \eqref{bb18}, one can reach \eqref{bb12} by using the energy estimate \eqref{bb1}.
\vskip 0.2 cm
\noindent\textbf{Case III:} 	\emph{$d=r=3$ with $\beta \mu >1$.} Using H\"older's and Young's inequalities, we have
	\begin{align}\label{bb19}
	|	\langle(\u \cdot \nabla) \u, \u_t \rangle|
	\leq  \frac{1}{4}\|\u_t\|_{\H}^2+\big\||\u||\nabla\u| \big\|_{\H}^2.
\end{align}
Substituting the estimates \eqref{bb16} and \eqref{bb19} in \eqref{bb14} and integrating it from $\epsilon$ to $t$, we deduce
\begin{align}\label{bb20}
	\nonumber&\int_{\epsilon}^t\|\u_t(s)\|_{\H}^2\d s+\mu\|\nabla \u(t)\|_{\H}^2+\alpha\|\u(t)\|_{\H}^2+\frac{\beta}{2}\|\u(t)\|_{\widetilde{\L}^{4}}^{4}\nonumber\\&\leq \mu \|\nabla\u(\epsilon)\|_{\H}^2+\alpha\|\u(\epsilon)\|_{\H}^2+\frac{\beta}{2}\|\u(\epsilon)\|_{\widetilde{\L}^{4}}^{4}+2(t-\epsilon)\|g\|_0^2\|\f\|_{\H}^2\nonumber\\&\quad+2  \int_{\epsilon}^t \big\||\u(s)||\nabla\u(s)|\big\|_{\H}^2 \d  s,
\end{align}
for all $t\in[\epsilon,T]$. Using  \eqref{bb6} in \eqref{bb20} results in
\begin{align*}
	\mu\|\nabla \u(t)\|_{\H}^2	\nonumber&\leq \mu \|\nabla\u(\epsilon)\|_{\H}^2+\alpha\|\u(\epsilon)\|_{\H}^2+\frac{\beta}{2}\|\u(\epsilon)\|_{\widetilde{\L}^{4}}^{4}+2(t-\epsilon)\|g\|_0^2\|\f\|_{\H}^2 \\& \quad+\frac{1}{2(\beta\mu-1)}
	\bigg\{ \frac{1}{t}\|\u_0\|_{\H}^2+\bigg(\frac{1}{\alpha}+2t\bigg)\|g\|_0^2\|\f\|_{\H}^2\bigg\}.
\end{align*}
	Integrating the above estimate over $\epsilon$ from $0$ to $t$, we get
\begin{align*}
	\mu\|\nabla \u(t)\|_{\H}^2
	&	\leq \frac{1}{ t}  \bigg\{\mu \int_{0}^{t}\|\nabla\u(\epsilon)\|_{\H}^2 \d \epsilon+\alpha \int_{0}^{t}\|\u(\epsilon)\|_{\H}^2 \d \epsilon+\frac{\beta}{2}\int_0^t\|\u(\epsilon)\|_{\widetilde{\L}^{4}}^{4}\d \epsilon\\&\quad
	+t^2\|g\|_0^2\|\f\|_{\H}^2+\frac{t}{2(\beta\mu-1)}
	\bigg\{ \frac{1}{t}\|\u_0\|_{\H}^2+\bigg(\frac{1}{\alpha}+2t\bigg)\|g\|_0^2\|\f\|_{\H}^2\bigg\},
\end{align*}
which easily leads to \eqref{bb13} by \eqref{bb1}.
	\end{proof}
	\begin{lemma}\label{lemb4}
		Let $(\u(\cdot),\nabla p(\cdot))$ be the unique solution of the CBF  equations \eqref{1a}-\eqref{2a} and $\u_0 \in \H$. Then, for all $t\in[\epsilon_1,T]$ and for any $0<\epsilon<\epsilon_1<T$,
		\begin{itemize}
			\item [$(i)$] 	for $d=2$ and $r \in[1,3],$   we have
		\begin{align}\label{bb21}
			&\sup_{s\in[\epsilon_1,t]}\|\u_t(s)\|_{\H}^2+\mu\int_{\epsilon_1}^t\|\nabla\u_t(t)\|_{\H}^2\d t +\beta \int_{\epsilon_1}^t\big\||\u(t)|^\frac{r-1}{2}|\u_t(t)|\big\|_{\H}^2 \d t
			\no\\&	\leq C\bigg(1+\frac{1}{(t-\epsilon)^3}\bigg) \big(1+\|\u_0\|_{\H}^6\big)+ C\bigg(1+\frac{1}{(t-\epsilon)^3}\bigg) \bigg(1+\|\f\|_{\H}^6\big(\|g\|_0^6+\|g_t\|_0^6\big)\bigg),
\end{align}
		\item [$(ii)$] 	for $d=2,3$ and $r >3,$   we have
		\begin{align}\label{bb22}
			&\sup_{s\in[\epsilon_1,t]}\|\u_t(s)\|_{\H}^2+\mu\int_{\epsilon_1}^t\|\nabla\u_t(s)\|_{\H}^2\d s +\beta \int_{\epsilon_1}^t\big\||\u(s)|^\frac{r-1}{2}|\u_t(s)|\big\|_{\H}^2 \d s
			\no\\&	\leq C\bigg(1+\frac{1}{(t-\epsilon)^2}\bigg) \|\u_0\|_{\H}^2+C\bigg(1+\frac{1}{t-\epsilon}\bigg)\|\f\|_{\H}^2\big(\|g\|_0^2+\|g_t\|_0^2\big),
		\end{align}
	\item [$(iii)$] 	for $d=r=3$ with $\beta \mu >1,$  we have
	\begin{align}\label{bb221}
		&\sup_{s\in[\epsilon_1,t]}\|\u_t(s)\|_{\H}^2+\mu\int_{\epsilon_1}^t\|\nabla\u_t(s)\|_{\H}^2\d s +\beta \int_{\epsilon_1}^t\big\||\u(s)|^\frac{r-1}{2}|\u_t(s)|\big\|_{\H}^2 \d s
		\no\\&	\leq \frac{C}{(t-\epsilon)^2} \|\u_0\|_{\H}^2+C\bigg(1+\frac{1}{t-\epsilon}\bigg)\|\f\|_{\H}^2\big(\|g\|_0^2+\|g_t\|_0^2\big).
	\end{align}
	\end{itemize}
	\end{lemma}
	\begin{proof}
		Applying $\partial / \partial t$ to the equation \eqref{1a}, we find 
		\begin{align}\label{bb23}
			\u_{tt}-\mu \Delta \u_t+(\u \cdot \nabla)\u_t+(\u_t \cdot \nabla)\u+\alpha \u_t+\beta  \mathcal{C}'(\u)\u_t+\nabla p_t=\f g_t,
		\end{align}
		where $\mathcal{C}'(\cdot)$ is defined in \eqref{29}.	Multiplying both sides of \eqref{bb23} by  $\u_t(\cdot)$ and then integrating over $\T$, we deduce 
		\begin{align}\label{bb24}
			&\frac{1}{2}\frac{\d}{\d t}\|\u_t(t)\|_{\H}^2+\mu\|\nabla\u_t(t)\|_{\H}^2+\alpha \|\u_t(t)\|_{\H}^2+\beta( \mathcal{C}'(\u)\u_t(t),\u_t(t))\nonumber\\&=(\f g_t(t),\u_t(t))-\big((\u_t(t) \cdot \nabla)\u(t),\u_t(t)\big),
		\end{align}
		for a.e. $t\in[\epsilon_1,T]$, for some $0<\epsilon <\epsilon_1 \leq T$.
		\vskip 0.2 cm
		\noindent\textbf{Case I:} 	\emph{$d=2$ and $r \in [1,3]$.}
			Plugging the relations \eqref{n50}-\eqref{n52} in \eqref{bb24}, and  integrating the resulting relation from $\epsilon_1$ to $t$ leads to
		\begin{align}\label{bb28}
			&\|\u_t(t)\|_{\H}^2+\mu\int_{\epsilon_1}^t\|\nabla\u_t(s)\|_{\H}^2\d s+2\beta \int_{\epsilon_1}^t\big\||\u(s)|^\frac{r-1}{2}|\u_t(s)|\big\|_{\H}^2 \d s\no\\&\quad+2\beta(r-1)\int_{\epsilon_1}^t\bigg\|\frac{1}{|\u|^\frac{3-r}{2}}\big(\u(s) \cdot \u_t(s)\big)\bigg\|_{\H}^2\d s\no\\&
			\leq \|\u_t(\epsilon_1)\|_{\H}^2+ \frac{t-\epsilon_1}{ \alpha}\|g_t\|_{0}^2\|\f\|_{\H}^2+\frac{2}{\mu}\sup_{s\in[\epsilon_1,t]}\|\nabla\u(s)\|_{\H}^2\int_{\epsilon}^t\|\u_t(s)\|_{\H}^2\d s,
		\end{align} 
		for all $t\in[\epsilon_1,T]$. Using \eqref{bb4} and \eqref{bb11} in \eqref{bb28}, we get 
		\begin{align*}
			\|\u_t(t)\|_{\H}^2 &\leq  \|\u_t(\epsilon_1)\|_{\H}^2+ \frac{t-\epsilon_1}{ \alpha}\|g_t\|_{0}^2\|\f\|_{\H}^2+\frac{1}{\mu^2}
			\bigg\{ \frac{1}{t}\|\u_0\|_{\H}^2+\bigg(\frac{1}{\alpha}+t\bigg)\|g\|_0^2\|\f\|_{\H}^2\bigg\}\\&\qquad\times\bigg[ \frac{N_{11}}{t}\|\u_0\|_{\H}^2+(N_{12}+t)\|g\|_0^2\|\f\|_{\H}^2 \no \\&\quad+\frac{C}{2 \mu^3}\bigg\{\frac{1}{t^2}\|\u_0\|_{\H}^4+2\bigg(\frac{1}{\alpha^2}+t^2\bigg)	\|g\|_0^4\|\f\|_{\H}^4\bigg\}\bigg],
		\end{align*}
	 where $N_{2i}, \ i=1,2,$  are defined in Lemma \ref{lemb3}. Integrating the above estimate over $\epsilon_1$ from $\epsilon$ to $t$, we arrive at
		\begin{align*}
			\|\u_t(t)\|_{\H}^2 &\leq \frac{1}{t-\epsilon}\bigg[\ \int_\epsilon^t\|\u_t(\epsilon_1)\|_{\H}^2\d \epsilon_1+ \|g_t\|_{0}^2\|\f\|_{\H}^2\int_\epsilon^t\frac{t-\epsilon_1}{ \alpha}\d \epsilon_1\\&\quad 
		+\frac{ (t-\epsilon)}{\mu^2}
		\bigg\{ \frac{1}{t}\|\u_0\|_{\H}^2+\bigg(\frac{1}{\alpha}+t\bigg)\|g\|_0^2\|\f\|_{\H}^2\bigg\}\times\bigg[ \frac{N_{11}}{t}\|\u_0\|_{\H}^2\\&\quad+(N_{12}+t)\|g\|_0^2\|\f\|_{\H}^2 +\frac{C}{2 \mu^3}\bigg\{\frac{1}{t^2}\|\u_0\|_{\H}^4+2\bigg(\frac{1}{\alpha^2}+t^2\bigg)	\|g\|_0^4\|\f\|_{\H}^4\bigg\}\bigg]\bigg].
		\end{align*}
		Thus, by using \eqref{bb11}, from \eqref{bb28}, we have 
		\begin{align*}
			&\sup_{t\in[\epsilon_1,T]}\|\u_t(t)\|_{\H}^2+\mu\int_{\epsilon_1}^t\|\nabla\u_t(t)\|_{\H}^2\d t+2\beta \int_{\epsilon_1}^t\left\||\u(t)|^\frac{r-1}{2}|\u_t(t)|\right\|_{\H}^2 \d t\\& \leq
			\frac{1}{t-\epsilon}\bigg[	 \frac{N_{11}}{t}\|\u_0\|_{\H}^2+(N_{12}+t)\|g\|_0^2\|\f\|_{\H}^2 +\frac{C}{2 \mu^3}\bigg\{\frac{1}{t^2}\|\u_0\|_{\H}^4+2\bigg(\frac{1}{\alpha^2}+t^2\bigg)	\|g\|_0^4\|\f\|_{\H}^4\bigg\}\bigg]\\&\quad+\frac{t-\epsilon}{2\alpha}\|g_t\|_0^2\|\f\|_{\H}^2+\frac{ 1}{\mu^2}
			\bigg\{ \frac{1}{t}\|\u_0\|_{\H}^2+\bigg(\frac{1}{\alpha}+t\bigg)\|g\|_0^2\|\f\|_{\H}^2\bigg\}\times\bigg[ \frac{N_{11}}{t}\|\u_0\|_{\H}^2\\&\quad+(N_{12}+t)\|g\|_0^2\|\f\|_{\H}^2 +\frac{C}{2 \mu^3}\bigg\{\frac{1}{t^2}\|\u_0\|_{\H}^4+2\bigg(\frac{1}{\alpha^2}+t^2\bigg)	\|g\|_0^4\|\f\|_{\H}^4\bigg\}\bigg],
		\end{align*}
	which leads to \eqref{bb21}.
		\vskip 0.2 cm
		\noindent\textbf{Case II:} 	\emph{$d=2,3$ and $r > 3$.} 	Plugging the relations \eqref{n50}, \eqref{n54} and \eqref{n55} in \eqref{bb24}, and integrating it from $\epsilon_1$ to $t$, we deduce 
		\begin{align}\label{bb31}
			&\|\u_t(t)\|_{\H}^2+\mu\int_{\epsilon_1}^t\|\nabla\u_t(s)\|_{\H}^2\d s+\beta \int_{\epsilon_1}^t\big\||\u(s)|^\frac{r-1}{2}|\u_t(s)|\big\|_{\H}^2 \d s \no\\&\quad+2 \beta(r-1)\int_{\epsilon_1}^t\big\||\u(s)|^\frac{r-3}{2}(\u(s) \cdot \u_t(s))\big\|_{\H}^2\d s\nonumber\\&
			\leq \|\u_t(\epsilon_1)\|_{\H}^2+ \frac{t-\epsilon_1}{ \alpha}\|g_t\|_{0}^2\|\f\|_{\H}^2+(\alpha+\eta^*)\int_{\epsilon_1}^t\|\u_t(s)\|_{\H}^2\d s,
		\end{align} 
		for all $t\in[\epsilon_1,T]$, 	where $\eta^*=\frac{r-3}{ \mu (r-1)}\left(\frac{2}{\beta \mu (r-1)}\right)^\frac{2}{r-3}$. Using the estimate \eqref{bb12} in \eqref{bb31}, we find
		\begin{align*}
			\|\u_t(t)\|_{\H}^2 &\leq  \|\u_t(\epsilon_1)\|_{\H}^2+ \frac{t-\epsilon_1}{ \alpha}\|g_t\|_{0}^2\|\f\|_{\H}^2+(\alpha+\eta^*)\bigg(\frac{N_{21}}{t}+N_{22}\bigg)\|\u_0\|_{\H}^2\no \\& \quad+(\alpha+\eta^*)\big(N_{23}t+N_{24}\big)\|g\|_0^2\|\f\|_{\H}^2,
		\end{align*}
	 where $N_{2i}, \ i=1,\ldots, 4,$  are defined in Lemma \ref{lemb3}. Integrating the above estimate over $\epsilon_1$ from $\epsilon$ to $t$ and then using the estimate \eqref{bb12} in it, we obtain
	\begin{align*}
		\|\u_t(t)\|_{\H}^2 &\leq \frac{1}{t-\epsilon}\bigg[\ \int_\epsilon^t\|\u_t(\epsilon_1)\|_{\H}^2\d \epsilon_1+ \|g_t\|_{0}^2\|\f\|_{\H}^2\int_\epsilon^t\frac{t-\epsilon_1}{ \alpha}\d \epsilon_1\no \\& \quad+(\alpha+\eta^*)(t-\epsilon)\bigg\{\bigg(\frac{N_{21}}{t}+N_{22}\bigg)\|\u_0\|_{\H}^2+\big(N_{23}t+N_{24}\big)\|g\|_0^2\|\f\|_{\H}^2\bigg\}\bigg] 
		\\& \leq
		\bigg(\frac{N_{21}}{(t-\epsilon)^2}+
		\frac{N_{22}+(\alpha+\eta^*)N_{21}}{t-\epsilon}+(\alpha+\eta^*) N_{22}\bigg) \|\u_0\|_{\H}^2+\frac{t}{2 \alpha} \|g_t\|_{0}^2\|\f\|_{\H}^2\\&\quad+\bigg(\frac{N_{23}t}{t-\epsilon}+\frac{N_{24}}{t-\epsilon}+(\alpha+\eta^*) N_{23} t+(\alpha+\eta^*) N_{24}\bigg)\|g\|_0^2\|\f\|_{\H}^2,
	\end{align*}
which give rise to \eqref{bb22}.
		\vskip 0.2 cm
		\noindent\textbf{Case III:} \emph{$d=r=3$ with $ \beta \mu > 1$.}
		Plugging the relations \eqref{n50}, \eqref{n58} and \eqref{n59} in \eqref{bb24} and then integrating it from $\epsilon_1$ to $t$, we obtain the following estimate:
		\begin{align}\label{bb34}
			\no	&\|\u_t(t)\|_{\H}^2+\mu\int_{\epsilon_1}^t\|\nabla\u_t(s)\|_{\H}^2\d s +2 \beta\int_{\epsilon_1}^t\left\|(\u(s) \cdot\u_t(s))\right\|_{\H}^2 \d s\nonumber\\&\quad+2\bigg(\beta -\frac{1}{2\mu}\bigg)\int_{\epsilon_1}^t\big\||\u(s)||\u_t(s)|\big\|_{\H}^2 \d s
	+	\leq \|\u_t(\epsilon_1)\|_{\H}^2+ \frac{t-\epsilon_1}{ \alpha}\|g_t\|_{0}^2\|\f\|_{\H}^2,
		\end{align} 
		for all $t\in[\epsilon_1,T]$. From \eqref{bb34}, we get
		\begin{align*}
			\|\u_t(t)\|_{\H}^2 &\leq  \|\u_t(\epsilon_1)\|_{\H}^2+ \frac{t-\epsilon_1}{ \alpha}\|g_t\|_{0}^2\|\f\|_{\H}^2.
		\end{align*}
	 Integrating the above estimate over $\epsilon_1$ from $\epsilon$ to $t$ and then using the estimate \eqref{bb13} in it, we obtain
	\begin{align*}
		\|\u_t(t)\|_{\H}^2 &\leq \frac{1}{t-\epsilon} \int_\epsilon^t\|\u_t(\epsilon_1)\|_{\H}^2\d \epsilon_1+\frac{t-\epsilon}{2 \alpha
		} \|g_t\|_{0}^2\|\f\|_{\H}^2
		\no \\& \leq \frac{1}{t-\epsilon}\bigg\{ \frac{N_{31}}{t}\|\u_0\|_{\H}^2+\bigg(N_{32}t+N_{33}\bigg)\|g\|_0^2\|\f\|_{\H}^2\bigg\}+\frac{t-\epsilon}{2 \alpha} \|g_t\|_{0}^2\|\f\|_{\H}^2,
	\end{align*} 
		where $N_{3i}, \ i=1,2,3,$  are defined in Lemma \ref{lemmae3} and  \eqref{bb221} follows, provided $\beta\mu > 1$, which completes the proof.
	\end{proof}
	The next lemma provides the regularity of the solution of the CBF  equations \eqref{1a}-\eqref{2a}. 
	\begin{lemma}\label{lemb5}
		Let $\u_0\in \H$ and  $\boldsymbol{F}(x,t):=\f(x)g(x,t)\in\mathrm{W}^{1,2}(0,T;\H)$ (which implies $\boldsymbol{F}\in\C([0,T];\H)$ also).  Then, for all $t\in[\epsilon_1,T]$ and for any $0<\epsilon<\epsilon_1<T$,
		\begin{itemize}
			\item [$(i)$] 	for $d=2,3$ and $r >3,$ and for $d=r=3$ with $\beta \mu >1$,   we have
		\begin{align}\label{bb35}
			&\sup_{t\in[\epsilon_1,T]}\left(\|\Delta\u(t)\|_{\H}^2+\|\nabla p(t)\|_{\L^2}^2\right)\no\\&\leq  C\bigg(1+\frac{1}{(t-\epsilon)^2}\bigg) \|\u_0\|_{\H}^2+C\bigg(1+\frac{1}{t-\epsilon}\bigg)\|\f\|_{\H}^2\big(\|g\|_0^2+\|g_t\|_0^2\big),
		\end{align}
			\item [$(ii)$] 	for $d=2$ and $r \in[1,3],$   we have
		\begin{align}\label{bb211}
			&\sup_{t\in[\epsilon_1,T]}\left(\|\Delta\u(t)\|_{\H}^2+\|\nabla p(t)\|_{\L^2}^2\right)
			\no\\&\leq C\bigg(1+\frac{1}{(t-\epsilon)^3}\bigg) \big(1+\|\u_0\|_{\H}^6\big)+ C\bigg(1+\frac{1}{(t-\epsilon)^3}\bigg) \bigg(1+\|\f\|_{\H}^6\big(\|g\|_0^6+\|g_t\|_0^6\big)\bigg).
		\end{align}
		\end{itemize}
	\end{lemma}
Using the maximal elliptic regularity to the elliptic boundary value problem obtained from \eqref{1a}, we get the estimate \eqref{bb35} (see Theorem 4.2, \cite{KT2}).	The estimate for the case of $d=2, \ r\in[1,3]$ can be obtained by using the $m$-accretive quantization of the linear and nonlinear operators (cf. Theorems 1.6 and 1.8 in Chapter 4, \cite{VB} for the abstract theory and Section 5, \cite{VBSS} for 2D NSE).
	\begin{lemma}\label{lemb6}
		Let $(\u(\cdot),\nabla p(\cdot))$ be the unique solution of the CBF  equations \eqref{1a}-\eqref{2a} and $\u_0 \in \H$.  Then,  for all $t\in[\epsilon_2,T]$ and for any $0<\epsilon<\epsilon_1<\epsilon_2 < T$,
			\begin{itemize}
			\item [$(i)$] 	for $d=2$ and $r >3,$   we have
			\begin{align}\label{bb36}
				&\sup_{t\in[\epsilon_2,T]}\|\nabla\u_t(t)\|_{\H}^2 +\int_{\epsilon_2}^T\|\Delta\u_t(t)\|_{\H}^2\d t\no\\&\leq  C\bigg\{\bigg(1+\frac{1}{(t-\epsilon_{1})^2}\bigg) \|\u_0\|_{\H}^2+\bigg(1+\frac{1}{(t-\epsilon_{1})}\bigg)\|\f\|_{\H}^2\big(\|g\|_0^2+\|g_t\|_0^2\big)\bigg\}^{\frac{r+1}{2}},
			\end{align}
			\item [$(ii)$] 	for $d=2$ and $r \in[1,3],$   we have
			\begin{align}\label{bb2111}
					&\sup_{t\in[\epsilon_2,T]}\|\nabla\u_t(t)\|_{\H}^2 +\int_{\epsilon_2}^T\|\Delta\u_t(t)\|_{\H}^2\d t
				\no\\&\leq C\bigg\{\bigg(1+\frac{1}{(t-\epsilon_1)^3}\bigg) \big(1+\|\u_0\|_{\H}^6\big)+ \bigg(1+\frac{1}{(t-\epsilon_1)^3}\bigg) \bigg(1+\|\f\|_{\H}^6\big(\|g\|_0^6+\|g_t\|_0^6\big)\bigg)\bigg\}^{\frac{r+1}{2}},
			\end{align}
			\item [$(iii)$] 	for $d=3$ and $r \geq3,$   we have
		\begin{align}\label{bb2112}
			&\sup_{t\in[\epsilon_2,T]}\|\nabla\u_t(t)\|_{\H}^2 +\int_{\epsilon_2}^T\|\Delta\u_t(t)\|_{\H}^2\d t\no\\&\leq  C\bigg\{\bigg(1+\frac{1}{(t-\epsilon_{1})^2}\bigg) \|\u_0\|_{\H}^2+\bigg(1+\frac{1}{(t-\epsilon_{1})}\bigg)\|\f\|_{\H}^2\big(\|g\|_0^2+\|g_t\|_0^2\big)\bigg\}^{\frac{3r+1}{4}}.
		\end{align}
		\end{itemize}
	\end{lemma}
	\begin{proof}
		Taking the inner product with $-\Delta\u_t(\cdot)$ in the equation \eqref{bb23}, we find 
		\begin{align}\label{bb38}
			\nonumber&\frac{1}{2}\frac{\d}{\d t}\|\nabla\u_t(t)\|_{\H}^2+ \mu\|\Delta\u_t(t)\|_{\H}^2+\alpha \|\nabla\u_t(t)\|_{\H}^2\\&=(\f g_t(t),-\Delta\u_t(t))-((\u_t(t) \cdot\nabla)\u(t),-\Delta\u_t(t))-((\u(t) \cdot \nabla)\u_t(t),-\Delta\u_t(t))\nonumber\\&\quad-\beta \big(\mathcal{C}'(\u)\u_t(t),-\Delta\u_t(t)\big)-(\nabla p_t,-\Delta \u_t(t))=:\sum_{i=1}^{5} I_i,
		\end{align}
		for a.e. $t\in[\epsilon_2, T]$, for some $0<\epsilon_2\leq T$.	Using H\"older's and Young's inequalities, we estimate $I_1$ as 
		\begin{align}\label{bb39}
			I_1 \leq	|(\f g_t,-\Delta\u_t)|\leq\|g_t\|_{\mathrm{L}^{\infty}}\|\f\|_{\H}\|\Delta\u_t\|_{\H}\leq\frac{7}{2\mu}\|g_t\|_{\mathrm{L}^{\infty}}^2\|\f\|_{\H}^2+\frac{\mu }{14}\|\Delta\u_t\|_{\H}^2.
		\end{align}
	 \vskip 0.2 cm
	\noindent\textbf{Case I:} 	\emph{$d=2$ and $r > 3$.}
		Using H\"older's, Gagliardo-Nirenberg's and Young's inequalities, we estimates $I_2$ as 
		\begin{align}\label{bb40}
			I_2 &\leq	|((\u_t \cdot\nabla)\u,-\Delta\u_t)| \leq  \|\u_t\|_{\widetilde{\L}^4} \|\nabla\u\|_{\widetilde{\L}^4} \|\Delta\u_t\|_{\H} \nonumber\\&\leq
			C\|\u_t\|_{\H}^\frac{1}{2}\|\nabla\u_t\|_{\H}^\frac{1}{2}\|\u\|_{\H}^\frac{1}{4}\|\u\|_{\H_p^2}^\frac{3}{4}\|\Delta\u_t\|_{\H}
			 \nonumber\\&\leq
			\frac{\mu}{14}\|\Delta\u_t\|_{\H}^2+C\|\u_t\|_{\H}\|\nabla\u_t\|_{\H}\|\u\|_{\H}^\frac{1}{2}\left(\|\u\|_{\H}^\frac{3}{2}+\|\Delta\u\|_{\H}^\frac{3}{2}\right)
			\no \\& \leq \frac{\mu}{14}\|\Delta\u_t\|_{\H}^2+\frac{\alpha}{4}\|\nabla\u_t\|_{\H}^2+C\left(\|\u\|_{\H}^4+\|\u\|_{\H}\|\Delta\u\|_{\H}^3\right)\|\u_t\|_{\H}^2.
		\end{align}
	 Using H\"older's, Agmon's and Young's inequalities, we have
	 	\begin{align}\label{bb41}
	 	I_3 \leq	|((\u \cdot\nabla)\u_t,-\Delta\u_t)|&\leq\|\u\|_{\widetilde{\L}^\infty} \|\nabla\u_t\|_{\H}\|\Delta\u_t\|_{\H}
	 	\no\\&\leq\|\u\|_{\H}^\frac{1}{2}\|\u\|_{\H_p^2}^\frac{1}{2} \|\nabla\u_t\|_{\H}\|\Delta\u_t\|_{\H}
	 	\no\\&\leq\|\u\|_{\H}^\frac{1}{2}\left(\|\u\|_{\H}^\frac{1}{2}+\|\Delta\u\|_{\H}^\frac{1}{2}\right) \|\nabla\u_t\|_{\H}\|\Delta\u_t\|_{\H}
	 	\no\\ &\leq
	 	C\left(\|\u\|_{\H}^2+\|\u\|_{\H}\|\Delta\u\|_{\H}\right)\|\nabla\u_t\|_{\H}^2+\frac{\mu}{14}\|\Delta\u_t\|_{\H}^2,
	 \end{align}
		\begin{align}\label{bb42}
			I_4&\leq	\beta\big|\big(\mathcal{C}'(\u)\u_t,-\Delta\u_t\big)\big| \no \\&= \left\{\begin{array}{cc}\left\{\begin{array}{cc}\beta\left|\left(|\u|^{r-1}\u_t+(r-1)\frac{\u}{|\u|^{3-r}}(\u\cdot\u_t),-\Delta\u_t\right)\right|,&\text{ for }\u\neq\mathbf{0},\\
					\mathbf{0},&\text{ for }\u=\mathbf{0},\end{array} \right.  \text{ for } 1\leq r<3,\\ \beta\big|\big(|\u|^{r-1}\u_t+(r-1)\u|\u|^{r-3}(\u\cdot\u_t),-\Delta\u_t\big) \big|, \ \ \ \text{ for }r\geq 3\end{array}\right.\no \\ &\leq C\|\u\|_{\widetilde{\L}^\infty}^{r-1}\|\u_t\|_{\H}\|\Delta\u_t\|_{\H} \leq C\|\u\|_{\H}^{\frac{r-1}{2}}\left(\|\u\|_{\H}^{\frac{r-1}{2}}+\|\Delta\u\|_{\H}^{\frac{r-1}{2}}\right)\|\u_t\|_{\H}\|\Delta\u_t\|_{\H}	\nonumber\\&\leq C \left(\|\u\|_{\H}^{2(r-1)}+\|\u\|_{\H}^{r-1}\|\Delta\u\|_{\H}^{r-1}\right)\|\u_t\|_{\H}^2+\frac{\mu}{14}\|\Delta\u_t\|_{\H}^2.
		\end{align}
		Taking divergence on both sides of \eqref{bb23}, we get
		\begin{align}\label{bb43}
			-\Delta p_t&=\nabla \cdot\big[\f g_t+ (\u \cdot \nabla)\u_t+ (\u_t \cdot \nabla)\u+\beta \mathcal{C}'(\u)\u_t\big]\nonumber\\
			p_t&=(-\Delta)^{-1}\big[\nabla \cdot \big\{(\u \cdot \nabla)\u_t+ (\u_t \cdot \nabla)\u+\beta \mathcal{C}'(\u)\u_t\big\}\big],
		\end{align}
		in the weak sense and we have used $\nabla \cdot \f=0$. Taking gradient on both sides in \eqref{bb43} and then using H\"older's and Young's inequalities, we estimate the term $I_{5}$ as
		\begin{align}\label{bb44}
			I_{5}&\leq	|(\nabla p_t,-\Delta \u_t)| \leq \|\nabla p_t\|_{\H}\|\Delta \u_t\|_{\H} \nonumber\\&\leq C\big( \|(\u \cdot \nabla)\u_t\|_{\H}+\|(\u_t \cdot \nabla)\u\|_{\H} +\beta \|\mathcal{C}'(\u)\u_t\|_{\H}\big)\|\Delta \u_t\|_{\H} \nonumber\\&\leq \frac{3\mu}{14}\|\Delta \u_t\|_{\H}^2+\frac{\alpha}{4}\|\nabla\u_t\|_{\H}^2+C\left(\|\u\|_{\H}^4+\|\u\|_{\H}\|\Delta\u\|_{\H}^3\right)\|\u_t\|_{\H}^2\no \\&\quad+C\left(\|\u\|_{\H}^2+\|\u\|_{\H}\|\Delta\u\|_{\H}\right)\|\nabla\u_t\|_{\H}^2 +C \left(\|\u\|_{\H}^{2(r-1)}+\|\u\|_{\H}^{r-1}\|\Delta\u\|_{\H}^{r-1}\right)\|\u_t\|_{\H}^2,
		\end{align}
		where we have used the estimates \eqref{bb40}-\eqref{bb42} in the final  inequality.
		Substituting the estimates \eqref{bb39}-\eqref{bb42} and \eqref{bb44} in \eqref{bb38}, and then integrating the resulting estimate from $\epsilon_2$ to $t$, we arrive at
		\begin{align}\label{bb45}
			&\|\nabla\u_t(t)\|_{\H}^2+\mu \int_{\epsilon_2}^t\|\Delta \u(s)\|_{\H}^2 \d s+2\alpha  \int_{\epsilon_2}^t \|\nabla\u_t(s)\|_{\H}^2\d s \nonumber\\&\leq 	\|\nabla\u_t(\epsilon_2)\|_{\H}^2+C(t-\epsilon_2)\|g_t\|_0^2\|\f\|_{\H}^2+C(t-\epsilon_2)\sup_{t\in[\epsilon_2,T]}\left(\|\u(t)\|_{\H}^4+\|\u(t)\|_{\H}\|\Delta\u(t)\|_{\H}^3\right)
			\no\\ & \quad \times \sup_{t\in[\epsilon_2,T]}\|\u_t(t)\|_{\H}^2 +C\sup_{t\in[\epsilon_2,T]}\left(\|\u(t)\|_{\H}^2+\|\u(t)\|_{\H}\|\Delta\u(t)\|_{\H}\right)\int_{\epsilon_2}^t\|\nabla\u_t(s)\|_{\H}^2 \d s
			\nonumber\\&\quad	 +C(t-\epsilon_2)\sup_{t\in[\epsilon_2,T]}\left(\|\u(t)\|_{\H}^{2(r-1)}+\|\u(t)\|_{\H}^{r-1}\|\Delta\u(t)\|_{\H}^{r-1}\right)\sup_{t\in[\epsilon_2,T]}\|\u_t(t)\|_{\H}^2,
		\end{align}
		for all $t \in [\epsilon_2,T]$.  Using the energy estimates obtained in Lemmas \ref{lemb1}, \ref{lemb4} and \ref{lemb5} in \eqref{bb45}, we obtain the following estimate:
		\begin{align*}
			\|\nabla\u_t(t)\|_{\H}^2 &\leq	\|\nabla\u_t(\epsilon_2)\|_{\H}^2\\&+C\bigg\{\bigg(1+\frac{1}{(t-\epsilon)^2}\bigg) \|\u_0\|_{\H}^2+\bigg(1+\frac{1}{t-\epsilon}\bigg)\|\f\|_{\H}^2\big(\|g\|_0^2+\|g_t\|_0^2\big)\bigg\}^{\frac{r+1}{2}},
		\end{align*}
		for all $t \in [\epsilon_2,T]$. Integrating the above estimate over $\epsilon_2$ from $\epsilon_1$ to $t$ and then using the energy estimate obtained in Lemma \ref{lemb4}, we deduce 
		\begin{align}\label{bb46}
			\|\nabla\u_t(t)\|_{\H}^2&\leq \frac{1}{ t-\epsilon_1}\int_{\epsilon_1}^t\|\nabla\u_t(\epsilon_2)\|_{\H}^2\d\epsilon_2\no\\&\quad+C\bigg\{\bigg(1+\frac{1}{(t-\epsilon)^2}\bigg) \|\u_0\|_{\H}^2+\bigg(1+\frac{1}{t-\epsilon}\bigg)\|\f\|_{\H}^2\big(\|g\|_0^2+\|g_t\|_0^2\big)\bigg\}^{\frac{r+1}{2}}
			\no\\&\leq C\bigg\{\bigg(1+\frac{1}{(t-\epsilon_{1})^2}\bigg) \|\u_0\|_{\H}^2+\bigg(1+\frac{1}{(t-\epsilon_{1})}\bigg)\|\f\|_{\H}^2\big(\|g\|_0^2+\|g_t\|_0^2\big)\bigg\}^{\frac{r+1}{2}},
		\end{align}
	 for any $0<\epsilon<\epsilon_1<\epsilon_2<T$.	Thus, from \eqref{bb45}, it is immediate that
		\begin{align}\label{bb47}
			&\sup_{t\in[\epsilon_2,T]}	\|\nabla\u_t(t)\|_{\H}^2+\mu\int_{\epsilon_2}^T\|\Delta\u_t(t)\|_{\H}^2 \d t\no\\&\leq C\bigg\{\bigg(1+\frac{1}{(t-\epsilon_{1})^2}\bigg) \|\u_0\|_{\H}^2+\bigg(1+\frac{1}{(t-\epsilon_{1})}\bigg)\|\f\|_{\H}^2\big(\|g\|_0^2+\|g_t\|_0^2\big)\bigg\}^{\frac{r+1}{2}},
		\end{align}
		for all $t \in [\epsilon_2,T]$ and $\Delta \u_t \in \mathrm{L}^2(\epsilon_2,T;\H)$.  One can infer that $\u_t\in\mathrm{L}^2(\epsilon_2,T;\H_p^2 \cap \V)$ (cf. subsection \ref{sub2.1}).
		\vskip 0.2 cm
			\noindent\textbf{Case II:} 	\emph{$d=2$ and $r \in [1,3]$.}
			Substituting the estimates \eqref{bb39}-\eqref{bb42} and \eqref{bb44} in \eqref{bb38}, and by using the similar arguments as $d=2$ and $r > 3$, we obtain the required estimate \eqref{bb2111} and $\u_t\in\mathrm{L}^2(\epsilon_2,T;\H_p^2 \cap \V)$.
			\vskip 0.2 cm
			\noindent\textbf{Case III:} 	\emph{$d=3$ and $r \geq 3$.} Using H\"older's, Gagliardo-Nirenberg's and Young's inequalities, we have
			\begin{align}\label{bb401}
				I_2 &\leq	|((\u_t \cdot\nabla)\u,-\Delta\u_t)| \leq  \|\u_t\|_{\widetilde{\L}^4} \|\nabla\u\|_{\widetilde{\L}^4} \|\Delta\u_t\|_{\H} \nonumber\\&\leq
				C\|\u_t\|_{\H}^\frac{1}{4}\|\nabla\u_t\|_{\H}^\frac{3}{4}\|\u\|_{\H}^\frac{1}{8}\|\u\|_{\H_p^2}^\frac{7}{8}\|\Delta\u_t\|_{\H}
				\no \\& \leq \frac{\mu}{14}\|\Delta\u_t\|_{\H}^2+C\|\u_t\|_{\H}^\frac{1}{2}\|\nabla\u_t\|_{\H}^\frac{3}{2}\|\u\|_{\H}^\frac{1}{4}\left(\|\u\|_{\H}^{\frac{7}{4}}+\|\Delta\u\|_{\H}^{\frac{7}{4}}\right)
					\no \\& \leq \frac{\mu}{14}\|\Delta\u_t\|_{\H}^2+\frac{\alpha}{4}\|\nabla\u_t\|_{\H}^{2}+C\left(\|\u\|_{\H}^8+\|\u\|_{\H}\|\Delta\u\|_{\H}^{7}\right)\|\u_t\|_{\H}^{2}.
			\end{align}
	Making the use of H\"older's, Agmon's and Young's inequalities, we compute
			\begin{align}\label{bb411}
				I_3 \leq	|((\u \cdot\nabla)\u_t,-\Delta\u_t)|&\leq\|\u\|_{\widetilde{\L}^\infty} \|\nabla\u_t\|_{\H}\|\Delta\u_t\|_{\H}
				\no\\&\leq C \|\u\|_{\H}^{\frac{1}{4}}\left(\|\u\|_{\H}^{\frac{3}{4}}+\|\Delta\u\|_{\H}^{\frac{3}{4}}\right) \|\nabla\u_t\|_{\H}\|\Delta\u_t\|_{\H}
				\no\\ &\leq
				C \left(\|\u\|_{\H}^2+\|\u\|_{\H}^{\frac{1}{2}}\|\Delta\u\|_{\H}^{\frac{3}{2}}\right)\|\nabla\u_t\|_{\H}^2+\frac{\mu}{14}\|\Delta\u_t\|_{\H}^2.
			\end{align}
		A calculation similar to \eqref{bb42} gives 
			\begin{align}\label{bb421}
				I_4&\leq	\beta\big|\big(\mathcal{C}'(\u)\u_t,-\Delta\u_t\big)\big| 	\leq C\|\u\|_{\widetilde{\L}^\infty}^{r-1}\|\u_t\|_{\H}\|\Delta\u_t\|_{\H} \leq C\|\u\|_{\H}^{\frac{r-1}{4}}\|\u\|_{\H_p^2}^{\frac{3(r-1)}{4}}\|\u_t\|_{\H}\|\Delta\u_t\|_{\H}
				\no \\&  \leq C\|\u\|_{\H}^{\frac{r-1}{4}}\left(\|\u\|_{\H}^{\frac{3(r-1)}{4}}+\|\Delta\u\|_{\H}^{\frac{3(r-1)}{4}}\right)\|\u_t\|_{\H}\|\Delta\u_t\|_{\H}
				\nonumber\\&\leq C \left(\|\u\|_{\H}^{2(r-1)}+\|\u\|_{\H}^{\frac{r-1}{2}}\|\Delta\u\|_{\H^2}^{\frac{3(r-1)}{2}}\right)\|\u_t\|_{\H}^2+\frac{\mu}{14}\|\Delta\u_t\|_{\H}^2,
			\end{align}
 An estimate similar to \eqref{bb44} yields
		\begin{align}\label{bb422}
			I_{5} &\leq \frac{3\mu}{14}\|\Delta \u_t\|_{\H}^2+\frac{\alpha}{4}\|\nabla\u_t\|_{\H}^{2}+C\|\u_t\|_{\H}^{2}\|\u\|_{\H}\|\u\|_{\H^2}^{7}+C\left(\|\u\|_{\H}^2+\|\u\|_{\H}^{\frac{1}{2}}\|\Delta\u\|_{\H}^{\frac{3}{2}}\right)\|\nabla\u_t\|_{\H}^2\no \\& \quad+C \left(\|\u\|_{\H}^{2(r-1)}+\|\u\|_{\H}^{\frac{r-1}{2}}\|\Delta\u\|_{\H^2}^{\frac{3(r-1)}{2}}\right)\|\u_t\|_{\H}^2,
		\end{align}
where we have used the estimates \eqref{bb401}-\eqref{bb421}.	Substituting \eqref{bb39}, \eqref{bb401}-\eqref{bb422} in \eqref{bb38}, and by using the similar arguments as $d=2$ and $r > 3$, we obtain \eqref{bb2112} and $\u_t\in\mathrm{L}^2(\epsilon_2,T;\H^2_p \cap \V)$.
	\end{proof}
	\begin{lemma}\label{lemb7}
		Let $(\u(\cdot),\nabla p(\cdot))$ be the unique solution of the CBF  equations \eqref{1a}-\eqref{2a}.  Then,  for all $t\in[\epsilon_2,T]$ and for any $0<\epsilon<\epsilon_1<\epsilon_2 < T$,
			\begin{itemize}
			\item [$(i)$] 	for $d=2$ and $r >3,$   we have
			\begin{align}\label{bb53}
				&	\sup_{t\in[\epsilon_2,T]}\|\nabla\u_t(t)\|_{\H}^2 +\int_{\epsilon_2}^T\|\u_{tt}(t)\|_{\H}^2\d t\no\\&\leq  C\bigg\{\bigg(1+\frac{1}{(t-\epsilon_1)^2}\bigg) \|\u_0\|_{\H}^2+\bigg(1+\frac{1}{(t-\epsilon_1)}\bigg)\|\f\|_{\H}^2\big(\|g\|_0^2+\|g_t\|_0^2\big)\bigg\}^{\frac{r+1}{2}},
			\end{align}
			\item [$(ii)$] 	for $d=2$ and $r \in[1,3],$   we have
			\begin{align}\label{bb21111}
				&	\sup_{t\in[\epsilon_2,T]}\|\nabla\u_t(t)\|_{\H}^2 +\int_{\epsilon_2}^T\|\u_{tt}(t)\|_{\H}^2\d t
				\no\\&\leq  C\bigg\{\bigg(1+\frac{1}{(t-\epsilon_1)^3}\bigg) \big(1+\|\u_0\|_{\H}^6\big)+ \bigg(1+\frac{1}{(t-\epsilon_1)^3}\bigg) \bigg(1+\|\f\|_{\H}^6\big(\|g\|_0^6+\|g_t\|_0^6\big)\bigg)\bigg\}^{\frac{r+1}{2}}.
			\end{align}
			\item [$(iii)$] 	for $d=3$ and $r \geq3,$    we have
		\begin{align}\label{bb21112}
			&\sup_{t\in[\epsilon_2,T]}\|\nabla\u_t(t)\|_{\H}^2 +\int_{\epsilon_2}^T\|\u_{tt}(t)\|_{\H}^2\d t\no\\&\leq  C\bigg\{\bigg(1+\frac{1}{(t-\epsilon_{1})^2}\bigg)  \|\u_0\|_{\H}^2+\bigg(1+\frac{1}{(t-\epsilon_{1})}\bigg)\|\f\|_{\H}^2\big(\|g\|_0^2+\|g_t\|_0^2\big)\bigg\}^{\frac{3r+1}{4}}.
		\end{align}
		\end{itemize}
	\end{lemma}
	\begin{proof}
		Taking the inner product with $\u_{tt}(\cdot)$ in the equation \eqref{bb23}, we deduce 
		\begin{align}\label{bb54}
			&\|\u_{tt}(t)\|_{\H}^2+\frac{\mu}{2}\frac{\d}{\d t}\|\nabla\u_t(t)\|_{\H}^2+\frac{\alpha}{2}\frac{\d}{\d t} \|\u_t(t)\|_{\H}^2\nonumber\\&=(\f g_t(t),\u_{tt}(t))-((\u_t(t) \cdot\nabla)\u(t),\u_{tt}(t))\nonumber\\&\quad-((\u(t) \cdot \nabla)\u_t(t),\u_{tt}(t))-\beta (\mathcal{C}'(\u)\u_t(t),\u_{tt}(t)),
		\end{align}
		for a.e. $t\in[\epsilon_2, T]$, for some $0<\epsilon_2\leq T$.	Using H\"older's and Young's inequalities, we estimate  $|(\f g_t,\u_{tt})|$ as 
		\begin{align}\label{bb541}
			|(\f g_t,\u_{tt})|\leq\|g_t\|_{\mathrm{L}^{\infty}}\|\f\|_{\H}\|\u_{tt}\|_{\H}\leq 2 \|g_t\|_{\mathrm{L}^{\infty}}^2\|\f\|_{\H}^2+\frac{1 }{8}\|\u_{tt}\|_{\H}^2.
		\end{align}
		\vskip 0.2 cm
	\noindent\textbf{Case I:} 	\emph{$d=2$ and $r > 3$.}
		Calculations similar to \eqref{bb40}-\eqref{bb42} yield 
		\begin{align}
			|((\u_t \cdot\nabla)\u,\u_{tt})| 
			&\leq 	\frac{1}{8}\|\u_{tt}\|_{\H}^2+\|\u_t\|_{\H}\|\nabla\u_t\|_{\H}\left(\|\u\|_{\H}^2+\|\u\|_{\H}^\frac{1}{2}\|\Delta\u\|_{\H}^\frac{3}{2}\right),\label{bb542}\\
			|((\u \cdot\nabla)\u_t,\u_{tt})|&\leq 	\frac{1}{8}\|\u_{tt}\|_{\H}^2+	C\left(\|\u\|_{\H}^2+\|\u\|_{\H}\|\Delta\u\|_{\H}\right)\|\nabla\u_t\|_{\H}^2,\label{bb543}\\
|(\mathcal{C}'(\u)\u_t,\u_{tt})|
		&\leq \frac{1}{8}\|\u_{tt}\|_{\H}^2+ C \left(\|\u\|_{\H}^{2(r-1)}+\|\u\|_{\H}^{r-1}\|\Delta\u\|_{\H}^{r-1}\right)\|\u_t\|_{\H}^2.\label{bb544}
	\end{align}
		Substituting the estimates \eqref{bb541}-\eqref{bb544} in \eqref{bb54} and then integrating it from $\epsilon_2$ to $T$, we arrive at
		\begin{align}\label{bb55}
			\nonumber&\int_{\epsilon_2}^T\|\u_{tt}(s)\|_{\H}^2\d s+\mu \|\nabla\u_t(t)\|_{\H}^2+\alpha \|\u_t(t)\|_{\H}^2 
			\nonumber\\&\leq \mu \|\nabla\u_t(\epsilon_2)\|_{\H}^2+\alpha \|\u_t(\epsilon_2)\|_{\H}^2+4(t-\epsilon_2)\|g_t\|_0^2\|\f\|_{\H}^2
			\no \\& \quad+ C(t-\epsilon_2)\sup_{t\in[\epsilon_2,T]} \big(\|\u_t(t)\|_{\H}\|\nabla\u_t(t)\|_{\H}\big) \sup_{t\in[\epsilon_2,T]}\left(\|\u(t)\|_{\H}^2+\|\u(t)\|_{\H}^\frac{1}{2}\|\Delta\u(t)\|_{\H}^\frac{3}{2}\right)
			 \no\\&\quad+C(t-\epsilon_2)\sup_{t\in[\epsilon_2,T]}\left(\|\u(t)\|_{\H}^2+\|\u(t)\|_{\H}\|\Delta\u(t)\|_{\H}\right)\sup_{t\in[\epsilon_2,T]}\|\nabla\u_t(t)\|_{\H}^2
			 \no \\& \quad+C(t-\epsilon_2)\left(\|\u(t)\|_{\H}^{2(r-1)}+\|\u(t)\|_{\H}^{r-1}\|\Delta\u(t)\|_{\H}^{r-1}\right)\sup_{t\in[\epsilon_2,T]}\|\u_t(t)\|_{\H}^2,
		\end{align}
		for all $t \in [\epsilon_2,T]$. Using the energy estimates obtained in Lemmas \ref{lemb1}, \ref{lemb4} and \ref{lemb5} in \eqref{bb55}, we obtain
		\begin{align*}
			\|\nabla\u_t(t)\|_{\H}^2  
			&\leq  \|\nabla\u_t(\epsilon_2)\|_{\H}^2+\frac{\alpha}{\mu} \|\u_t(\epsilon_2)\|_{\H}^2\\&\quad+C\bigg\{\bigg(1+\frac{1}{(t-\epsilon)^2}\bigg) \|\u_0\|_{\H}^2+\bigg(1+\frac{1}{t-\epsilon}\bigg)\|\f\|_{\H}^2\big(\|g\|_0^2+\|g_t\|_0^2\big)\bigg\}^{\frac{r+1}{2}},
		\end{align*}
		for all $t \in [\epsilon_2,T]$. Integrating the above estimate over $\epsilon_2$ from $\epsilon_1$ to $t$ and then using the energy estimate obtained in Lemma \ref{lemb4}, we deduce 
		\begin{align*}
			\|\nabla\u_t(t)\|_{\H}^2&\leq \frac{1}{  t-\epsilon_1}\bigg( \int_{\epsilon_1}^t\|\nabla\u_t(\epsilon_2)\|_{\H}^2\d\epsilon_2+\frac{\alpha}{\mu}\int_{\epsilon_1}^t\|\u_t(\epsilon_2)\|_{\H}^2\d\epsilon_2 \bigg)\\&\quad+C\bigg\{\bigg(1+\frac{1}{(t-\epsilon)^2}\bigg) \|\u_0\|_{\H}^2+\bigg(1+\frac{1}{t-\epsilon}\bigg)\|\f\|_{\H}^2\big(\|g\|_0^2+\|g_t\|_0^2\big)\bigg\}^{\frac{r+1}{2}}
			\no\\& \leq C\bigg\{\bigg(1+\frac{1}{(t-\epsilon_1)^2}\bigg) \|\u_0\|_{\H}^2+\bigg(1+\frac{1}{(t-\epsilon_1)}\bigg)\|\f\|_{\H}^2\big(\|g\|_0^2+\|g_t\|_0^2\big)\bigg\}^{\frac{r+1}{2}},
		\end{align*}
		for any $0<\epsilon<\epsilon_1<\epsilon_2 < T$.	Thus, from  \eqref{bb55}, one can reach \eqref{bb53}
		\iffalse
		\begin{align*}
			&\sup_{t\in[\epsilon_2,T]}\mu\|\nabla\u_t(t)\|_{\H}^2 +\int_{\epsilon_2}^T\|\u_{tt}(t)\|_{\H}^2\d t \no\\&\leq
			C\bigg\{\bigg(1+\frac{1}{(t-\epsilon_1)^2}\bigg) \|\u_0\|_{\H}^2+\bigg(1+\frac{1}{(t-\epsilon_1)}\bigg)\|\f\|_{\H}^2\big(\|g\|_0^2+\|g_t\|_0^2\big)\bigg\}^{\frac{r+1}{2}},
		\end{align*}
	\fi
		and $\u_{tt} \in \mathrm{L}^{2}(\epsilon_2,T;\H)$.
			\vskip 0.2 cm
		\noindent\textbf{Case II:} 	\emph{$d=2$ and $r \in [1,3]$.} Substituting the estimates \eqref{bb541}-\eqref{bb544} in \eqref{bb54} and then
	 using the similar arguments as $d=2$ and $r > 3$, we obtain the required estimate \eqref{bb21111} 	and $\u_{tt} \in \mathrm{L}^{2}(\epsilon_2,T;\H)$.
	 	\vskip 0.2 cm
	 \noindent\textbf{Case III:} 	\emph{$d=3$ and $r \geq3$.} 
	 Calculations similar to \eqref{bb401}-\eqref{bb421} give
	 \begin{align}
	 	|((\u_t \cdot\nabla)\u,\u_{tt})|  &\leq \frac{1}{8}\|\u_{tt}\|_{\H}^2+ C\|\u_t\|_{\H}^\frac{1}{2}\|\nabla\u_t\|_{\H}^\frac{3}{2}\left(\|\u\|_{\H}^{2}+\|\u\|_{\H}^\frac{1}{4}\|\Delta\u\|_{\H}^{\frac{7}{4}}\right),\label{bb56} \\
	 	|((\u \cdot\nabla)\u_t,\u_{tt})| & \leq \frac{1}{8}\|\u_{tt}\|_{\H}^2+ C\left(\|\u\|_{\H}^2+\|\u\|_{\H}^{\frac{1}{2}}\|\Delta\u\|_{\H}^{\frac{3}{2}}\right)\|\nabla\u_t\|_{\H}^2, \label{bb57}\\
	 	\big|\big(\mathcal{C}'(\u)\u_t,\u_{tt}\big)\big| 		&\leq \frac{1}{8}\|\u_{tt}\|_{\H}^2+C  \left(\|\u\|_{\H}^{2(r-1)}+\|\u\|_{\H}^{\frac{r-1}{2}}\|\Delta\u\|_{\H^2}^{\frac{3(r-1)}{2}}\right)\|\u_t\|_{\H}^2.\label{bb58}
	 \end{align}
 Substituting \eqref{bb541} and \eqref{bb56}-\eqref{bb58} in \eqref{bb54} and then
 using the similar arguments as $d=2$ and $r > 3$, we obtain the required estimate \eqref{bb21112} and $\u_{tt} \in \mathrm{L}^{2}(\epsilon_2,T;\H)$.
	\end{proof}
	\begin{remark}\label{remb7}
		From Lemmas \ref{lemb6} and \ref{lemb7}, the facts$$\u_t\in\mathrm{L}^2(\epsilon_2,T;\H^2_p\cap \V) \ \ \text{and}  \ \ \u_{tt}\in\mathrm{L}^2(\epsilon_2,T;\H),$$ imply that $\u_t\in\C([\epsilon_2,T];\V)$ for any $0<\epsilon<\epsilon_1<\epsilon_2\leq T$. 
	\end{remark}
\end{appendix}

\medskip\noindent
{\bf Acknowledgments:} P. Kumar and M. T. Mohan would  like to thank the Department of Science and Technology (DST), India for Innovation in Science Pursuit for Inspired Research (INSPIRE) Faculty Award (IFA17-MA110). The first author expresses his heartfelt thanks to Mr. Kush Kinra for useful discussions.

\medskip\noindent	{\bf  Declarations:} 

\noindent 	{\bf  Ethical Approval:}   Not applicable 

\noindent  {\bf   Competing interests: } The authors declare no competing interests. 

\noindent 	{\bf   Authors' contributions: } All authors have contributed equally. 

\noindent 	{\bf   Funding: } DST, India, IFA17-MA110 (M. T. Mohan).

\noindent 	{\bf   Availability of data and materials: } Not applicable.

\end{document}